\documentclass[11pt]{amsart}
\usepackage{amsmath, amssymb, amsfonts, amsthm, verbatim, color}
\usepackage{mathtools}
\mathtoolsset{showonlyrefs}
\usepackage{bbm}
\usepackage{float}
\usepackage{amsmath, amssymb, amsfonts, amsthm, verbatim, dsfont, graphicx}
\usepackage{epstopdf}
\usepackage[hidelinks]{hyperref}
\usepackage[T1]{fontenc}
\usepackage{amssymb,amsmath, amsthm, amsfonts}
\usepackage{graphicx}

\usepackage{thmtools,thm-restate}

\newtheorem{theorem}{Theorem}[section]
\newtheorem{proposition}[theorem]{Proposition}
\newtheorem{corollary}[theorem]{Corollary}
\newtheorem{lemma}[theorem]{Lemma}
\newtheorem{definition}[theorem]{Definition}
\newtheorem{remark}[theorem]{Remark}

\numberwithin{equation}{section}

\numberwithin{figure}{section}

\newcommand{\bR}{{\mathbb R}}
\newcommand{\bN}{{\mathbb N}}
\newcommand{\bC}{{\mathbb C}}
\newcommand{\bZ}{{\mathbb Z}}

\newcommand{\bP}{{\mathbb P}}

\newcommand{\cF}{{\mathcal{F}}}
\newcommand{\cN}{{\mathcal{N}}}

\newcommand{\cM}{{\mathcal{M}}}

\newcommand{\cH}{{\mathcal{H}}}

\begin{document}

\title[Focusing NLW with random initial data]{The focusing energy-critical nonlinear wave equation with random initial data}
\author[C. Kenig]{Carlos Kenig}
\address{Department of Mathematics \\ University of Chicago \\ 5734 S. University Ave \\ Chicago, IL 60637}
\email{cek@math.uchicago.edu}
\author[D. Mendelson]{Dana Mendelson}
\address{Department of Mathematics \\ University of Chicago \\ 5734 S. University Ave \\ Chicago, IL 60637}
\email{dana@math.uchicago.edu}
\date{}

\begin{abstract}
We consider the focusing energy-critical quintic nonlinear wave equation in three dimensional Euclidean space. It is known that this equation admits a one-parameter family of radial stationary solutions, called solitons, which can be viewed as a curve in $ \dot H^s_x(\bR^3) \times   H^{s-1}_x(\bR^3)$, for any $s > 1/2$. By randomizing radial initial data in $ \dot H^s_x(\bR^3) \times   H^{s-1}_x(\bR^3)$ for $s > 5/6$, which also satisfy a certain weighted Sobolev condition, we produce with high probability a family of radial perturbations of the soliton which give rise to global forward-in-time solutions of the focusing nonlinear wave equation that scatter after subtracting a dynamically modulated soliton. Our proof relies on a new randomization procedure using distorted Fourier projections associated to the linearized operator around a fixed soliton. To our knowledge, this is the first long-time random data existence result for a focusing wave or dispersive equation on Euclidean space outside the small data regime.
\end{abstract}

\thanks{C. Kenig gratefully acknowledges  support from NSF grants DMS-1265249, DMS-1463746 and DMS-1800082. D.~Mendelson gratefully acknowledges support from NSF grant DMS-1800697.} 

\maketitle

\section{Introduction}
We consider the Cauchy problem for the focusing quintic nonlinear wave equation with radial initial data in three space dimensions
\begin{equation} \label{equ:ivp}
 \left\{ \begin{aligned}
  -\partial_t^2 \psi + \Delta \psi &= - \psi^5 , \\
  (\psi, \partial_t \psi)|_{t=0} &= (\psi_0, \psi_1) \in  \dot H^s_x(\bR^3) \times   H^{s-1}_x(\bR^3) =:  \cH^s(\bR^3).
 \end{aligned} \right.
\end{equation}
where $ \dot H^s_x(\bR^3)$ and $H^{s-1}(\bR^3)$ are the usual homogeneous and inhomogeneous Sobolev spaces, respectively. For initial data in $\dot H^1(\bR^3) \times L^2(\bR^3)$, the conserved energy for (sufficiently regular) solutions of \eqref{equ:ivp} is given by
\[
E(\psi) = \frac{1}{2} \int |\nabla \psi|^2 + \frac{1}{2} \int |\partial_t \psi|^2 - \frac{1}{6} \int |\psi|^6.
\]
The energy is invariant under the scaling symmetry of the equation
\[
\psi \mapsto \psi_\lambda(x,t) = \sqrt{\lambda} \psi(\lambda x, \lambda t),
\]
and hence \eqref{equ:ivp} is the so-called energy critical nonlinear wave equation in three dimensions. 

Throughout this work, we will largely follow the notation of \cite{Beceanu, KS14}. The equation \eqref{equ:ivp} admits a one-parameter family of radial stationary solutions, called solitons, which solve the elliptic equation
\begin{align}
\Delta \phi_a + \phi_a^5 = 0, \qquad a >0,
\end{align}
which were identified in \cite{Talenti} as extremizers of the Sobolev embedding $\dot H^1(\bR^3) \hookrightarrow L^6(\bR^3)$, and given explicitly by the formula
\[
\phi_a := \phi(x,a) = (3a)^{1/4}( 1 + a |x|^2)^{-1/2}.
\]
In particular, for every $a > 0$, 
\[
\phi_a  \in \bigcap_{s > 1/2} \dot H^s(\bR^3) \setminus L^2(\bR^3).
\]
In \cite{Beceanu, KS14}, the long-time dynamics of sufficiently smooth symmetric perturbations of $\phi_a$ under the evolution \eqref{equ:ivp} were studied. We also refer to \cite{KNS15} for the construction of a center-stable manifold for \eqref{equ:ivp} without additional symmetry assumptions on the perturbations, however we focus on \cite{Beceanu, KS14} since our methods relate more closely to these works.  A key component of the analysis in these works is understanding the linearized operators
\begin{align}\label{equ:lin_hamil}
H_a = -\Delta -5 \phi_a^4 =: -\Delta + V_a, \quad a > 0.
\end{align}
In dimension three, these operators have the following spectral properties: the absolutely continuous spectrum of the operator $H_a$ is $[0, \infty)$, its point spectrum contains a single negative eigenvalue $-\kappa_a^2$, with corresponding eigenfunction $Y_a$ which is a radially symmetric, smooth and exponentially decreasing, it has three eigenvectors at zero, given by $\partial_{x_j} \phi_a(x)$, $1 \leq j \leq 3$ and a zero resonance given by $\partial_a \phi_a(x) = : \varphi_a(x)$. Recall that in dimension three, a zero resonance is a distributional solution $f$ to the equation
\[
-\Delta f + V f = 0, \quad f \in \bigcap_{\sigma <- \frac{1}{2}} L^{2, \sigma} \setminus L^2,
\]
where the space $L^{2,\sigma}$ is defined to be the closure of Schwartz functions under the norm
\[
\|f\|_{L^{2,\sigma}} =  \|\langle x \rangle^{\sigma} f\|_{L^{2}} .
\]
with an analogous definition for $\dot H^{s, \sigma}$. Restricting to the subspace of radial functions, $H_a$ has only one radial negative eigenfunction $Y_a$, a resonance $\varphi_a$ at zero and absolutely continuous spectrum $[0, \infty)$. In what follows, we set $\phi = \phi_1$, $Y = Y_1$, $\varphi = \varphi_1$, $\kappa = \kappa_1$, $V = - 5 \phi_1^4$ and $H = H_1$. We note that the presence of a negative eigenvalue implies linear instability of the solutions $\phi_a$.

In the work of Krieger and Schlag \cite{KS14}, the following is proved: there exists $\delta > 0$ such that for any radial  $(f_0, f_1) \in B_{\delta}(0) \subseteq  H^3 \times  H^2$ with certain compact support assumptions, which the authors comment can be replaced with sufficient decay assumptions, satisfying the orthogonality condition
\[
\langle \kappa f_0 + f_1, Y \rangle = 0,
\]
there exists a unique Lipschitz function $h$, which is quadratically small, such that \eqref{equ:ivp} with initial data
\[
(\phi + f_0 + hY, f_1)
\]
has a unique solution for all positive times which is the sum of a dynamically rescaled soliton and a dispersive term. In particular, there exists a codimension one Lipschitz submanifold of initial data with well-defined long time asymptotics. Their proof is based on dispersive estimates for the linearized evolution, and a key difficulty in their work is the fact that the presence of a resonance at the bottom of the absolutely continuous spectrum of $H_a$ implies that the standard dispersive estimates do not hold. Nonetheless, the authors show that by subtracting an appropriate rank-one operator, which is formally a projection, but not literally since the resonance does not belong to $L^2$, the standard dispersive decay (in this case $t^{-1}$) can be recovered.

This problem was revisited in Beceanu in \cite{Beceanu}, where perturbations were considered in the topology 
\begin{align}\label{bec_space}
\langle x \rangle^{-1} \dot H^1 \times \langle x \rangle^{-1} L^2 =: \{ (f_0, f_1) \in \dot H^1 \times L^2 \,:\, \|\langle x \rangle f_0 \|_{\dot H^1} + \| \langle x \rangle f_1 \|_{L^2} < \infty \}.  
\end{align}
in the symmetry class of functions satisfying $f(x) = f(-x)$. The operator $H_a$ restricted to the subspace of functions with reflection symmetry exhibits the same spectral properties as when it is restricted to the subspace of radial functions. Beceanu mentions that the space \eqref{bec_space} is not optimal in terms of weights, for which the optimal power is one half, although he further mentions that he restricts to this setting for clarity. Beceanu also speculates that \eqref{bec_space} may be optimal in terms of derivatives, see \cite[Remark 1.7]{Beceanu}. A key technical ingredient in the work \cite{Beceanu} is the Wiener space framework, used originally in \cite{Bec11,BeGo12, Goldberg12} to establish a family of reversed Strichartz estimates for the linearized flow.

Since $\phi_a \in \dot H^{s}$ for any $s > 1/2$, it is natural to ask about stability of the curve $(\phi_a, 0)$ under suitable perturbations in these rougher topologies. In particular, one may ask whether there exists a family of perturbations in $\dot H^s \times  H^{s-1}$ for some $1/2 < s < 1$ for which the stability of the stationary solution is still true. One difficulty is that the equation \eqref{equ:ivp} is known to be ill-posed for initial data in $\dot H^s \times \dot H^{s-1}$ for $s<1$, see for instance Christ-Colliander-Tao~\cite{CCT}. The current situation is more subtle due to the fact that we consider perturbations of the soliton in certain weighted spaces (to which the soliton does not belong), and we are allowing for quadratically small corrections in the direction of $Y$, see Remark \ref{equ:ill-posedness} for more details.

\medskip
In recent years, a probabilistic approach has been used in a variety of settings where deterministic results are known to break down. We adopt this perspective in the current work and our principal aim is to extend the investigations of  \cite{Beceanu, KS14} to rough and random (in particular infinite energy) perturbations of $\phi_a$. 

The study of dispersive PDEs via a probabilistic approach was initiated by Bourgain~\cite{B94, B96} for the periodic nonlinear Schr\"odinger equation in one and two space dimensions, building upon the constructions of invariant measures by Glimm-Jaffe~\cite{GlimmJaffe} and Lebowitz-Rose-Speer~\cite{LRS}. Such questions were further explored by Burq-Tzvetkov~\cite{BT1, BT2} in the context of the cubic nonlinear wave equation on a three-dimensional compact Riemannian manifold. There has since been a vast body of research where probabilistic tools are used to study many nonlinear dispersive or hyperbolic equations in scaling super-critical regimes, see for example~\cite{CO, NORS, D1, BT4, Bourgain_Bulut1, Bourgain_Bulut2, Suzzoni1, NPS, NS, LM1, BOP2, Pocovnicu, DLuM1, DLuM2} and references therein.  

Certain global-in-time random data results in the compact setting which rely on invariant measures work equivalently in the focusing and defocusing cases, see for instance \cite{B94}. However, in the absence of an invariant measure, to the best of our knowledge existing large data probabilistic results treat only the \textit{defocusing} nonlinear Schr\"odinger and wave equations, see \cite{LM1} for a treatment of the probabilistic small data theory for \eqref{equ:ivp}. Thus, in particular, there are no long-time large data probabilistic results for the corresponding \textit{focusing} nonlinear wave or Schr\"odinger equations on Euclidean space, even in perturbative regimes around the soliton. Hence, a more general goal of the current work is to demonstrate that probabilistic techniques can be used to treat focusing problems, and more specifically, used to address questions concerning the stability of stationary or special solutions of dispersive equations in the focusing setting.

\subsection{Statement of the main theorem and discussion of methods}
We consider the stationary solutions $\phi_{a}$ and we make the following ansatz:
\[
\psi(x,t) = u(x,t) + \phi_{a(t)}(x), \qquad a(0) = 1.
\]
We take the scaling parameter $a(t)$ to be time dependent, and as remarked in \cite{KS07}, this leads to analysis similar to modulation theory, but distinct since resonances are not associated with $L^2$ projections. The resonance 
\[
\varphi_{a(t)} := \partial_a \phi_{a(t)}
\]
satisfies
\[
H_{a(t)} \varphi_{a(t)}(x) := (- \Delta + V)  \varphi_{a(t)}(x) = 0, \qquad V_{a(t)} := - 5 \phi_{a(t)}^4.
\]
If $\psi$ solves \eqref{equ:ivp} with initial data $(\psi_0, \psi_1)$, then $u$ solves
\begin{equation} \label{equ:ivp_ansatz2}
 \left\{ \begin{aligned}
  -\partial_t^2 u - H_{a(t)} u &=  \partial_t^2 \phi_{a(t)} + N(u, \phi_{a(t)}) , \\
  (u, \partial_t u)|_{t=0} &= (\psi_0 - \phi_{a(0)} , \psi_1 - \dot a(0) \partial_a \phi_{a(0)}) \in  \cH^s(\bR^3),
 \end{aligned} \right.
\end{equation}
where the nonlinearity is given by
\begin{align}\label{equ:nonlin}
N (x,y) = 10 y^3 x^2 + 10 y^2 x^3 + 5 y x^4 + x^5.
\end{align}
We can write \eqref{equ:ivp_ansatz2} equivalently as 
\begin{equation} \label{equ:ivp_ansatz}
 \left\{ \begin{aligned}
  -\partial_t^2 u - H_{a(0)} u &=  \partial_t^2 \phi_{a(t)} + (V_{a(0)} - V_{a(t)}) u + N(u, \phi_{a(t)})   \\
  (u, \partial_t u)|_{t=0} &=  (\psi_0 - \phi_{a(0)} , \psi_1 - \dot a(0) \partial_a \phi_{a(0)}) \in  \cH^s(\bR^3).
 \end{aligned} \right.
\end{equation}

\medskip
To state our main theorem, we need to introduce some notation. We let $P_0$ denote a (smooth) distorted Fourier projection to low frequencies, whose precise definition we postpone to Section \ref{sec:randomization}, and let $f_{lo}$ and $f_{hi}$ denote smooth truncations to low and high regular Fourier frequencies, respectively, also defined in Section \ref{sec:randomization}. We let $P_{ac}$ denote the projection onto the absolutely continuous spectrum of $H$, see \eqref{p_ac}. We define two norms: for $0 < s < 1$, we let
\[
\| f\|_{\widetilde{X}_s} = \|f\|_{\cH^s} + \|P_{ac} f_{lo} + P_0 f_{hi}\|_{|\nabla|^{-s} L^{3/2,1} \times \langle \nabla\rangle^{-s + 1}  L^{3/2,1}}.
\]
and for $0 < s < 1$ and $s_1 > 3 \nu > 0$ we define
\[
\| f\|_{X_s} = \|f\|_{\cH^s} + \|\langle x \rangle^{1-\nu}|\nabla|^{s_1} f_0\|_{L^2_x} + \|\langle x \rangle^{1-\nu}\langle \nabla\rangle ^{s_1-1} f_1\|_{L^2_x}.
\]

\begin{remark}
For $s_1 = s$ and $\nu < 1/2$, one has $X_s \subseteq \widetilde{X}_s$, however we will retain the parameters $s_1$ and $\nu$ explicitly to highlight that we have some freedom in the $X_s$ norm. We note that we are considering inhomogeneous norms in the second coordinate since we do not want to necessitate stronger conditions on the low frequency component of $f_1$ than required in \cite{Beceanu}. 
\end{remark}

\begin{remark}
As one can see from the proof of Lemma \ref{lem:low_freq_bds}, the low frequency projection $P_0$ is bounded on $|\nabla|^{-s} L^{3/2,1} \times \langle \nabla\rangle^{-s + 1}  L^{3/2,1}$, and similarly for the smooth Fourier truncations. Hence, from the exponential decay of $Y$, and the fact that
\[
P_{ac} f = f - \langle f , Y \rangle Y
\]
from the definition of $P_{ac}$, see \eqref{p_ac} below, we see that the condition in the $\widetilde{X}_s$ norm is weaker than requiring the same bounds for the function $f$ itself. Moreover, this definition highlights the fact that the Lorentz condition will only be required for the low-frequency component, see \eqref{equ:phi_bds}.
\end{remark}

Let $s \in \bR$ and define
\[
\cM_s = \bigl \{ f \in \cH^s(\bR^3) \, :  \langle \kappa f_0 + f_1, Y \rangle = 0,\,\,\| f \|_{\widetilde{X}_s} < \varepsilon, \,\, \|f\|_{X_s} < \sqrt{ c \varepsilon^2/ \log C} \bigr \},
\]
for some $\varepsilon < \varepsilon_0$,  where $\varepsilon_0, C, c > 0$ are determined in Theorem \ref{codim_1} below.

\begin{remark}
In the case of finite energy initial data, the orthogonality condition
\[
 \langle \kappa \psi_0 + \psi_1, Y \rangle
\]
ensures that the second variation of the energy restricted to a certain codimension one Lipschitz submanifold is non-negative at $\phi_a$, see the discussion following \cite[Theorem 1]{KS07}. Although this does not apply to our (infinite energy) initial data, more concretely, we use this condition to avoid growth when solving an ODE for the pure point component of the solution, see \eqref{h_eq}.
\end{remark}

\begin{remark}\label{equ:ill-posedness}
Before stating our main theorem, we note that the ill-posedness results of \cite{CCT} do not apply directly to this setting since the soliton belongs to neither $L^2$ nor the weighted spaces defining $X_s$, and hence it is not known what the optimal well-posedness for \eqref{equ:ivp_ansatz} with small initial data in $\cM_s$ should be. In particular, the question of whether there exists a norm-inflation result analogous to \cite[Theorem 3]{CCT} for \eqref{equ:ivp_ansatz} with data in $\cM_s$ for $s < 1$ appears to be open.
\end{remark}

Let $(\Omega, \mathcal{A}, \mathbb{P})$ be a probability space.  We postpone the precise definition of our randomization to Definition \ref{def:randomization} below, but for the purposes of stating the main theorem, we denote the randomized data by $f^\omega = (f_0^\omega, f_1^\omega)$. As we will see below, for $f \in X_s$, we have $f^\omega \in X_s$ almost surely, see the discussion which concludes Section~\ref{sec:randomization} and Remark~\ref{rem:reg}.

\begin{theorem}\label{codim_1}
There exist absolute constants $C, c , \varepsilon_0> 0$ so that the following holds:  Let $s > 5/6$, $0< \varepsilon < \varepsilon_0$, and fix $f = (f_0, f_1) \in \cM_s$ and define its randomization $f^\omega = (f_0^\omega, f_1^\omega)$ according to Definition \ref{def:randomization}. Then there exists a subset $\Omega_{\varepsilon} \subseteq \Omega$ depending on $(f_0, f_1)$ satisfying
\[
\mathbb{P}(\Omega_{\varepsilon}) \geq 1 - C e^{-c \varepsilon^2 / \|f\|_{X_s}^2 },
\]
and a unique function $h = h_{(f_0, f_1)}: \Omega_{\varepsilon} \to \bR$ such that for every $\omega \in \Omega_{\varepsilon}$ we have $|h| \lesssim \varepsilon^2$ and there exists a unique solution $\psi(t)$ for $t \geq 0$ to the Cauchy problem
\begin{equation} \label{equ:ivp_cor}
 \left\{ \begin{aligned}
  -\partial_t^2 \psi + \Delta \psi &= - \psi^5 , \\
  (\psi, \partial_t \psi)|_{t=0} &=  (\phi + f_0^\omega + h_{(f_0, f_1)}(\omega) Y, f_1^\omega +  h_{(f_0, f_1)}(\omega)  \kappa Y),
 \end{aligned} \right.
\end{equation}
which is the sum of a modulated soliton $\phi_{a(t)}$, with $a(0) = 1$ and $\dot a \in L^1_t \cap L^\infty_t$, and a function in $L^8_{t,x}$, with norm $\lesssim \varepsilon$.
\end{theorem}

We will actually prove a slightly more general theorem, allowing a component of the initial data for \eqref{equ:ivp_cor} to lie in 
\begin{align}\label{n1}
\mathcal{N}_1 = \bigl \{ \psi \in \dot \cH^1(\bR^3) \, :  \langle \kappa \psi_0 + \psi_1, Y \rangle = 0,\,\, \|(\psi_0 - \phi, \psi_1) \|_{\widetilde{X}_1} < \varepsilon \bigr \},
\end{align}
see \eqref{equ:forced} and Proposition \ref{prop:fixed} below. The function $h$ is constructed as part of the contraction mapping argument in Proposition \ref{prop:fixed}. We note that in \cite{KS07}, initial data of the form
\[
 (\phi + f_0 + \widetilde{h}_{(f_0, f_1)}Y, f_1)
\]
for a quadratically small funtion $\widetilde{h}$, is considered. In this work, we use the same convention for the form of the initial data as in \cite{Beceanu}.

\begin{remark}[Uniqueness]
In the statement of Theorem \ref{codim_1}, uniqueness comes from the contraction mapping argument in Proposition \ref{prop:fixed}. In particular, uniqueness holds in the following sense: let $f \in  \cM_s$ and let $f^\omega$ be the randomization defined in Definition \ref{def:randomization}, and let $X$ be a complete metric space, the precise definition of which we record in \eqref{banach_x}.  We write the solution $\psi$ of \eqref{equ:ivp_cor} as
 \begin{align}
\psi(t) = \phi_{a(t)} +  u(t), \qquad u(t) = v(t) + W(t)  f_{\geq k_0}^\omega,
 \end{align}
 where $ f_{\geq k_0}^\omega$ is a certain high (in both distorted and standard Fourier) frequency component of the randomization, and
 \[
W(t)(f_0, f_1) = \cos(t \sqrt{|H|}) f_0 + \frac{\sin(t \sqrt{|H|}) }{\sqrt{|H|}} f_1.
 \]
  Then there exists a unique Lipschitz function $h_F$ of $(\overline{v}, \overline{a}) \in B_{\varepsilon}(0,1) \subseteq X$, for $\varepsilon > 0$ sufficiently small, and a unique solution  defined for all positive times
 \[
v \in C \bigl( \bR_+ ;  \dot{H}^1(\bR^3) \cap |\nabla|^{-1} L^{3/2,1}  \bigr), \quad (v,a) \in  X \quad (\textup{in }B_{\varepsilon}(0,1))
 \]
to the system of forced equations in \eqref{equ:x_eqn} -  \eqref{equ:a} with forcing term $F = W(t)  f_{\geq k_0}^\omega$ and $(\psi_0, \psi_1)$ replaced by
\[
(\phi + \gamma_0 + h_F(v,a) Y, \gamma_1+ h_F(v,a) \kappa Y),
\]
for $(\gamma_0, \gamma_1)$ certain low (distorted and standard) frequency components of the initial data. Furthermore, for fixed $(f_0, f_1)$ and $\omega$, the function $h$ is Lipschitz in the variables $(\overline{v},\overline{a})$, in a sufficiently small ball around $(0,1) \in X$.  See Proposition~\ref{prop:fixed} for more details
\end{remark}

We now make some comments regarding the interpretation of Theorem \ref{codim_1} within the known deterministic results for \eqref{equ:ivp}. The dichotomy result of Kenig and Merle \cite{KM08acta} addresses the dynamics of solutions to \eqref{equ:ivp} for initial data with energy below the energy of the soliton. In the discussion following \cite[Theorem 1]{KS07}, the authors comment that the orthogonality condition yields a codimension one manifold of initial data with energy \textit{larger} than the energy of the soliton, and in particular, the dichotomy result of  \cite{KM06} does not apply. In our current setting, the random initial data for \eqref{equ:ivp_cor} has infinite energy and hence the dichotomy result of \cite{KM06} also yields no information on the dynamics of the corresponding solutions. Finally, we refer to \cite{KNS13} for additional work on \eqref{equ:ivp} in a neighborhood of the soliton, without symmetry assumptions.

\begin{remark}
The small data type requirement on the function $f$ is reminiscent of the restriction in the small-data random-data global existence and scattering result of \cite[Theorem 1.6]{LM1}, and stems from the requirement that the subset $\Omega_{\varepsilon} \subseteq \Omega$ is required to have positive probability. This condition on $f$ is characteristic of perturbative random data arguments.
\end{remark}

\begin{remark}
Since it is unclear what ``optimal regularity'' means precisely for random data problems, we do not strive to obtain very low regularities in Theorem \ref{codim_1}. Rather, we wish to provide a proof of concept that random data methods are applicable in the focusing setting.
\end{remark}

\medskip
We will now discuss some aspects of the proof. Many of the previous random data results for nonlinear wave equations on Euclidean space rely on improved Strichartz estimates for the standard free propagator
\[
(f_0, f_1) \mapsto \cos(t \sqrt{-\Delta})f_0 + \frac{\sin(t \sqrt{-\Delta})}{\sqrt{-\Delta}} f_1 .
\]
To achieve this, either one randomizes with respect to an orthonormal eigenbasis of the Laplacian in the compact manifold setting, or, in previous works on Euclidean space, the typical randomization consists of randomizing with respect to unit-scale Fourier multipliers, see for instance \cite{ZF, LM1, LM2} and \cite{BOP1}. We also refer to the recent work \cite{Bringmann} where the initial data was randomized with respect to annular Fourier multipliers and \cite{Bringmann2} where a ``microlocal randomization'' is used. A key property these randomizations exploit is that the free wave propagator acts diagonally on Fourier space. Hence, when the randomization decouples the estimates into unit-scale (or annular) pieces, one can exploit the restricted Fourier support of these pieces to improve corresponding Strichartz estimates for the free evolution.

We will refer to solutions of 
\begin{align}\label{equ:linearized}
  -\partial_t^2 u - H_{a(0)} u = 0
\end{align}
as linearized solutions. In the current work, compared to previous works, we will need to obtain improved estimates for certain components of the linearized flow of the random initial data. To introduce a suitable randomization for this problem, we will randomize our initial data according to annular projections in \textit{distorted} frequency space, that is, in the frequency space associated with the distorted Fourier transform of the operator $H$. We postpone the precise definition of these projections until Section~\ref{sec:randomization}.

Many of the difficulties we encounter in adapting probabilistic arguments to our current setting are ultimately technical in nature, stemming from a lack of available harmonic analysis tools for the distorted Fourier transform related to the Schr\"odinger operator $H_a$. In particular, a key difficulty we encounter is the lack of an explicit expression for the generalized eigenfunctions associated to the operator $H_a$. These functions are the  distorted Fourier theory analogues of the characters $\{e^{\pm i x \cdot \xi}\}$ from standard Fourier theory, see Section \ref{sec:dist_four} for more details. Proving direct kernel estimates for distorted Fourier multipliers involves working with certain implicit equations for the generalized eigenfunctions, see \eqref{equ:radial_e} and \eqref{jost}. Furthermore, as we will see, the precise properties of the generalized eigenfunctions depend on the decay at infinity of the potential $V$. Moreover, our proofs rely crucially on the fact that the functions we consider are radial. We establish estimates for these functions and certain explicit kernels associated to distorted Fourier multipliers in Appendix \ref{a:kernel}. In addition to these harmonic analysis tools, we rely on the Wiener space machinery used in \cite{Beceanu}, see Section~\ref{sec:improved_strich} for more details.

With these tools in hand, one then proceeds to establish improved probabilistic estimates for the linearized flow. Some of these arguments are somewhat similar to previous works, largely relying on a combination of large deviation and Bernstein estimates, albeit with some additional complications in the current setting. The most delicate estimate, however, is a (negatively) weighted $L^2_x L^1_t$ estimate which arises due to a linear term in the contraction mapping argument. This estimate is treated in Proposition \ref{prop:delicate}, and is the reason we require the initial data to lie in weighted spaces. In certain previous works on almost sure global wellposedness \cite{Pocovnicu} and almost sure scattering \cite{Bringmann2} for the defocusing energy critical equations in four dimensions, an $L^1_tL^\infty_x$ estimate arises when controlling the energy of a certain nonlinear component of the solution via Gronwall's inequality. In \cite{Pocovnicu}, this estimate is obtained locally in time, while in \cite{Bringmann2} a microlocal randomization is used to control a component of the free evolution projected onto frequency $N$ in $L^1_t L^\infty_x([N^{1+\theta}, \infty) \times \bR^4)$ for some $\theta > 0$. Here we require global in time bounds, but the key difference between those settings and the current one is that while the dimension three case is a priori more difficult due to reduced dispersion for the free evolution, we allow the random data to enjoy weighted Sobolev bounds, as opposed to only imposing Sobolev bounds. We also refer to Remark \ref{rem:delicate} for further discussion of this estimate.

Upon establishing improved probabilistic estimates for the relevant component of the linearized evolution of the random data, we will solve for both $u(x,t)$ and $a(t)$ via a fixed point argument in a suitable complete metric space. To carry out the rest of our argument, we use the functional framework of \cite{Beceanu}, see Section \ref{sec:contraction_setup}. Since our randomization will almost surely not regularize in Sobolev spaces, we will need to proceed as in previous works on random data theory and recenter the solution around the linearized evolution of the random data. A key difference in this setting, however, is that we will first need to treat an equation for the projection of the solution onto the negative eigenvalue. As in \cite{DLuM1, DLuM2}, we will work with the framework of a generalized forced nonlinear wave equation, but in this case we will only use this framework for the component of the solution projected onto the absolutely continuous subspace of the linearized operator $H_a$, see Section \ref{sec:contraction} as well as the system of equations \eqref{equ:x_eqn} -  \eqref{equ:a} for more details.

We conclude this discussion by elaborating more on the function $h$ in the statement of Theorem~\ref{codim_1} below. For a sufficiently small $\varepsilon > 0$, and $f\in \cM_s$ the set $\Omega_\epsilon \subset \Omega$ in the probability space (which depends, in particular, on $f$) is such that for $\omega \in \Omega_\epsilon$, we have that
\[
F:=W(t)  f_{\geq k_0}^\omega \in  Z
\]
where $Z$ is the Banach space of improved spacetime norms required for the contraction mapping argument, and is based on Strichartz spaces and certain weighted spactime norms. On the $\varepsilon$-ball around $(0,1)$ in $X$, the proof of Proposition \ref{prop:fixed} yields a unique function $h$, which depends on $F$, defined for $(\overline{v}, \overline{a})$ in this $\varepsilon$-ball, and which is Lipschitz in $(\overline{v}, \overline{a})$. We then solve the corresponding  nonlinear problem by fixed point in $(v,a)$, which is how we choose the final function $h$. Ultimately, this yields a solution to the Cauchy problem \eqref{equ:ivp_cor}. We emphasize that in particular, the final choice for $h$, depends on $f^\omega$.

\subsection*{Organization of paper}
In Section~\ref{sec:prelim} we record some definitions and preliminary facts about the stationary solution, resonance, and potential. In Section~\ref{sec:dist_four} we establish some facts about the distorted Fourier transform for a class of Schr\"odinger operators $H= - \Delta + V$, which include our linearized operator. In Section~\ref{sec:randomization} we introduce the randomization procedure. In Section~\ref{sec:improved_strich} we establish certain deterministic estimates for the linearized flow including stating our key kernel estimates. In Section~\ref{prob:strichartz} we establish improved probabilistic estimates for the linearized flow. In Section~\ref{sec:contraction} we perform the set-up and estimates for our main contraction mapping argument and prove Theorem~\ref{codim_1}. The Appendix is dedicated to the proof of the kernel estimates.

\subsection*{Notation}
We let $\mathcal{S}(\bR^d)$ denote Schwartz functions on $\bR^d$. We denote by $C > 0$ an absolute constant which only depends on fixed parameters and whose value may change from line to line. We write $X \lesssim Y$ to indicate that $X \leq C Y$ and we use the notation $X \sim Y$ if $X \lesssim Y \lesssim X$. Moreover, we write $X \lesssim_\nu Y$ to indicate that the implicit constant depends on a parameter $\nu$ and we write $X \ll Y$ if the implicit constant should be regarded as small. We will write $c+$ to denote $c+ \varepsilon$ for any $\varepsilon > 0$. We also use the notation $\langle x \rangle := (1+x^2)^{1/2}$.

\section{Preliminaries}\label{sec:prelim}
First we record the definition of certain weighted Sobolev spaces which were already introduced in the introduction.

\begin{definition} We let $L^{2,\sigma}$ denote the closure of Schwartz functions under the norm
\begin{align}\label{equ:weighted}
\|f\|_{L^{2,\sigma}} =  \|\langle x \rangle^{\sigma} f\|_{L^{2}} .
\end{align}
We similary define $\dot H^{s,\sigma}$ as the closure of Schwartz functions under the norm
\begin{align}
\|f\|_{\dot H^{s,\sigma}} =  \|\langle x \rangle^{\sigma}|\nabla|^s f\|_{L^{2}} .
\end{align}
with an analogous definition for the inhomogeneous space $H^{s,\sigma}$.
\end{definition}

\medskip
Next, we recall the definition of the Lorentz spaces and several of their properties.

\begin{definition}[Lorentz spaces]
For a function $f : \bR^n \to \mathbb{C}$, we let
\[
\mu_f(\lambda) = |\{x \,:\, |f(x)| > \lambda\}|.
\]
Then the $L^{p,q}(\bR^n)$ quasi-norm for $f$ is defined by
\[
\|f\|_{L^{p,q}} := \biggl( p\int_0^\infty (\lambda \mu_f(\lambda)^{\frac{1}{p}})^q \frac{d\lambda}{\lambda} \biggr)^{\frac{1}{q}}.
\]
\end{definition}
\begin{remark}
For $q = \infty$ the previous definition corresponds to the weak $L^p$ space, while $L^{p,p} = L^p$. 
\end{remark}
The following lemma summarizes some of the key real interpolation properties of the Lorentz spaces.
\begin{lemma}[\protect{Cf. \cite[Theorem 5.3.1]{Ber-Lof}}]\label{lem:interp}
For $1 < p < \infty$, we have
\[
L^{p,q} = [L^1, L^\infty]_{\theta, q}, \qquad \frac{1}{p} = 1 - \theta.
\]
Moreover $L^{p, q_1} \subset L^{p, q_2}$ for $0<q_1 \leq q_2 \leq \infty$. 
\end{lemma}

\begin{remark}
In particular, we have $L^{p,1} \subset L^p \subset L^{p,\infty}$.
\end{remark}

\begin{lemma}[H\"older's inequality in Lorentz spaces, \protect{ \cite[Theorem 3.4]{Oneil}}]\label{lem:lorentz_holder}
For $1< p_1,p_2,p_3 < \infty$, and $0 < q_1, q_2, q_3  \leq \infty$, with
\[
\frac{1}{p_1} = \frac{1}{p_2} + \frac{1}{p_3} , \qquad \frac{1}{q_1} = \frac{1}{q_2} + \frac{1}{q_3} ,
\]
one has
\begin{align}
\|f g\|_{L^{p_1, q_1}} \lesssim \|f\|_{L^{p_2, q_2}} \| g\|_{L^{p_3, q_3}}.
\end{align}
\end{lemma}

\begin{lemma}[Young's inequality in Lorentz spaces,\protect{ \cite[Theorem 3.1]{Oneil}}]\label{lem:lorentz_young}
For $1< p_1,p_2,p_3 < \infty$, and $0 < q_1, q_2, q_3  \leq \infty$, with
\[
\frac{1}{p_1} + 1 = \frac{1}{p_2} + \frac{1}{p_3} , \qquad \frac{1}{q_1} \leq \frac{1}{q_2} + \frac{1}{q_3} ,
\]
one has
\begin{align}
\|f * g\|_{L^{p_1, q_1}} \lesssim \|f\|_{L^{p_2, q_2}} \| g\|_{L^{p_3, q_3}}.
\end{align}
\end{lemma}

We will need some properties of the stationary solution
\begin{align}\label{equ:form_stat}
\phi_a =  (3a)^{1/4}( 1 + a |x|^2)^{-1/2},
\end{align}
the resonance $\varphi_a = \partial_a \phi_a$, and the potential $V_a = - 5 \phi_a^4$. We record the following result from \cite{KS07}, with a proof for completeness.
\begin{lemma} \label{lem:symb}
Let $a$ be a function with $a(s) \in (1/2, 3/2)$, $a(0) = 1$ and $\dot a \in L^1$. Then the resonance $\varphi$ satisfies
\[
\varphi_{a(s)} = a(s)^{- 5/4} \varphi_{a(0)} + (a(s) - a(0)) \mathcal{O}(\langle x \rangle^{-3}), 
\]
where $\mathcal{O}$ satisfies symbol type estimates. Consequently
\[
\|\dot a(s) ( \varphi_{a(s)} - a(s)^{-5/4} \varphi_{a(0)}) \|_{L^1_t |\nabla|^{-1} L^{3/2,1}_x \cap L^1_t \dot H^1_x \cap L^\infty_t |\nabla|^{-1} L^{3/2,1}_x \cap L^\infty_t \dot H^1_x} \lesssim \|\dot a\|_{L_t^1} \|\dot a\|_{L_t^1 \cap L^\infty_t}.
\]
\end{lemma}
\begin{proof}
By differentiating \eqref{equ:form_stat} in $a$, we have 
\begin{align}\label{equ:res_description}
\varphi_a = \frac{1}{4}3^{1/4} a^{-3/4}( 1 + a |x|^2)^{-1/2} - \frac{1}{2} |x|^2 (3a)^{1/4} ( 1 + a |x|^2)^{-3/2},
\end{align}
which for $|x|$ large yields
\[
\varphi_a = -\frac{1}{4}3^{1/4} a^{-5/4}|x|^{-1} + \mathcal{O} (\langle x \rangle^{-3}),
\]
where we use the $\mathcal{O}$ notation to denote a smooth remainder term satisfying symbol-type estimates, which depends on $a$ but is uniformly bounded for $a$ in the specified range. We can similarly check in \eqref{equ:res_description} that
\[
\varphi_{a_1} = \left(a_2/a_1 \right)^{-5/4} \varphi_{a_2} + (a_1 - a_2) \mathcal{O} (\langle x \rangle^{-3}),
\]
which yields the first claim. The second claim is then immediate.
\end{proof}
The following result, which proceeds by direct computation, will be used in conjunction with Proposition \ref{prop:delicate} to prove an estimate for a linear term in the contraction mapping argument of Section \ref{sec:contraction}.

\begin{lemma}\label{v_bds}
Let $a$ be a function with $a(s) \in (1/2, 3/2)$, $a(0) = 1$ and $\dot a \in L^1$. Then for any $1 < p, q \leq \infty$, $0 <  \theta <  3 - \frac{3}{p}$, we have
\begin{align}
\|\langle x \rangle^{1 + \theta} (V -V_{a(t)} )\|_{L^{p,q}_x L_t^\infty} &\lesssim \|\dot a\|_{L^1}.
\end{align}
\end{lemma}
\begin{proof}
By direct computation, we have the following point-wise decay bounds
\[
|\partial_a V_a(x)| \lesssim  \langle \sqrt{a} x \rangle^{-4}.
\]
By the fundamental theorem of calculus
\[
|V - V_{a(t)}| \leq  \int_{1}^{a(t)} |\partial_a V_{a}| da  \leq  \int_{1}^{1 + \|\dot a\|_{L^1}} |\partial_a V_{a}| da  .
\]
Hence, integrating in $a$ after a change of variables, we obtain
\begin{align}
\langle x \rangle^{1+\theta} |V - V_{a(t)}| \leq \biggl| \int_{1}^{1 + \|\dot a\|_{L^1}} \langle x \rangle^{1+\theta} \langle \sqrt{a}x \rangle^{-4} da \biggr|  \lesssim \|\dot a\|_{L^1} \langle x \rangle^{1 + \theta - 4},
\end{align}
which yields the result.
\end{proof}

Arguing similarly, we have the following result
\begin{corollary}\label{cor:phi_bds}
Under the assumptions of Lemma \ref{v_bds}, for any $3 < p,q \leq \infty$, $0 <  \alpha <  1 - \frac{3}{p}$, we have
\begin{align}
\|\langle x \rangle^{\alpha} \phi_{a(t)} \|_{L^{p,q}_x L_t^\infty} &\lesssim (\|\dot a\|_{L^1} + 1)
\end{align}
\end{corollary}
\begin{proof}
We write
\[
\|\langle x \rangle^{\alpha} \phi_{a(t)} \|_{L^{p,q}_x L_t^\infty} = \|\langle x \rangle^{\alpha} ( \phi - \phi_{a(t)}) \|_{L^{p,q}_x L_t^\infty} + \|\langle x \rangle^{\alpha} \phi \|_{L^{p,q}_x L_t^\infty} 
\]
and use the same proof as Lemma \ref{v_bds}, and the fact that $\phi_{a} \lesssim \langle x \rangle^{-1}$.
\end{proof}

\section{Distorted Fourier theory} \label{sec:dist_four}

In this section, we will collect the necessary preliminaries on distorted Fourier theory for the linearized operator $-\Delta - 5 \phi^4$, stated for  slightly more general operators $H = -\Delta + V$, where $V$ is a real-valued potential which satisfies certain decay assumptions which we make explicit shortly.  Many of the results we state can be found in the work of Agmon \cite{Agmon}, and we refer to \cite{Agmon} work and references therein for further details. 

Before beginning our discussion, we set the usual notation for the resolvent set of $H$,
\begin{align}
\rho(H) &= \{ z \in \mathbb{C} \,: (H- z): H^2 \to L^2 \textup{ is bijective and } (H - z)^{-1} \textup{ is bounded}\}, 
\end{align}
and we define the spectrum of $H$
\begin{align}
\sigma(H) &= \mathbb{C} \setminus \rho(H).
\end{align}
We recall that the essential spectrum of an operator $T$, denoted $\sigma_{\textup{ess}}(T)$, is the set of all $\lambda \in \bC$ such that $\lambda I - T$ is not a Fredholm operator. When $T$ is self-adjoint, we may equivalently define $\sigma_{\textup{ess}}(T) := \sigma(T) \setminus \sigma_{\textup{disc}}(T)$, for $\sigma_{\textup{disc}}(T)$ the set of all isolated points of $\sigma(T)$, see \cite{ReedSimon}.

Let $H_0$ denote the self-adjoint realization of $-\Delta$ with domain $ \mathcal{D}(-\Delta) = H^2(\bR^n)$. For real-valued potentials $V$ which are $H_0$-compact, which we recall means that multiplication by $V$ is compact  as an operator from $\mathcal{D}(-\Delta) \to L^2$,  the operator $H = -\Delta +V$ has a unique self-adjoint realization on $L^2(\bR^n)$ with domain $\mathcal{D}(H) =  \mathcal{D}(-\Delta)  =  H^2(\bR^n)$. Furthermore, 
\[
\sigma_{\textup{ess}}(H) = \sigma_{\textup{ess}}(-\Delta) = \bR_+,
\]
 see for instance \cite[\S XIII.4]{ReedSimonIV}. In this setting, the spectral theorem guarantees the existence of a unique projection valued spectral measure $E_{\lambda}$ such that for bounded Borel functions $g$, we have
\[
g(H) = \int_{\sigma(H)} g(\lambda) d E_{\lambda}.
\]
This measure enjoys a Lebesgue decomposition into singular, pure point, and absolutely continuous component, see \cite[p230]{ReedSimon}, and we will denote the absolutely continuous component of the spectral measure by $E_{\lambda, ac}$.

For many of the results in \cite{Agmon},  the potential $V$ is assumed to satisfy
\begin{align}\label{equ:decay}
\sup_{x \in \bR^n} \left[ ( 1 + |x|)^{2 + 2 \varepsilon}  \int_{|x-y| \leq 1} |V(y)|^2 |x - y|^{-n + \mu} \right] < \infty
\end{align}
for some $\varepsilon > 0$ and $0 < \mu < 4$. It is further noted that \eqref{equ:decay} holds in particular for potentials which satisfy 
\begin{align}\label{decay2}
V(x) = O(|x|^{-1-\varepsilon}), \quad \textup{as  }|x| \to \infty \qquad \textup{and} \qquad V \in L^p_{loc}(\bR^n)\textup{ where }
\begin{cases}
p=2 & n \leq 3 \\
p > n/2 & n \geq 4.
\end{cases}
\end{align}
For potentials satisfying \eqref{equ:decay}, it is known (see for instance \cite[Theorem 6.1]{Agmon}) that
\begin{align}\label{orthog_decomp}
L^2(\bR^n) =L^2_{ac}(\bR^n) \oplus L^2_p(\bR^n),
\end{align}
where $L_p^2(\bR^n)$ denotes the closed span of the eigenfunctions of $H$, and where $L^2_{ac}$ is the absolutely continuous subspace of $L^2(\bR^n)$ for $H$, that is, the set of all vectors in $L^2$ whose corresponding spectral measures are absolutely continuous, see, for instance \cite[Theorem VII.4]{ReedSimon} for a general definition of this subspace and more details. In particular, we have
\[
\sigma_{ac}(H) := \sigma(H \bigl|_{L^2_{ac}}) = \bR_+.
\]
We let $P_{ac}$ denote the orthogonal projection of $L^2(\bR^n)$ onto $L^2_{ac}(\bR^n)$.

\medskip
In dimension three, we use the following definition of resonance, which was stated in the introduction but which we recall here:

\begin{definition}\label{def:res}
A zero resonance is a distributional solution $f$ to the equation
\[
-\Delta f + V f = 0, \quad f \in \bigcap_{\sigma <- \frac{1}{2}} L^{2, \sigma} \setminus L^2,
\]
where the space $L^{2,\sigma}$ is defined in \eqref{equ:weighted}. We may analogously define a $\lambda$-resonance to be a distributional solution $f$ to the equation
\[
-\Delta f + V f = \lambda f, \quad f \in \bigcap_{\sigma <- \frac{1}{2}} L^{2, \sigma} \setminus L^2,
\]

\end{definition}

\begin{remark}\label{no_eig}
Let $H = - \Delta + V$, where $V$ satisfies \eqref{equ:decay}. By \cite[Theorem C.4.2]{SimonBAMS}, if there exists a non-zero, distributional solution $f$ to
\[
H f = \lambda f
\]
such that $f$ is polynomially bounded but not in $L^2$, then $\lambda \in \sigma_{\textup{ess}}(H) = \bR_+$, and hence there are no $\lambda$ resonances for $\lambda < 0$. Furthermore, by \cite[\S3]{Agmon} there are no solutions
\[
-\Delta f + V f = \lambda f, \quad f \in \bigcap_{\sigma < -\frac{1}{2}} L^{2, \sigma} \setminus L^2,
\]
for $\lambda > 0$ for potentials $V$ satisfying \eqref{equ:decay}, and thus there are no $\lambda$ resonances for $\lambda > 0$ either. Hence, the set of non-zero $\lambda \in \bR$ such that there exists a non-zero, polynomially bounded, distributional solution $f$ to
\[
H f = \lambda f
\]
is precisely the set of non-zero eigenvalues of $H$. Of course, this discussion does not preclude the existence of a resonance at zero, which will be the case in our setting.
\end{remark}

We let $e_+(H)$ denote the set of positive eigenvalues of $H$ and we let
\[
\mathcal{N}(H) = \{\xi \,:\, \xi \in \bR^n, |\xi|^2 \in e_+(H)\} \cup \{0\}.
\] 

In the sequel we will work with rapidly decaying, smooth and radial potentials $V$. To streamline our exposition, we define the following:

\medskip
\noindent \textbf{Assumption (A1):} $V$ is real-valued, radial, smooth and satisfies
\[
|V(x)| \lesssim \langle x\rangle^{-4},   \qquad |\nabla V(x) | \lesssim \langle x\rangle^{-5}.
\]

\medskip

In the majority of our results below, we will restrict our attention to Schr\"odinger operators $H = - \Delta + V$ with radial potentials satisfying Assumption (A1). In particular, for such potentials we have  $V \in L^{2,\sigma}$ with $\sigma > 1/2$, which is the assumption typically used in \cite{Agmon}, $V \in L^{3/2,1}$ which is the assumption used in \cite{BS16}, as well as $V \in L^\infty$. We emphasize that in the results of this section, Assumption (A1) is far from optimal, and we refer to \cite{BS16} for some discussion of the optimal assumptions on the potential in three dimensions.

Under Assumption (A1), the Schr\"odinger operator $H$ satisfies the following properties:
\begin{enumerate}
\item[(i)] $\mathcal{N}(H) = \{0\}$,
\item[(ii)] $\sigma_{ac}(H) = \bR_+$,
\item[(iii)] no positive or negative resonances, but possibly a resonance at zero,
\item[(iv)] finite dimensional kernel,
\item[(v)] $e_{+}(H) = \emptyset$ and there are finitely many negative eigenvalues with exponentially decaying negative eigenfunctions, $\{Y_j\}_{j=1}^m$.
\end{enumerate}
See for instance \cite{Kato59}, \cite[Section 8]{DaviesBook} and \cite{Meshkov}. If we let $\{Y_j\}_{j=1}^m$ denote the set of eigenfunctions of $H$ associated to the negative eigenvalues $\{-\kappa_j^2\}_{j=1}^m$, then
\begin{align}\label{p_ac}
P_{ac} f = f - \sum_{j = 1}^m \langle Y_j, f \rangle Y_j.
\end{align}
Furthermore, $P_{ac}$ is bounded on $\dot H^s \cap L^p$ for $s \in \bR$ and $1 \leq p \leq \infty$.

By definition of $\sigma(H)$, for $z \in \mathbb{C} \setminus \sigma(H)$ the resolvent operator
\[
R_V(z):= (-\Delta + V - z)^{-1}
\]
is a bounded operator from $L^2 \to H^2$. For $\lambda \in \mathbb{R} \setminus e_+(H)$, the limits
\begin{align}\label{equ:lim_res}
R^{\pm}_V(\lambda^2) := \lim_{ \varepsilon \to 0^+} (-\Delta + V - (\lambda \pm i \varepsilon)^2)^{-1},
\end{align}
can be defined via the limiting absorption principle in certain spaces, see for instance \cite[Theorem 4.1]{Agmon} which demonstrates that this limit holds in $\mathcal{B}(L^{2,\sigma}, H^{s,-\sigma})$ for $\sigma > 1/2$ equipped with the usual operator topology, see also \cite{RodTao15, GoldbergSchlag04} and references therein for other works on the limiting absorption principle. It will be useful in the sequel to note that 
\begin{align}\label{equ:resolvent_complex_conj}
R^{+}_V(\lambda^2) = \overline{R^{-}_V(\lambda^2)}.
\end{align}

In the case that $V = 0$, we obtain the free resolvent, which in three dimensions is given explicitly for $\textup{Im}\: \lambda \neq 0 $ by the kernel
\[
R_0(\lambda^2)(x,y) = \frac{e^{ i \lambda |x-y|}}{|x-y|}.
\]
The free resolvent is analytic in $\mathbb{C} \setminus [0, \infty)$ and discontinuous along $[0, \infty)$.
 
 \medskip
Distorted Fourier theory for the operator $H$ is based on the construction of generalized eigenfunctions $\{e_{\pm}(x,\xi)\}$ which play the role of the characters $\{e^{\pm ix \cdot \xi}\}$ in the Fourier theory of the Laplacian.  The existence of these generalized eigenfunctions and many of their properties were established by Agmon \cite{Agmon} under the condition \eqref{equ:decay}. More precisely, to establish the existence of a family of generalized eigenfunctions, i.e. functions 
\[
e(x,\xi):  \bR^n \times \mathbb{R}^n \setminus \mathcal{N}(H) \to \bC
\]
which satisfy
 \begin{align}\label{efunction_equation}
( -\Delta + V(x) - |\xi|^2) e(x,\xi) = 0,
 \end{align}
we consider the ansatz
\[
e_{\pm}(x,\xi) = e^{\pm i x \cdot \xi} + v_{\pm}(x,\xi).
\]
From \eqref{efunction_equation}, one derives that the functions $v_{\pm}$ satisfy
 \begin{align}\label{eval1}
 (-\Delta + V(x) - |\xi|^2) v_{\pm}(x,\xi) = - V(x) e^{ \pm i x \cdot \xi}.
 \end{align}
For potentials $V \in L^{2, \sigma}$ for $\sigma > 1/2$, the limiting operators $R_V^{\pm}(|\xi|^2)$ exist in $\mathcal{B}(L^{2,\sigma}, H^{2,-\sigma})$ by \cite[Theorem 4.1]{Agmon}, and solutions $\{e_{\pm}(x, \xi)\}$ of the so-called Lippman-Schwinger equation
\begin{align}\label{equ:lippman_schwinger}
e_\pm(\cdot, \xi) = e_{\pm,\xi} - R_V^{\mp}(|\xi|^2) [V e_{\pm,\xi}]  , \qquad e_{\pm,\xi}(x) = e^{\pm ix \cdot \xi}
\end{align}
are solutions to \eqref{efunction_equation}.

\medskip
We state the following theorem for the generalized eigenfunction expansion due to \cite{Agmon}, see also \cite{KatoKuroda70} for an earlier result with more restrictive assumptions on the potential.

 \begin{theorem}[Cf. \protect{\cite[Theorem 5.1]{Agmon}}]\label{thm:agmon}
Let $H = -\Delta + V$ with $V \in L^{2, \sigma}$ for $\sigma < - 1/2$. For every $\xi \in \bR^n \setminus \mathcal{N}(H)$, let
\begin{align}\label{equ:lippman_schwinger_thm}
e_\pm(\cdot, \xi) = e_{\pm,\xi} - R_V^{\mp}(|\xi|^2) [V e_{\pm,\xi}]  , \qquad e_{\pm,\xi}(x) = e^{\pm ix \cdot \xi}.
\end{align}
Then the families $e_\pm(x,\xi)$ have the following properties:
\begin{enumerate}
\item[a)] As a function of $x$ and $\xi$, $e_\pm(x, \xi)$ is a measurable function with
\[
e_\pm(x, \xi) \in L^2_{loc}(\mathbb{R}^n \times (\mathbb{R}^n \setminus  \mathcal{N}(H))).
\]
\item[b)] For every fixed $\xi$, the function $e_\pm(x, \xi) \in C(\mathbb{R}^n) \cap H^2_{loc}(\mathbb{R}^n)$ and satisfies the differential equation
\[
(-\Delta +  V(x) -|\xi|^2) e_\pm(x, \xi) = 0.
\]
\end{enumerate}
\end{theorem}

Now that we have defined the generalized eigenfunctions for $H$, we record a definition and some properties of the distorted Fourier transform.  The version we state is due to Agmon, but see for instance \cite{Kato70,KatoKuroda70} and references therein for earlier results with different assumptions on the potential.

\begin{theorem}[Cf. \protect{\cite[Theorem 6.2]{Agmon}}] \label{distorted}
Let $H = -\Delta + V$ for a real-valued potential $V \in L^{2, \sigma}$ for $\sigma < - 1/2$. Let $e_{\pm}(x,\xi)$ denote the two families of generalized eigenfunctions for $H$. Then there exist bounded linear operators 
\[
\cF_{V,\pm} : L^2(\bR^n) \to L^2(\bR^n),
\]
satisfying the following properties:
\begin{enumerate}
\item[(i)] The restriction of $\cF_{V,\pm}$ to $L^2_{ac}$ is a unitary operator to $L^2(\bR^n)$ and
\begin{align}\label{equ:iso}
\langle P_{ac} f ,P_{ac} g \rangle = \langle F_{V, \pm} f , F_{V, \pm} g \rangle, \qquad f, g \in L^2(\bR^n).
\end{align}
In particular, $\textup{Ker}(\mathcal{F}_{V,\pm}) = L^2_{p}$.
\item[(ii)] For any $f \in L^2(\bR^n)$, the maps $\cF_{V, \pm}$ take the form
\begin{align}\label{distorted_limit}
(\cF_{V,\pm} f)(\xi) = (2\pi)^{-n/2}  \lim_{N \to \infty} \int_{|x| < N} \overline{e_\pm (x, \xi)} f(x) dx, 
\end{align}
and
\[
(\cF_{V,\pm}^* f)(x) = (2\pi)^{-n/2} \lim_{j \to \infty} \int_{K_j} e_\pm (x, \xi) f(\xi) d\xi, 
\]
where $K_j$ is any increasing sequence of compact sets with $\cup_{j=1}^\infty K_j = \mathbb{R}^n \setminus \mathcal{N}(H)$.
\end{enumerate}
\end{theorem}
\begin{remark}
From $(i)$ in the previous theorem, $\mathcal{F}_{V, \pm}(Y_j) = 0$, hence using the orthogonal decomposition \eqref{orthog_decomp}, it holds that 
\[
\mathcal{F}_{V, \pm}P_{ac} = \mathcal{F}_{V, \pm}.
\]
In \cite[Theorem 6.2]{Agmon}, Agmon in fact proves that $\cF_{V,\pm}$ restricted to $L^2_{ac}$ is a unitary operator onto $L^2(\bR^n)$, but we will not need this fact. 
\end{remark}

\begin{remark}\label{fin_supp}
If $\textup{supp }f(\xi) \subseteq \mathbb{R}^n \setminus \mathcal{N}(H)$ is compact, then
\[
(\cF_{V,\pm}^* f)(x) = (2\pi)^{-n/2} \int_{\bR^n} e_\pm (x, \xi) f(\xi) d\xi.
\]
\end{remark}

\medskip
We now return to the setting of dimension $n= 3$. In previous works on random data Cauchy theory, a key ingredient in the improved probabilistic estimates for the free evolution was a unit-scale Bernstein estimate for certain Fourier projections, see for instance \cite{LM1, LM2}. Analogously, in this work, we will crucially rely on a Bernstein type estimate for certain distorted Fourier projections to obtain improved probabilistic estimates for the linearized evolution of the random initial data. In the proof of these estimates, we will need to show that for sufficiently smooth and decaying potentials, $e_{\pm}(x, \xi) \in L^\infty_x$. We now turn to establishing these bounds.

By the resolvent identity, 
\begin{align}\label{res_identity}
R_V(z) =  R_0(z) - R_0(z) V R_V(z) = R_0(z) -  R_V(z) V R_0(z), \quad z \in \mathbb{C} \setminus \bigl(\bR_+ \cup \sigma(H) \bigr)
\end{align}
which holds in the sense of bounded operators on $L^2$, or equivalently
\begin{align}\label{res_formula_2}
(I + R_0(z) V)^{-1} = I - R_V(z ) V.
\end{align}

We need to use this identity in the limit as $\textup{Im }z \to 0$. By \cite[Lemma 3.2]{BS16}, for $V \in L^{3/2,1}$ the operator $(I + R_0^{\mp}(|\xi|^2) V)$ is invertible in $\mathcal{B}(L^\infty, L^\infty)$, and furthermore, this holds if and only if $R_V^{\mp}(|\xi|^2)$ is bounded from $L^{3/2,1}$ to $L^\infty$, see \cite[(3.3)]{BS16}. Consequently, for such potentials, $R_V^{\mp}(|\xi|^2) V$ is a bounded operator on $L^\infty$, and we may use \eqref{res_formula_2} to rewrite the expression \eqref{equ:lippman_schwinger} for $e_{\pm}(x,\xi)$ as
\begin{align}\label{eq:gen_four}
e_{\pm} (\cdot, \xi) = \bigl(I - R_V^{\mp}(|\xi|^2) V \bigr) \, e_{\pm,\xi} = (I + R_0^{\mp}(|\xi|^2) V)^{-1} e_{\pm,\xi}.
\end{align}

Thus, to establish bounds for $e_{\pm}(x, \xi)$, we prove a bound for the operator $(I + R_0^{\pm}(\lambda^2)V)^{-1}$. The proof of a similar fact is discussed in \cite[Lemma 2.4]{Bec14} but we include here the precise statement we need and a proof for completeness.
 
 \begin{lemma}\label{lem:lim_ap}
 Let $V$ be a potential satisfying Assumption (A1). Then for every $\lambda_0 > 0$, one has
 \begin{align}\label{equ:lim_ap}
\sup_{\lambda \geq \lambda_0} \|(I + R_0^{\pm}(\lambda^2)V)^{-1} \|_{L^\infty \to L^\infty} \leq C(\lambda_0, V).
 \end{align}
 \end{lemma}

\begin{remark}
Recently, it was shown in \cite[Lemma 3.1]{BS16} that for $V \in L^{3/2,1}$ the following definition is equivalent to Definition \ref{def:res}: We say that $H = -\Delta +V$ has a resonance at zero if there exists $f \in L^\infty \setminus L^2$, $f \neq 0$ which satisfies
\[
f = -R_0^{\pm}(0) V f.
\]
Thus, the operator $(I + R_0^{\pm}(0) V)$ is invertible on $L^\infty$ if and only if zero is neither a resonance nor an eigenvalue. Hence, in the current setting, the condition $\lambda_0 > 0$ in the previous lemma is unavoidable.
\end{remark}

\begin{proof}[Proof of Lemma \ref{lem:lim_ap}]
We only prove the statement for $R_0^{+}$ as the result for $R_0^{-}$ follows by complex conjugation, see \eqref{equ:resolvent_complex_conj}. We write
\[
(I + R_0^+(\lambda^2)V)^{-1} = (I - R_0^+(\lambda^2)V)\bigl[I - (R_0^+(\lambda^2)V)^2\bigr]^{-1}.
\]
By \cite[Lemma 3.3]{BS16}, for $V \in L^{3/2,1}(\bR^3)$ we have
\[
\|(R_0^+(\lambda^2)V)^{2}\|_{L^\infty \to L^\infty} \longrightarrow 0, \quad\textup{as } \lambda \to \infty.
\]
Hence there exists $\lambda_1 \gg1$ such that
\[
\|(R_0^+(\lambda^2)V)^{2}\|_{L^\infty \to L^\infty} < \frac{1}{2}, \qquad \lambda \geq \lambda_1,
\]
and so for $\lambda \geq \lambda_1$, we can define $\bigl[I - (R_0^+(\lambda^2)V)^2\bigr]^{-1}$ via a uniformly convergent Neumann series in $\mathcal{B}(L^\infty, L^\infty)$, satisfying
\[
\|\bigl[I - (R_0^+(\lambda^2)V)^2\bigr]^{-1}\|_{L^\infty \to L^\infty} \leq \sum_{\ell = 0}^\infty \left( \frac{1}{2}\right)^{\ell} \leq 1, \qquad \lambda \geq \lambda_1.
\]
By direct computation (see \cite[p8]{BS16}), we have
\[
\|R_0^+(\lambda^2)V\|_{L^\infty \to L^\infty} = \|R_0^+(0)V\|_{L^\infty \to L^\infty},
\]
which is independent of $\lambda$.  Hence,
\begin{align}
\| (I + R_0^+(\lambda^2)V)^{-1} \|_{L^\infty \to L^\infty} &\leq (1 + \|R_0^+(\lambda^2)V\|_{L^\infty \to L^\infty})\| \bigl[I - (R_0^+(\lambda^2)V)^2\bigr]^{-1} \|_{L^\infty \to L^\infty} \\
&\leq C(V, \lambda_1), \qquad \lambda \geq \lambda_1.
\end{align}
Thus, it suffices to prove that $(I + R_0^+(\lambda^2)V)^{-1}$ remains uniformly bounded in $\lambda$ for $\lambda_0 \leq \lambda \leq \lambda_1$. We will show the stronger statement that 
\[
\lambda \to (I + R_0^+(\lambda^2)V)^{-1}
\]
is a continuous function in $\lambda > 0$ with values in $\mathcal{B}(L^\infty, L^\infty)$ on this set. To see this, we first note that if $(I + A_1)$ and $(I+ A_2)$ are invertible operators on a normed space, which satisfy, in the operator norm,
\begin{align}
\|(I+ A_1)\| = M_1, \quad \|(I+ A_1)^{-1}\| = \widetilde{M}_1, \quad \|A_1 - A_2\| \leq \varepsilon_0,
\end{align}
for $\varepsilon_0 > 0$, small, then
\begin{align}
\|I - (I+A_1)^{-1}(I+A_2) \| = \|(I+A_1)^{-1}[(I + A_1) - (I+A_2) ]\| \leq \widetilde{M}_1 \varepsilon_0,
\end{align}
and hence, using a Neumann series,
\[
\|(I+ A_2)^{-1}(I+ A_1)\| \leq 1 + 2 \widetilde{M}_1 \varepsilon_0.
\]
But then
\begin{align}
\|(I+ A_2)^{-1}\| = \|(I+ A_2)^{-1} (I+A_1)(I+A_1)^{-1}\| \leq [ 1 + 2 \widetilde{M}_1 \varepsilon_0] \widetilde{M_1} =: \widetilde{M}_2,
\end{align}
and hence
\begin{align}
\|(I+ A_2)^{-1} - (I+ A_1)^{-1}\| &= \|(I+ A_2)^{-1}[I - (I+A_2)(I+A_1)^{-1}] \| \\
&\leq \widetilde{M}_2\|I - (I+A_2)(I+A_1)^{-1} \| \\
&= \widetilde{M}_2\|\bigl[ (I + A_1) - (I+A_2)\bigr](I+A_1)^{-1} \| \\
&\leq \widetilde{M}_2\widetilde{M}_1 \varepsilon_0.
\end{align}
By \cite[Lemma 3.2]{BS16}, the operator $I + R_0^+(\lambda^2) V$ is invertible in $\mathcal{B}(L^\infty, L^\infty)$ for any $\lambda \neq 0$ and the map
\[
\lambda \to I + R_0^+(\lambda^2)V
\]
is continuous in $\mathcal{B}(L^\infty, L^\infty)$ for such $\lambda$, hence the assertion follows.
\end{proof}

\medskip
\begin{corollary}\label{cor:lim_ap}
Fix $c > 0$. There exists $C \equiv C(c, V)$ such that
\[
\|e_{\pm}(\cdot, \xi) \|_{L^\infty} \leq C, \qquad |\xi| \geq c.
\]
\end{corollary}
\begin{proof}
We have
\[
e_{\pm}(\cdot ,\xi) =  (I + R_0^{\mp}(|\xi|^2) V)^{-1} e_{\pm,\xi},
\]
and the result follows by Lemma \ref{lem:lim_ap}.
\end{proof}

 \medskip
We may now use Corollary \ref{cor:lim_ap} to establish an $L_{x}^1\to L_{\xi}^\infty$ boundedness result for the distorted Fourier transform. 

\begin{lemma}\label{lem:l1linfty}
Let $V$ be a potential satisfying Assumption (A1). Let $f \in L^1 \cap L^2$ and $c > 0$. Then 
\[
\mathcal{F}_{V,\pm} f \in  L^\infty_\xi(|\xi| \geq c)
\]
and we have the estimate
\[
\|\mathcal{F}_{V,\pm} f \|_{L^\infty_\xi(|\xi| \geq c) } \leq C \|f \|_{L^1_x}
\]
where $C\equiv C(c, V)$.
\end{lemma}
\begin{proof}
Using Theorem \ref{distorted}, we write
\begin{align}
|\mathcal{F}_{V,\pm} f(\xi)| &= (2\pi)^{-n/2}  \lim_{N \to \infty} \int_{|x| < N} \overline{e_\pm (x, \xi)} f(x) dx.
\end{align}
Fix $N \gg 1$, and estimate
\begin{align}
 \left| \int_{|x| < N} \overline{e_\pm (x, \xi)} f(x) dx \right| \leq \|e_{\pm}(\cdot ,\xi)\|_{L^\infty_x}\int_{|x| < N} | f | dx \leq \|e_{\pm}(\cdot ,\xi)\|_{L^\infty_x} \|f\|_{L^1_x} 
\end{align}
Hence
\begin{align}
|\mathcal{F}_{V,\pm} f(\xi) | &\leq \|e_{\pm}(\cdot ,\xi)\|_{L^\infty_x} \|f\|_{L^1_x}
\end{align}
and the result follows from Corollary \ref{cor:lim_ap}.
\end{proof}

\subsection{Distorted Fourier Multipliers}\label{ssec:dist_four}
First we begin with a Bernstein result for general distorted Fourier multipliers. This result will play a key role in many of our improved probabilistic estimates in Section~\ref{prob:strichartz}. Before proceeding, we note that for a bounded Borel function, the multiplier operator
\[
T_m^{\pm} (f) = \mathcal{F}_{V, \pm}^* m(\cdot)  \mathcal{F}_{V, \pm} f.
\]
is $L^2$ bounded by Theorem \ref{distorted}.

\begin{proposition}[Bernstein inequality]\label{prop:bernstein}
Let $V$  potential satisfying Assumption (A1). Consider $f \in L^2 \cap L^\infty$, and let $m : \bR^n \to \bC$ be a Borel measurable function satisfying $\|m\|_{L^\infty} \leq M$, with $\textup{supp } m \subset E \subset \mathbb{R}^n$, for $E$ a bounded, measurable subset with $|E| < \infty$ and $ \textup{dist}(E, 0) = d > 0$. Let 
\[
T_m^{\pm} (f) = \mathcal{F}_{V, \pm}^* m(\cdot)  \mathcal{F}_{V, \pm} f.
\]
Assume further that there exists $2 \leq r \leq \infty$ such that $\|T_m^{\pm} \|_{L^r \to L^r} \leq A$. 

Then for $0 \leq \theta \leq 1$ and any $r \leq q \leq \frac{r}{\theta}$ we have $T_m^{\pm} f \in L^q$ and there exists $C \equiv C(\theta,r, q, d, V, M,A)$ such that
\[
\|T_m^{\pm} f\|_{L^q_x} \leq C |E|^{\frac{r -(r-2)\theta}{2r} - \frac{1}{q}}  \|f\|_{L^{\frac{2r}{r-(r-2)\theta}}_x}.
\]
\end{proposition}

\begin{proof}
Since $\mathcal{N}(H) = \{0\}$, under the assumptions on the support of $m$, we obtain using Remark~\ref{fin_supp} that
\begin{align}
|T_m^{\pm} f(x) | = |\cF_{V,\pm}^*m(\cdot) \cF_{V,\pm} f| &= \left| \int e_{\pm} (x, \xi) m(\xi) \mathcal{F}_{V,\pm} f (\xi)d\xi \right| \\&\leq M  \int_E | (1 + R_0^{\pm}(|\xi|^2) V)^{-1} e_{\pm, \xi}|| \mathcal{F}_{V,\pm} f (\xi)|d\xi
\end{align}
By Lemma \ref{lem:lim_ap} and Theorem \ref{distorted},
\begin{align}
\int_E | (1 + R_0^{\pm}(|\xi|^2) V)^{-1} e_{\pm, \xi} || \mathcal{F}_{V,\pm} f (\xi)| & \leq C |E|^{\frac{1}{2}} \|f\|_{L^2_x},
\end{align}
and hence
\[
\|T_m^{\pm} f\|_{L^\infty_x} \leq C |E|^\frac{1}{2}  \|f\|_{L^2_x}.
\]
Since $\|T_m^{\pm} f\|_{L^2_x} \leq C \|f\|_{L^2_x}$, by the Riesz-Thorin interpolation theorem we obtain
\begin{align}\label{interp1}
\|T_m^{\pm} f\|_{L^p_x} \leq C |E|^{\frac{1}{2} - \frac{1}{p}}  \|f\|_{L^2_x}, \qquad 2 \leq p \leq \infty.
\end{align}
Since by assumption 
\begin{align}\label{interp2}
\|T_m^{\pm} f\|_{L^r_x} \leq A \|f \|_{L^r_x},
\end{align}
and for any $0 \leq \theta \leq 1$ and $r \leq q \leq \frac{r}{\theta}$, there exists $r \leq p \leq \infty$ such that 
\[
\frac{1}{q} = \frac{(1-\theta)}{p} + \frac{\theta}{r},
\]
the result then follows by interpolating between \eqref{interp1} and \eqref{interp2} using the Riesz-Thorin interpolation theorem.
\end{proof}

\begin{remark}
In our applications, we will take $r = 3 -$ in the previous proposition.
\end{remark}

We will only use $\mathcal{F}_{V, -}$ and $e_-(x,\xi)$ and hence we drop the $\pm$ notation. In the sequel, we will consider radial multipliers, in which case there is an equivalent formulation of distorted Fourier multipliers as spectral multipliers. We elaborate a bit on this equivalence. For the self-adjoint realization of $H = -\Delta + V$ on $L^2$, it holds (cf. \cite[Theorem VIII.5]{ReedSimon}) that for any bounded Borel function $m$, the spectral multiplier $m(H)$ is defined via
 \begin{align}\label{spec_mult}
\langle m(H) f, f \rangle = \int_{-\infty}^\infty m(\lambda) d \langle E_\lambda f, f \rangle,
 \end{align}
and the polarization identity, where once again, we recall that $E_\lambda = E_{(-\infty, \lambda)}$ denotes the projection valued spectral measure for $H$. In this case, the multiplier $m(H)$ is a bounded operator on $L^2$ which satisfies
 \[
 \|m(H)\|_{2 \to 2} \leq \|m\|_{L^\infty}.
 \]
For unbounded Borel functions, we let
 \[
\mathcal{D}(m) = \bigl  \{ f \in L^2 \,:\, \int_{- \infty}^\infty |m (\lambda)|^2 d \langle E_\lambda f , f \rangle < \infty  \bigr\},
 \]
and we note that $\mathcal{D}(m)$ is dense in $L^2$ provided $m$ is finite on the discrete spectrum of $H$ and $m$ is almost everywhere finite on the absolutely continuous spectrum of $H$. Then for  $f \in \mathcal{D}(m)$ we may similarly define
  \begin{align}\label{spec_mult_unbdd}
\langle m(H) f, f \rangle = \int_{-\infty}^\infty m(\lambda) d \langle E_\lambda f, f \rangle.
 \end{align}
 In both the unbounded and bounded cases, we will write
 \[
 m(H) = \int_{-\infty}^\infty m(\lambda) d  E_\lambda,
 \]
 where this operator is understood as an operator on $L^2 \cap \mathcal{D}(m)$ via \eqref{spec_mult_unbdd} and the polarization identity.
 Before proceeding, we note for later that Stone's formula for the spectral measure implies (see for instance \cite[Theorem V11.13 and remarks on p264]{ReedSimon}):
\[
\frac{1}{2} \left(E_{[a,b]}  + E_{(a,b)} \right) = \lim_{\varepsilon \to 0} \frac{1}{2\pi i} \int_a^b \bigl[(H - \lambda - i\varepsilon)^{-1} - (H - \lambda + i\varepsilon)^{-1} \bigr]d\lambda,
\]
Consequently (see for instance \cite[\S1.4.3]{YafaevBook}) 
\begin{align}\label{equ:meas_prop}
\frac{\langle dE_{\lambda, ac} f, f \rangle}{d\lambda}  =  \lim_{\varepsilon \to 0}   \pi^{-1}  \langle \bigl[(H - \lambda - i\varepsilon)^{-1} - (H - \lambda + i\varepsilon)^{-1} \bigr] f, f \rangle, \quad \textup{a.e. } \lambda \in \bR
\end{align}
for a set of full measure in \eqref{equ:meas_prop} which depends on the function $f$, where we recall that $dE_{\lambda, ac}$ is the absolutely continuous component (in the sense of Lebesgue decomposition) of the spectral measure, and further
\[
\frac{\langle dE_{\lambda, ac} f, f \rangle}{d\lambda} \in L^1(\bR).
\]
Thus, for an operator $-\Delta + V$ with potential satisfying Assumption (A1), and for any bounded Borel measurable function $m : \bR \to \bC$, we have
 \begin{align}\label{stone}
 m(H) P_{ac} f = \frac{1}{\pi} \int_{0}^\infty m(\lambda) [R^{+}_V(\lambda) - R^{-}_V(\lambda)] d\lambda f
  \end{align}
for $f \in L^2$. We note that in our setting $P_{ac} f = \chi_{\bR_+}(H) f$.
 
Now, by the proof of \cite[Theorem 6.2]{Agmon} and the polarization identity, it holds that
\begin{align}\label{borel_equiv}
\langle \chi_{\mathcal{O}}(H) f, g \rangle = \int_{|\xi|^2 \in \mathcal{O}} (\cF_{V} f)(\xi) \overline{(\mathcal{F}_{V} g)(\xi)} d\xi,
\end{align}
for any open set $\mathcal{O}$ in $\bR_+ \setminus e_+(H)$ and any $f, g \in L^2$, and in our setting, $e_+(H) = \emptyset$. Consequently, for a bounded Borel measurable function $m: \bR \to \bC$, we may define the distorted Fourier multiplier
\begin{align}
&m_H : L^2(\bR^n) \to L^2(\bR^n)\\
&m_H f = \mathcal{F}_{V}^* \, m(|\cdot|^2) \mathcal{F}_{V} f, \qquad f \in L^2(\bR^n),
\end{align}
and by Theorem \ref{distorted}, we similarly have
\[
 \|m_H\|_{2 \to 2} \leq \|m\|_{L^\infty}.
 \]
By \cite[Theorem VIII.6]{ReedSimon},
\[
m(H) P_{ac} = m_H = m_H P_{ac}
\]
for bounded Borel functions $m$, hence we will use the notation $m(H)$. Furthermore, if $H g = \lambda_0 g$, then $m(H) g = m(\lambda_0) g$, and
\[
HP_{ac}  f = \cF_{V,\pm}^* |\cdot|^2 \cF_{V,\pm} f, \quad f \in \mathcal{D}(H),
\]
In particular, the right-hand side is a well-defined element of $L^2$ for $f \in \mathcal{D}(H)$, see for instance \cite[Theorem 6.2]{Agmon} or \cite[\S VII]{ReedSimon}.

\medskip
Next, we record the following result of Jensen and Nakamura \cite{JN95, JN94} which will be used to control certain low distorted frequency components of the initial data.

\begin{theorem}[\protect{\cite[Theorem 2.1]{JN94}}]\label{thm:jn}
Let $H= - \Delta + V$ with $V$ a potential satisfying Assumption (A1), let $1 \leq p \leq \infty$, and let $g \in C_0^\infty(\bR)$. Then there exists $C > 0$ such that for any $0 < \theta \leq1$,
\[
\|g(\theta H) \|_{L^p \to L^p} \leq C.
\]
Moreover, the estimate is uniform for $g$ in a bounded subset $G \subseteq C_0^\infty(\bR)$, i.e. if there is $R > 0$ and a sequence $\{C_\alpha\}$ such that for any $g \in G$, 
\[
\textup{supp}\, g \subset [-R, R]  \quad \textup{ and } \quad |\partial_x^\alpha g| \leq C_{\alpha}.
\]
\end{theorem}

\begin{corollary}\label{cor:low}
Let $\psi \in C_0^\infty(\bR)$ be a smooth bump function supported on $|x| \leq c$, and let $H= -\Delta + V$ for $V$ a  potential satisfying Assumption (A1). Then for $1 \leq  p \leq \infty $, there exits $C \equiv C(c, \psi)$ such that for $ 0 \leq s \leq 1$,
\[
\| |\nabla|^s \psi (H) \|_{L^p \to L^p} \leq C
\]
\end{corollary}
\begin{proof}
We argue using analytic interpolation with the family of operators 
\[
e^{z^2} (-\Delta)^z \psi(H),
\] 
The exponential factor is required to cancel with the bound
\begin{align}\label{complex_bd}
\|(-\Delta)^{i\gamma} \|_{L^p \to L^p} \leq A (1 + |\textup{Im}\, z|)^N,
\end{align}
for $N \equiv N(d,p)$ and constant $A$, see \cite{Fefferman_stein}. For $f, g \in \mathcal{S}$ with $\textup{supp } \widehat{f} \subseteq B(0, R)$ for some $R > 0$, we consider
\[
F(z) = \langle f, e^{z^2} (-\Delta)^z \psi(H) g \rangle.
\]
We note that such functions $f$ and $g$ are dense in $L^{p'}$ and $L^p$ respectively and for such functions $g$ we have $\psi(H) g \in H^2$ since $\psi(H) g \in L^2$ and  $H\psi(H) g \in L^2$. Thus, for this class of $f,g$ the function $F(z)$ is well-defined. Furthermore, $F(z)$ is continuous on 
\[
S = \{z \,:\, 0 \leq \textup{Re z} \leq 1/2\} \subseteq \mathbb{C}
\]
and analytic on the interior of $S$.  We then have from \eqref{complex_bd} and Theorem \ref{thm:jn} that
\[
|F(i \gamma)| \leq A \|f\|_{L^{p'}} \|g\|_{L^p}.
\]
Next, we claim that
\[
\| (-\Delta) \psi (H) f \|_{L^p} \lesssim_c \|  f \|_{L^p}.
\] 
Indeed, we write
\begin{align}
\| (-\Delta) \psi (H)  f \|_{L^p}& \leq \| H \psi (H)  f \|_{L^p} + \| V \psi (H)  f \|_{L^p} \\
&\leq \| H\psi (H) f \|_{L^p} + \| V\|_{L^\infty_x} \| \psi (H) f \|_{L^p},
\end{align}
and the claim follows from Theorem \ref{thm:jn} with $g(x) = x \psi(x)$ or $g(x) = \psi(x)$. Hence we similarly have that
\[
|F(1+i \gamma)| \leq A \|f\|_{L^{p'}} \|g\|_{L^p}.
\]
and we apply the three lines theorem and the density of the chosen class of $f$ to conclude that
\[
e^{z^2} (-\Delta)^z \psi(H) g \in L^p,
\]
and then the density of the chosen class of $g$ in $L^p$ to conclude.
\end{proof}

The following inhomogeneous Sobolev-type estimate for the operator $H$ follows directly from \cite[Theorem B.2.1]{SimonBAMS}.

\begin{proposition}[cf. \protect{\cite[Theorem B.2.1]{SimonBAMS}}]\label{h_sobolev}
Let $H = -\Delta +V$ with $V$ a potential satisfying Assumption (A1). Let $M > - \inf \sigma(H)$. Then for $1 < p \leq q < \infty$ and $\beta \in \bR$ satisfying
\[
\beta > \frac{3}{2} \left( \frac{1}{p} - \frac{1}{q} \right),
\]
$(H + M)^{-\beta}$ is a bounded operator from $L^p \to L^q$. Consequently, for such $p ,q , \beta$ we have the inhomogenous Sobolev embedding estimate:
\[
\| f\|_{L^q} \leq \|(H + M)^{\beta} f \|_{L^p}.
\]
\end{proposition}

In the sequel, we will also need the following weighted multiplier estimates, the proof of which we postpone until Appendix \ref{a:kernel}.

 \begin{lemma}\label{lem:weighted_ests}
Let $H = -\Delta + V$ with $V$ a potential satisfying Assumption (A1). Let $0 < \alpha < 1$, and let $\varphi \in C^\infty \cap L^\infty$ be supported on $x \geq c > 0$, then there exists $N = N(\alpha ) > 0$ so that the following holds: suppose that
\[
|\partial_x^\ell \varphi(x)| \lesssim \frac{1}{x} ,\quad  \ell = 1, \ldots N(\alpha).
\]
Then for radial functions $f \in L^{2,\alpha}$, we have
\[
\|\langle r \rangle^{\alpha} \varphi(H) f \|_{L^2_x} \lesssim \||\langle r \rangle^{\alpha}   f \|_{L^2_x},
\]
where the implicit constant depends on semi-norms of $\varphi$ up to order $N(\alpha)$.
 \end{lemma}
 The proof of the previous lemma will also yield the following fact:
 
 \begin{lemma}\label{lem:weighted_ests_free}
Let $\alpha > 0$, and let $\varphi \in C^\infty$ be a bounded function. Then for $f \in L^{2,\alpha}$ we have
\[
\|\langle r \rangle^{\alpha} \varphi(-\Delta) f \|_{L^2_x} \lesssim \|\langle r \rangle^{\alpha}  f \|_{L^2_x},
\]
where the implicit constant depends on semi-norms of $\varphi$ up to order $N \equiv N(\alpha)$.
 \end{lemma}

\subsection{Coercivity results}
We record some coercivity results for the operator $H = -\Delta + V$ in the case that $H$ has a resonance at zero, which is our current setting. Due to the zero-resonance, we will restrict to functions supported away from zero in distorted Fourier space in some of our results. In what follows, it will be useful to note the following facts: consider the set
\[
\mathcal{C} = \{ g \in L^2 \,:\, \chi_{[\varepsilon, R]}(H) g = g \, \textup{ for some } \varepsilon, R > 0 \},
\]
then $\mathcal{C}$ is dense in $L^2_{ac}$. Indeed, we have that
\[
 \|P_{ac} g - \chi_{[\varepsilon , R]}(H) P_{ac} g \|_{L^2} \to 0
\]
by the isometry property of Theorem \ref{distorted}. We further note that if $Y$ is an eigenfunction associated to an eigenvalue $-\kappa^2$, then
\[
\chi_{(-\infty, R]}(H) Y = \chi_{(-\infty, R]}(-\kappa^2) Y = Y,
\]
hence for $g \in L^2$, functions of the form $\chi_{(-\infty, R]}(H) g$ are dense in $L^2$.

Now, for $g \in \mathcal{C}$ and $s \in \bR$, the operator $|H|^s := H^s P_{ac}$ can be defined either by the functional calculus or via distorted Fourier multipliers, and furthermore these definitions agree by the above discussion since 
\[
|H|^s g = |H|^s \chi_{[\varepsilon, R]}(H) g \qquad \textup{for some } \varepsilon, R > 0,
\]
and $m(\lambda) = \lambda^{s} \chi_{[\varepsilon, R]}(\lambda)$ is a bounded function. Now let $g \in L^2$ and suppose
\[
\int |\xi|^{2s} |\mathcal{F}_V g(\xi)|^2 < \infty \qquad \textup{and} \qquad \int_{0}^\infty |\lambda|^{2s} d \langle E_\lambda g , g \rangle < \infty.
\]
For such functions, $|H|^s \,\chi_{[\varepsilon , R]}(H) g\in L^2$ uniformly in $R, \varepsilon$.

\begin{lemma}\label{lem:norm bound}
Let  $H = - \Delta + V$ with $V$ a potential satisfying Assumption (A1). Let $f \in L^2$, then the following estimates hold:
\begin{enumerate}
\item[(i)] Let $0 \leq s \leq 1/2$. Then
\begin{align}\label{coercive1}
\| |H|^s  (-\Delta)^{-s} f \|_{L^2} &\lesssim \| f \|_{L^2}.
\end{align}
\item[(ii)] Let $0 \leq s \leq 1$. If $\varphi \in C^\infty(\bR) \cap L^\infty(\bR)$ is a bounded function supported on $x \geq c > 0$, then
\begin{align}\label{coercive1b}
\| (-\Delta)^s  |H|^{-s} \varphi(H) f \|_{L^2} &\lesssim \| f \|_{L^2}.
\end{align}
\item[(iii)] Let $-1 \leq s \leq 1$. If $M >  - \inf \sigma(H)$, then
\begin{align}\label{coercive2b}
\|  (H + M)^{s} (-\Delta + M)^{-s}  f \|_{L^2} &\lesssim \| f \|_{L^2}.
\end{align}
\end{enumerate}
\end{lemma}
\begin{proof}
We argue using interpolation for analytic families of operators. Let $f, g \in \mathcal{S}$ with 
\[
\textup{supp }\widehat{f} \subseteq \mathbb{R}^3 \setminus B(0,r),\quad \textup{and} \quad \chi_{(-\infty, R]}(H) g = g, \quad r, R > 0, 
\]
and recall that such functions are dense in $L^2$. Consider
\[
F(z) = \langle |H|^z (-\Delta)^{-z} f, g \rangle_{L^2} = \langle  (-\Delta)^{-i\, \textup{Im}\: z} (-\Delta)^{-\textup{Re} \:z} f, |H|^{-i\,\textup{Im} \:z}|H|^{\textup{Re}\: z}  g \rangle_{L^2} , \qquad \textup{Re z} \geq 0.
\]
By the assumptions on $f$ and $g$, the function $F(z)$ is well-defined, since 
\[
H^{-i \,\textup{Im} \:z}  |H|^{\textup{Re}\: z}  g \in L^2\quad \textup{and} \quad(-\Delta)^{-i\, \textup{Im}\: z} (-\Delta)^{-\textup{Re} \:z} f \in L^2.
\]
Further, $F(z)$ is continuous on 
\[
S = \{z \,:\, 0 \leq \textup{Re z} \leq 1/2\} \subseteq \mathbb{C}
\]
and analytic on the interior of $S$.  We claim that
\[
\| |H|^{\frac{1}{2}} (-\Delta)^{-\frac{1}{2}} f \|_{L^2} \lesssim \| f \|_{L^2}.
\]
First observe that $\mathcal{D}(|H|^{\frac{1}{2}}) = \dot H^1$, since $\mathcal{D}(|H|^{\frac{1}{2}})$ is equal to the form-domain of $HP_{ac}$, that is, the largest dense subspace of $L^2_{ac}$ for which the quadratic form
\[
\psi \mapsto \langle \psi, HP_{ac} \psi \rangle
\]
is bounded. We write
\begin{align}
\| |H|^{\frac{1}{2}} (-\Delta)^{-\frac{1}{2}} f \|_{L^2}^2 &=  \langle  |H|^{\frac{1}{2}} (-\Delta)^{-\frac{1}{2}} f,  |H|^{\frac{1}{2}}  (-\Delta)^{-\frac{1}{2}} f\rangle \\
&=  \langle  H P_{ac} (-\Delta)^{-\frac{1}{2}} f, P_{ac} (-\Delta)^{-\frac{1}{2}} f\rangle,
\end{align}
and we can use the definition of $P_{ac}$ from \eqref{p_ac} to expand this expression as
\begin{align}
 &\langle  H P_{ac} (-\Delta)^{-\frac{1}{2}} f,  P_{ac}(-\Delta)^{-\frac{1}{2}} f\rangle \\
 &=  \langle  H (-\Delta)^{-\frac{1}{2}} f,  (-\Delta)^{-\frac{1}{2}} f\rangle - 2 \sum_{j = 1}^m (-\kappa_j)^2 | \langle Y_j, (-\Delta)^{-\frac{1}{2}} f\rangle |^2  + \sum_{j = 1}^m |\kappa_j|^4 | \langle Y_j, (-\Delta)^{-\frac{1}{2}} f\rangle|^2,
 \end{align}
 where $-\kappa_j^2$ is the eigenvalue associated to $Y_j$. Since $Y_j \in L^{6/5}$ and $(-\Delta)^{-1/2} f \in L^6$, the second and third terms are bounded. For the first term, we write
\begin{align}
\langle H (-\Delta)^{-\frac{1}{2}}  f ,  (-\Delta)^{-\frac{1}{2}}  f\rangle
& = \langle (I + (-\Delta)^{- \frac{1}{2}} V (-\Delta)^{-\frac{1}{2}} ) f , f\rangle \\
&=  \langle f, f \rangle + \langle V (-\Delta)^{-\frac{1}{2}} f ,  (-\Delta)^{-\frac{1}{2}}  f\rangle,
\end{align}
The first term yields exactly $\|f\|^2_{L^2}$, while the second term can be bounded using the Sobolev embedding $\dot H^{1}(\bR^3) \subset L^6(\bR^3)$ together with the fact the $V \in L^{3/2,1}$. 

Using that, by the spectral theorem, imaginary powers of $(-\Delta)$ and $|H|$ are $L^2$-bounded,
\begin{align}
|F(i\gamma)| &\leq \||H|^{i\gamma} (-\Delta)^{-i\gamma} f\|_{L^2} \| g\|_{L^{2}} \leq \| f\|_{L^{2}} \| g\|_{L^{2}}\\
|F(1/2 + i\gamma)| &\leq \||H|^{i\gamma}  |H|^{\frac{1}{2}}  (-\Delta)^{-\frac{1}{2} } (-\Delta)^{- i\gamma} f\|_{L^2} \| g\|_{L^{2}} \lesssim \| f\|_{L^{2}} \| g\|_{L^{2}}, 
\end{align}
and we apply the three lines theorem, the density of the chosen class of $g$ to obtain $(i)$ for a dense set, and density of the chosen class of $f$ to conclude. 

For $(ii)$, we argue similarly: let $f, g \in \mathcal{S}$ be such that $\textup{supp }\widehat{g} \subseteq B(0,R)$ for some $R > 0$ and again note that such functions are dense in $L^2$. We set
\[
F(z) = \langle (-\Delta)^z |H|^{-z} \varphi(H) f, g \rangle_{L^2}, \qquad \textup{Re }z \geq 0.
\]
and again have that $F(z)$ is well-defined. The rest of the argument proceeds similarly, noting that
\[
\| (-\Delta) |H|^{-1}  \varphi(H) f \|_{L^2} \leq \| (-\Delta + V) |H|^{-1}  \varphi(H) f \|_{L^2} + \| V |H|^{-1}  \varphi(H) f \|_{L^2} \leq \|f\|_{L^2},
\]
where we are using in the second term that $V\in L^\infty$, that $|H|^{-1}  \varphi(H)$ is an $L^2$ bounded operator because of the support of $\varphi$, and that $(-\Delta + V) |H|^{-1} = H |H|^{-1} = P_{ac}$.

For $(iii)$, let $f , g \in \mathcal{S}$ with 
\[
\textup{supp }\widehat{f} \subseteq B(0,R_1) \setminus B(0,r), \quad \textup{and} \quad  \chi_{(-\infty, R_2)}(H) g = g, \quad r, R_1, R_2 > 0,
\]
(which are again dense in $L^2$) and let
\[
F(z) =  \langle  (H + M)^{z} (-\Delta + M)^{-z}  f, g \rangle_{L^2} ,
\]
For $M > - \inf \sigma(H)$, the operator $(H+M)^{-1}$ is bounded on $L^2$, hence $F(z)$ is well-defined.  Furthermore, $F(z)$ is continuous on  
\[
S = \{z \,:\, -1 \leq \textup{Re z} \leq 1\}
\]
 and analytic on the interior of $S$.  Since $V \in L^\infty$, we may estimate
\begin{align}
&\|(H + M)^{-1} (-\Delta + M)f\|_{L^2} \\
&\leq \|  (H + M )^{-1} (H + M) f\|_{L^2} +  \| (H + M)^{-1} V f\|_{L^2}   \\
&\lesssim \| f\|_{L^2}.
\end{align}
Similarly, using that since $M > 0$, $(-\Delta + M)^{-1}$ is $L^2$ bounded, we obtain
\begin{align}
&\|(H + M) (-\Delta + M)^{-1} f\|_{L^2} \\
&\leq \| (-\Delta + M) (-\Delta + M)^{-1} f\|_{L^2} +  \| V (-\Delta + M)^{-1} f\|_{L^2}   \\
&\lesssim \| f\|_{L^2},
\end{align}
Finally, using that $H+M$ is a positive operator, $(H+M)^{i\gamma}$ is $L^2$ bounded by the spectral theorem. Hence, we conclude as above using the three lines theorem. 
\end{proof}

\begin{corollary}\label{cor:equivalence}
Let $H = -\Delta + V$ with $V$ a potential satisfying Assumption (A1). Then for $f \in H^2 $, we have
\begin{enumerate}
\item[(i)] For $0 \leq s \leq 1/2$, we have
\begin{align}
\||H|^{s} f\|_{L^{2}} &\lesssim \|(-\Delta)^{s} f\|_{L^{2}},
\end{align}
and provide $f \in L^2_{ac}$, we have
\begin{align}
\|(-\Delta)^{-s} f\|_{L^{2}} &\lesssim \||H|^{-s} f\|_{L^{2}}.
\end{align}
\item[(ii)] For $0 \leq s \leq 1$, and for $\varphi \in C^\infty(\bR)$ supported on $x \geq c > 0$ with $\|\varphi\|_{L^\infty} \leq 1$,
\begin{align}
\|(-\Delta)^{s} \varphi(H) f\|_{L^{2}} &\lesssim \| |H|^{s} f\|_{L^{2}}. \\
\||H|^{-s} \varphi(H) f\|_{L^{2}} &\lesssim \|(-\Delta)^{-s} f\|_{L^{2}}.
\end{align}
\item[(iii)] For $-1 \leq s \leq 1$ and for $M > -\inf \sigma(H)$,
\begin{align}
\|(H+M)^{s}  f\|_{L^{2}} &\sim \|(-\Delta +M)^{s} f\|_{L^{2}}.
\end{align}
\end{enumerate}
\end{corollary}
\begin{proof}
To prove the first inequality in $(i)$ we use Lemma \ref{lem:norm bound} $(i)$ for $(-\Delta)^{s} f$ with $f \in \mathcal{S}$ as in the proof of $(i)$, while for the second inequality, we have for $f \in \mathcal{C}$ and $g \in \mathcal{S}$ with $\textup{supp }\widehat{g} \subseteq \mathbb{R}^3 \setminus B_r$ that
\begin{align}
|\langle (-\Delta)^{-s} f , g \rangle| = |\langle(-\Delta)^{-s} |H|^s |H|^{-s}  f , g \rangle| \leq \| |H|^{-s} f \|_{L^2} \| |H|^s (-\Delta)^{-s}  g\|_{L^2} \leq \| |H|^{-s} f \|_{L^2}  \|g\|_{L^2}
\end{align}
using Lemma \ref{lem:norm bound} $(i)$, and we use density to conclude. For $(ii)$, we use Lemma \ref{lem:norm bound} $(ii)$ for $|H|^{s} f$ for $f \in \mathcal{C}$ and density, and for the second second inequality, we have for $g \in \mathcal{C}$ and $f \in \mathcal{S}$ with $\textup{supp } \widehat{f} \subseteq \mathbb{R}^3 \setminus B_r$ that
\begin{align}
|\langle |H|^{-s} \varphi(H) f, g \rangle| = |\langle |H|^{-s} \varphi(H) (-\Delta)^s(-\Delta)^{-s}  f, g \rangle| \leq  \|(-\Delta)^{-s}  f\|_{L^2} \|   (-\Delta)^s |H|^{-s} \varphi(H) g \|_{L^2}
\end{align}
and we use Lemma \ref{lem:norm bound} $(ii)$ and density to conclude. Inequality $(iii)$ follows analogously.
\end{proof}

\begin{remark}\label{equ:inhomog_equiv}
Let $\varphi \in C^\infty(\bR)$ be supported on $x \geq c > 0$ with $\|\varphi\|_{L^\infty} \leq 1$. Then for any $M > 0 $, and $s \in \bR$,
\[
\|(H+M)^{s} \varphi(H) f\|_{L^{2}}  \sim_c \||H|^{s}   \varphi(H) f\|_{L^{2}} .
\]
Indeed, for $f \in \mathcal{C}$
\begin{align}
\|(H+M)^{s} \varphi(H) f\|_{L^{2}}   &= \|(H+M)^{s} |H|^s |H|^{-s}\varphi(H) f\|_{L^{2}}   \\
&=\|(H+M)^{s} |H|^s \widetilde{\varphi}(H) |H|^{-s}\varphi(H) f\|_{L^{2}} ,
\end{align}
for $\widetilde{\varphi}$ a slight ``enlargement'' of $\varphi$, and we note that the spectral multiplier
\[
m(\lambda) = \frac{(\lambda^2 + M)^s}{\lambda^{2s}} \widetilde{\varphi}(\lambda)
\]
is bounded from above and below on the support of $\varphi$.
\end{remark}

We will also need some coercivity estimates in weighted spaces. We will argue similarly to the proof of Lemma \ref{lem:norm bound}, using Lemma \ref{lem:weighted_ests} and Lemma \ref{lem:weighted_ests_free}.
\begin{lemma}\label{lem:weighted_coercive}
Let $H = -\Delta + V$ with $V$ a potential satisfying Assumption (A1). Let  $0 < \alpha < 1$ and $-1 \leq s \leq 1$. Let $\varphi \in C^\infty(\bR)$ satisfy the hypotheses of Lemma \ref{lem:weighted_ests}.  Then,
\begin{align}
\|\langle x \rangle^{\alpha} |H|^s \varphi(H) (-\Delta)^{-s} \varphi(-\Delta) \langle x \rangle^{-\alpha} g \|_{L^2_x} \lesssim \|g\|_{L^2_x}.
\end{align}
\end{lemma}
\begin{proof}
For $\varphi$ satisfying the hypotheses of Lemma \ref{lem:weighted_ests}, we in particular have that  $\varphi \in L^\infty$ and $\varphi$ is supported on $x \geq c > 0$. We fix $0< \alpha < 1$ and $0 \leq s \leq  1$. Arguing as in the proof of Lemma~\ref{lem:norm bound}, for $f, g  \in \mathcal{S}$ we set
\[
F(z) = \langle \langle x \rangle^{\alpha} e^{z^2} |H|^z \varphi(H) (-\Delta)^{-z} \varphi(-\Delta)\langle x \rangle^{-\alpha} g , f \rangle,
\]
then $F(z)$ is well-defined since $\langle x \rangle^{\alpha} f \in H^2$ and $(-\Delta)^{-z} \varphi(-\Delta)$ is $L^2$-bounded, and it is continuous on 
\[
S = \{z \,:\, -1 \leq \textup{Re z} \leq 1\} \subseteq \mathbb{C}
\]
and analytic on the interior of $S$.  By Lemma \ref{lem:weighted_ests} applied to
\[
e^{-\gamma^2} |H|^{i\gamma} \varphi(H),
\]
and by Lemma \ref{lem:weighted_ests_free} we can estimate
\begin{align}
|F(i\gamma)| &\lesssim \|\langle x \rangle^{\alpha} e^{-\gamma^2} |H|^{i\gamma} \varphi(H) (-\Delta)^{-i\gamma}\varphi(-\Delta) \langle x \rangle^{-\alpha} g\|_{L^2} \| f\|_{L^{2}}\\
& \lesssim \|\langle x \rangle^{\alpha} (-\Delta)^{-i\gamma} \varphi(-\Delta) \langle x \rangle^{-\alpha}    g\|_{L^{2}} \| f\|_{L^{2}} \\
&\lesssim \|g\|_{L^{2}} \| f\|_{L^{2}}.
\end{align}
Similarly
\begin{align}
|F(1 + i\gamma)| &\lesssim \|\langle x \rangle^{\alpha}  e^{(1 + i\gamma)^2} |H|^{i\gamma} \varphi(H) (-\Delta + V)  (-\Delta)^{-1 } (-\Delta)^{- i\gamma} \varphi(-\Delta) \langle x \rangle^{-\alpha}    g\|_{L^2} \| f\|_{L^{2}}\\
& \lesssim \|\langle x \rangle^{\alpha} (I + V (-\Delta)^{-1})  (-\Delta)^{- i\gamma} \varphi(-\Delta) \langle x \rangle^{-\alpha} g\|_{L^{2}} \| f\|_{L^{2}} \\
& \lesssim \|g\|_{L^{2}} \| f\|_{L^{2}}  + \|\langle x \rangle^{\alpha}V (-\Delta)^{-1}  (-\Delta)^{- i\gamma} \varphi(-\Delta) \langle x \rangle^{-\alpha} g\|_{L^{2}} \| f\|_{L^{2}}.
\end{align}
Using the decay properties of $V$,  we estimate
\begin{align}
\|\langle x \rangle^{\alpha} V (-\Delta)^{-1}  (-\Delta)^{- i\gamma}\varphi(-\Delta)  \langle x \rangle^{-\alpha} g\|_{L^{2}} &\lesssim \| (-\Delta)^{-1}  (-\Delta)^{- i\gamma} \varphi(-\Delta) \langle x \rangle^{-\alpha} g\|_{L_x^{2}}\\
& \lesssim \| g\|_{L^{2}_x}.
\end{align}
This concludes the proof for $s \geq 0$.  For $s \leq 0$, we expand
\begin{align}
|F(-1 + i\gamma)|& = | \langle \langle x \rangle^{\alpha} |H|^{i\gamma} |H|^{-1} \varphi(H)  (-\Delta) (-\Delta)^{- i \gamma} \varphi(-\Delta)   \langle x \rangle^{-\alpha} g , f \rangle| \\
&= | \langle \langle x \rangle^{\alpha} |H|^{i\gamma} |H|^{-1} \varphi(H)  (-\Delta + V - V) (-\Delta)^{- i \gamma} \varphi(-\Delta)   \langle x \rangle^{-\alpha} g , f \rangle| \\
& \leq |\langle\langle x \rangle^{\alpha}|H|^{i\gamma} \varphi(H)     (-\Delta)^{- i \gamma} \varphi(-\Delta)    \langle x \rangle^{-\alpha} g , f \rangle| \\ &\hspace{24mm} +  |\langle \langle x \rangle^{\alpha} |H|^{i\gamma} \varphi(H)   |H|^{-1} V (-\Delta)^{- i \gamma} \varphi(-\Delta)    \langle x \rangle^{-\alpha} g , f \rangle|,
\end{align}
and these terms can be estimated as above, using that $|H|^{i\gamma} \varphi(H)   |H|^{-1} $ is $L^2(\langle x \rangle^{\alpha}dx)$ bounded by Lemma \ref{lem:weighted_ests} and that $V \in L^\infty$. Thus we have shown that for every $R > 0$,
\begin{align}
\|\chi_{|x| < R} \langle x \rangle^{\alpha} e^{z^2} |H|^z \varphi(H) (-\Delta)^{-z} \varphi(-\Delta) \langle x \rangle^{-\alpha} g \|_{L^2_x}.
\end{align}
and the result follows by Fatou's Lemma.
\end{proof}
\begin{corollary}\label{cor:weighted_coercive}
Let $H = -\Delta +V$ with $V$ a potential satisfying Assumption (A1). Let  $0 < \alpha < 1$ and $0 \leq s \leq 1$, and let $\varphi$ satisfy the assumptions of Lemma \ref{lem:weighted_ests}.  Then
\begin{align}
\|\langle x \rangle^{\alpha}|H|^s \varphi(H)\varphi(-\Delta)  f \|_{L^2_x} \lesssim \|\langle x \rangle^{\alpha} (-\Delta)^{s} f\|_{L^2_x}
\end{align}
\end{corollary}
\begin{proof}
Let $f \in \mathcal{S}$, and let 
\[
g =  \langle x \rangle^{\alpha} (-\Delta)^s  \widetilde{\varphi}(-\Delta) f
\]
for a suitable ``enlargement'' $\widetilde{\varphi}$ of $\varphi$.
Applying Lemma \ref{lem:weighted_coercive}, we obtain
\begin{align}
\|\langle x \rangle^{\alpha} |H|^s \varphi(H) \varphi(-\Delta)  f \|_{L^2_x}& \lesssim  \|\langle x \rangle^{\alpha} (-\Delta)^{s} \varphi(-\Delta) f \|_{L^2_x},
\end{align}
and we use Lemma \ref{lem:weighted_ests_free} and density to conclude.
\end{proof}

\section{Randomization Procedure}\label{sec:randomization}
We will randomize our initial data using projections defined via the distorted Fourier transform for the operator $H = -\Delta - 5 \phi^4$, so we consider Schr\"odinger operators $H = -\Delta + V$ with potential $V$ satisfying Assumption (A1). We also assume $H$ has a resonance at zero and a single negative eigenvalue $-\kappa^2$ and associated eigenfunction $Y$ which is smooth and exponentially decaying. We fix a constant $k_0 \geq 1$ to be determined later, see Remark \ref{rem:k_0_choice}.

Let $\psi \in C_0^\infty(\bR)$ with $\psi \equiv 1$ on $[0,1]$ and $\textup{supp}(\psi) \subseteq (-1/2,3/2)$ and for $k \in \mathbb{N}$, and $\xi \in \mathbb{R}^3$ we define
\begin{align}\label{equ:psi_k}
\psi_k(|\xi|) = \psi(|\xi| - k).
\end{align}
Let $\psi_0 \in C_0^\infty(\bR)$ be a function with $\psi_0 \equiv 1$ on $[0, \sqrt{k_0 + 1}]$, and $\textup{supp}(\psi_0) \subseteq (-\varepsilon, \sqrt{k_0 + 2})$ for some $0 < \varepsilon < \kappa^2$, fixed. We may renormalize $\psi_0$ and $\psi$ so that
\[
1 \equiv \psi_0(|\xi|^2) +  \sum_{k \geq k_0} \psi_k(|\xi|), \qquad \textup{for all }\xi \in \mathbb{R}^3.
\]
For $k \geq k_0$, we define the distorted Fourier projections
\begin{equation}\label{pk_def}
\begin{split}
P_{k} &:=  \mathcal{F}_V^{*} \psi_k\bigl(|\cdot| \bigr) \mathcal{F}_V = \psi_k(\sqrt{|H|})P_{ac} =  \psi_k(\sqrt{|H|}) \\
P_0 &:=  \mathcal{F}_V^{*} \psi_0 \bigl( |\cdot|^2 \bigr) \mathcal{F}_V =  \psi_0(H)P_{ac} = \psi_0(H).
\end{split}
\end{equation}
It will be useful in the sequel to note $P_k = P_{k} P_{ac}$,  $P_0 = P_{0} P_{ac}$ and that
\[
\psi_k (\sqrt{|H|}) P_{ac} = \theta_k(H) \qquad \textup{where} \quad \theta_k(x) = \psi_k(\sqrt{|x|}).
\]
We fix notation for a projection to high (distorted) frequencies
\[
P_{\geq k_0} f = \sum_{k \geq k_0} P_k f , \qquad P_{< k_0} = f - P_{\geq k_0} f,
\]
and we observe that $P_{< k_0}$ also contains the projection to the negative eigenspace.

\begin{remark} \label{rem:hs_bds}
For $f \in H^2$ we have $P_{\geq k_0} f, P_k f \in  H^2$, so in particular $P_{\geq k_0} f, P_k f \in  H^s$ for $-1 \leq s \leq 1$. Furthermore, for $0 \leq s \leq 1$, we have by Lemma \ref{lem:norm bound} (ii) and Corollary \ref{cor:equivalence} (i) that
\[
\| |\nabla|^{s} P_{\geq k_0} f \|_{L^2_x} = \| |\nabla|^{s} |H|^{-\frac{s}{2}} P_{\geq k_0} |H|^{\frac{s}{2}} f \|_{L^2_x}  \lesssim \| |H|^{\frac{s}{2}} f \|_{L^2_x} \lesssim \| |\nabla|^s f \|_{L^2_x},
\]
while for $-1 \leq s \leq 0$ we obtain by duality that
\begin{align}
\| \langle \nabla\rangle^{s} P_{\geq k_0} f \|_{L^2_x} = \sup_{ \|g\|_{ H^{-s}} \leq 1}| \langle f, P_{ \geq k_0} g \rangle| \leq\sup_{ \|g\|_{ H^{-s}} \leq 1} \| f\|_{ H^{s}} \| P_{\geq k_0} g\|_{ H^{-s}} \leq \| f\|_{ H^{s}},
\end{align}
by the $0 \leq s \leq 1$ estimates above, with similar bounds holding for $P_k f$. Thus $P_{\geq k_0}$ and $P_k$ extend to bounded operators on $\dot H^s$ for $0 \leq s \leq 1$ and on $H^s$ for $-1 \leq s \leq 0$. Moreover, since
\[
P_0 f =  P_{ac} f - P_{\geq k_0} f ,
\]
we have by the definition of $P_{ac}$ from \eqref{p_ac}, the exponential decay of the negative eigenfunction $Y$, and the above discussion, that $P_0$ is bounded on $\dot H^s$ for $0 \leq s \leq 1$ and $H^s$ for $-1 \leq s \leq 0$. See the discussion at the end of this section for more details.
\end{remark}

We begin by establishing certain $L^p$ and Sobolev bounds for our projection operators.

\begin{lemma} \label{lem:bernstein2}
Let $1 < p \leq \infty$, $1 \leq q \leq \infty$ and $f \in L^2 \cap L^{p,q}$. Then there exists $k_0 > 1$ such that for $k \geq k_0$ and $s \in \bR$, we have
\begin{align}
\||H|^{s} P_k f \|_{L^{p,q}} \leq |k|^{2 s} \|P_k f\|_{L^{p,q}}.
\end{align}
\end{lemma}
\begin{proof}
By real interpolation, it suffices to prove this estimate in $L^p$. We let $\varphi$ be a smooth bump function with $\textup{supp }\varphi \subseteq[- 3, 3] \setminus\{0\}$ and such that
\[
\varphi( H / k^2 ) \psi_k(\sqrt{|H|}) = \psi_k(\sqrt{|H|}) , \qquad \textup{for all }k \geq 5.
\]
To see that this can be done, note that for $k \geq 5$, say, $\psi_k$ is supported on
\[
||\xi| - k| \leq 2 \quad \Longrightarrow \qquad  \frac{1}{2} \leq \frac{|\xi|}{k} \leq \frac{3}{2},
\]
hence we set $\varphi \equiv 1$ on $[-\frac{9}{4},\frac{9}{4}] \setminus [-\frac{1}{4}, \frac{1}{4}]$ with $\varphi \equiv 0$ on $[-\frac{1}{8}, \frac{1}{8}]$. Now we write
\[
 |H|^{s} \varphi( H / k^2  )  =  |k|^{2 s} \varphi( H / k^2  )  (|H|^{s}/ |k|^{2s})
\]
and since $g(x) = |x|^{s} \varphi(x)$ satisfies the assumption of Theorem \ref{thm:jn}, we have
\[
\| |H|^{s} P_k f \|_{L^p} = \|  |H|^{ s} \varphi( H / k^2  ) P_k f \|_{L^p}  \lesssim |k|^{2s} \| P_k f\|_{L^p}. \qedhere
\]
\end{proof}

\begin{lemma} \label{lem:bernstein3}
Let $-1 \leq s \leq 1$ and $f \in H^2$. Then for $k \geq k_0$,
\begin{align}
|k|^{2s} \| P_k f \|_{L^2} \leq  \| |H|^{s} P_k f\|_{L^2}.
\end{align}
\end{lemma}
\begin{proof}
We write
\[
|k|^{2s} \| P_k f \|_{L^2_x} = |k|^{2s} \| \psi_k \mathcal{F}_V f \|_{L^2_\xi} =   \||k|^{2s} |\xi|^{-2s} |\xi|^{2s} \psi_k \mathcal{F}_V f \|_{L^2_\xi}.
\]
Noting that $|k|^{2s} |\xi|^{-2s} \lesssim 1$ on the support of $\psi_k$ for $k \geq 1$, we obtain the desired conclusion by the discussions at the beginning of \S\ref{ssec:dist_four}. 
\end{proof}

Next we record some bounds for the distorted Fourier projection operators. The proofs of Lemma \ref{lem:proj_bds}, Lemma \ref{lem:decay_sf}, and Corollary \ref{cor:weighted_pk} below will be deferred to Appendix \ref{a:kernel}. 

\begin{restatable}{lemma}{projbds}\label{lem:proj_bds}
Let $3/2 < p < 3$ and $f \in L^2 \cap L^p$, radial. Then there exists $C \equiv C(p)$ and $k_0 \geq 1$ so that for all $k \geq k_0$,
\[
\|P_k f \|_{L^p} \leq C \|f\|_{L^p}.
\]
\end{restatable}

This lemma, together with Proposition \ref{prop:bernstein}, imply Bernstein's inequality for the $P_k$ projection operators: 

\begin{proposition}\label{prop:pk_bern}
Let $2 < r < 3$ and $f \in L^{2} \cap L^{\frac{2r}{r-(r-2)\theta}}$, radial. Then for $0 \leq \theta \leq 1$ and any $r \leq q \leq \frac{r}{\theta}$, there exists $C \equiv C(\theta,r, q, V)$  and $k_0 \geq 1$ such that for all $k \geq k_0$
\[
\|P_k f\|_{L^q} \leq C |k|^{2 \left( \frac{r -(r-2)\theta}{2r} - \frac{1}{q}\right) }  \|P_k f\|_{L^{\frac{2r}{r-(r-2)\theta}}}.
\]
\end{proposition}
\begin{proof}
We write $P_k f = \widetilde{P}_k P_k f$ for a slightly enlarged projection $\widetilde{P}_k$ and apply Lemma \ref{lem:proj_bds} and the Bernstein's inequality of Proposition \ref{prop:bernstein}, noting that the size of the support of $\widetilde{\psi}_k \simeq |k|^2$, where $\widetilde{\psi}_k$ is the distorted multiplier corresponding to $\widetilde{P}_k$.
\end{proof}

\begin{lemma}\label{lem:decay_sf}
Let $-1 \leq s \leq 1$,  $0 < \alpha < 1$, and  $f \in H^{s,\alpha}(\bR^3)$, radial, satisfy
\begin{align}\label{hi}
f = \varphi(-\Delta) f
\end{align}
for a function $\varphi \in C^\infty$ supported away from zero. Then there exists $k_0 \geq 1$ so that
\[
\biggl(  \sum_{k \geq k_0}  |k|^{2s} \| \langle x \rangle^{\alpha} P_k f \|^2_{L^2(\bR^3)} \biggr)^{1/2} \lesssim \| \langle x \rangle^{\alpha} |\nabla|^s f\|_{L^2(\bR^3)} .
 \]
\end{lemma}

\begin{corollary}\label{cor:weighted_pk}
Let $0 < \alpha < 1$ and $f \in L^{2,\alpha}$, radial, satisfy
\[
f = \varphi(-\Delta) f
\]
for a function $\varphi$ as in Lemma \ref{lem:decay_sf}. Then there exists $k_0 \geq 1$ so that for $k \geq k_0$, we have
\[
\| \langle x \rangle^{\alpha} P_k f  \|_{L^2_x} \lesssim \|\langle x \rangle^{\alpha} f\|_{L^2_x}.
 \]
\end{corollary}

We will need the next result to establish that the randomization preserves the regularity and weighted estimates of the initial data.
\begin{lemma}\label{lem:weighted_2}
Let $0 < \alpha < 1$, $0 \leq s \leq 1$  and let $f \in  H^{s,\alpha}(\bR^3)$, radial, satisfy
\[
f = \varphi(-\Delta) f
\]
for a function $\varphi$ as in Lemma \ref{lem:decay_sf}. Then there exists $k_0 > 1$ such that 
\[
\|\langle x \rangle^{\alpha} |\nabla|^s P_{\geq k_0} f \|_{L^2(\bR^3)} \lesssim \|\langle x \rangle^{\alpha} |\nabla|^s f \|_{L^2(\bR^3)}.
\]
When $-1 \leq s \leq 0$, we have
\[
\|\langle x \rangle^{\alpha} \langle \nabla\rangle^s P_{\geq k_0} f \|_{L^2(\bR^3)} \lesssim \|\langle x \rangle^{\alpha} \langle \nabla\rangle^s f \|_{L^2(\bR^3)}.
\]
\end{lemma}
\begin{proof}
Let $0 \leq s \leq 1$. Since $P_{\geq k_0} = \varphi(H)$ for some function $\varphi$ satisfying the hypotheses of the lemma, It suffices to show the bound for
\[
 \|\langle x \rangle^{\alpha} |\nabla|^s \varphi(H) f \|_{L^2(\bR^3)}.
\]
We argue by interpolation with the analytic family of operators
 \[
 \langle x\rangle^\alpha e^{z^2} (-\Delta)^z  |H|^{-z} \varphi(H)  \langle x\rangle^{-\alpha}, 
 \]
 and the argument is similar to Lemma \ref{lem:weighted_coercive}, so we only present a sketch. Let $f, g \in \mathcal{S}$, and consider
 \[
 F(z) = \langle  \langle x\rangle^\alpha e^{z^2} (-\Delta)^z  |H|^{-z} \varphi(H)  \langle x\rangle^{-\alpha} f, g \rangle.
 \]
 The $F$ is well-defined since $\langle x \rangle^{\alpha} g \in H^2$ and it is continuous on
 \[
S = \{z \,:\, 0 \leq \textup{Re z} \leq 1\} \subseteq \mathbb{C},
\]
and analytic on the interior of $S$.  When $\textup{Re}(z) = 0$, we have
 \begin{align}
 \|\langle x\rangle^\alpha e^{-\gamma^2} (-\Delta)^{i\gamma}  |H|^{-i\gamma} \varphi(H) \langle x\rangle^{-\alpha} f\|_{L^2_x}  \leq  \| f\|_{L^2_x}
 \end{align}
using first the fact that $\langle x \rangle^{\alpha}$ is an $A_2$ weight in $\bR^3$ (see Definition \ref{def:ap}), the fact that $(-\Delta)^{i\gamma}$ is $A_2$-bounded, see \cite[\S V.4]{Stein_book}, and Lemma \ref{lem:weighted_ests}. When $\textup{Re}(z) = 1$, we estimate
  \begin{align}
 &\|\langle x\rangle^\alpha e^{(1 + i\gamma)^2}(-\Delta)^{i\gamma} (-\Delta) |H|^{-1} |H|^{-i\gamma} \varphi(H)  \langle x\rangle^{-\alpha} f\|_{L^2_x} \\
 & \leq\|\langle x\rangle^\alpha e^{(1 + i\gamma)^2}(-\Delta)^{i\gamma} P_{ac} |H|^{-i\gamma} \varphi(H)  \langle x\rangle^{-\alpha} f\|_{L^2_x}  \\
 & \hspace{44mm}+ \|\langle x\rangle^\alpha e^{(1 + i\gamma)^2}(-\Delta)^{i\gamma}V |H|^{-1} |H|^{-i\gamma} \varphi(H)  \langle x\rangle^{-\alpha} f\|_{L^2_x} ,
 \end{align}
 and we may again use Lemma \ref{lem:weighted_ests} and the fact that $(-\Delta)^{i\gamma}$ is $A_2$-bounded, see again \cite[\S V.4]{Stein_book}.  Thus by ``enlarging'' the multiplier $\varphi(H)$ we have shows that for $0 \leq s \leq 1$ and $0 < \alpha < 1$, we have
 \[
 \| \langle x\rangle^\alpha(-\Delta)^s \varphi(H) f\|_{L^2_x} \lesssim \| \langle x\rangle^{\alpha}  |H|^s \varphi(H) f \|_{L^2_x} = \| \langle x\rangle^{\alpha}  |H|^s \varphi(H) \varphi(-\Delta) f \|_{L^2_x},
 \]
where the last equality follows from the assumption on $f$, and we use Corollary \ref{cor:weighted_coercive}, the three linear theorem and density to conclude. 

For $-1 \leq s \leq 0$, we may similarly argue by interpolation with the analytic family of operators
\[
\langle x\rangle^\alpha e^{z^2} (1 -\Delta)^{z} |H|^{z} \varphi(H)  \langle x\rangle^{-\alpha}.
\]
As above, we define an analytic function $F(z)$ for $f,g$ belonging to a dense class of functions (see Lemma \ref{lem:weighted_coercive} for details). To sketch the $\textup{Re}(z) = -1$ estimate, we have
  \begin{align}
 &\|\langle x\rangle^\alpha e^{(-1 + i\gamma)^2} (I -\Delta)^{i\gamma} (I -\Delta)^{-1} |H|^{1} |H|^{-i\gamma} \varphi(H)  \langle x\rangle^{-\alpha} f\|_{L^2_x} \\
 & \leq\|\langle x\rangle^\alpha e^{(-1 + i\gamma)^2} (I -\Delta)^{i\gamma} (I -\Delta)^{-1} ( -\Delta) P_{ac} |H|^{-i\gamma} \varphi(H)  \langle x\rangle^{-\alpha} f\|_{L^2_x}  \\
 & \hspace{44mm}+ \|\langle x\rangle^\alpha e^{(-1 + i\gamma)^2} (I -\Delta)^{i\gamma}(I- \Delta)^{-1} V |H|^{-i\gamma} \varphi(H)  \langle x\rangle^{-\alpha} f\|_{L^2_x} .
 \end{align}
 To see that $(I-\Delta)^{i\gamma}$ is bounded in the weighted space, $L^2(\langle x \rangle^{\alpha} dx)$, one can argue as in  Lemma~\ref{lem:weighted_ests_free}. Then by Lemma \ref{lem:weighted_ests_free}, the boundedness of $V$ and Lemma \ref{lem:weighted_ests} we obtain
 \begin{align}
\|\langle x\rangle^\alpha e^{(-1 + i\gamma)^2} (-\Delta)^{i\gamma} (1 -\Delta)^{-1} |H|^{1} |H|^{-i\gamma} \varphi(H)  \langle x\rangle^{-\alpha} f\|_{L^2_x}  \lesssim \|f\|_{L^2}.
 \end{align}
 Thus arguing as above, ``enlarging'' the multiplier $\varphi(H)$ we have shown that for $-1 \leq s \leq 0$ and $0 < \alpha < 1$, we have
 \[
 \| \langle x\rangle^\alpha(1 -\Delta)^s \varphi(H) f\|_{L^2_x} \lesssim \| \langle x\rangle^{\alpha}  |H|^s \varphi(H) f \|_{L^2_x} = \| \langle x\rangle^{\alpha}  |H|^s \varphi(H) \varphi(-\Delta) f \|_{L^2_x},
 \]
where the last equality follows from the assumption on $f$. Using Corollary~\ref{cor:weighted_coercive} and  Lemma~\ref{lem:weighted_ests_free} applied to
\[
(1 - \Delta)^{-s} (-\Delta)^s \varphi(-\Delta) 
\]
 we obtain
\[
\| \langle x\rangle^{\alpha}  |H|^s \varphi(H) \varphi(-\Delta) f \|_{L^2_x} \lesssim \| \langle x\rangle^{\alpha}  (-\Delta)^s \varphi(-\Delta) f \|_{L^2_x} \lesssim  \|  \langle x\rangle^{\alpha} (1 - \Delta)^s f \|_{L^2_x},
\]
and we may again use the three lines theorem and density to conclude.
\end{proof}

\begin{remark}\label{rem:k_0_choice}
In the sequel, we will choose $k_0 \geq 1$ in our randomization to be such that the results of this section hold for $k \geq k_0$.
\end{remark}

The choice of a smooth multiplier in the definition of $P_0$ enables us to obtain bounds for the low distorted frequency component of the solution in smoother spaces in light of the results of Jensen and Nakamura in Theorem \ref{thm:jn} and Corollary \ref{cor:low}. The following lemma summarizes this fact. 

\begin{lemma} \label{lem:low_freq_bds}
Let $0 \leq s \leq 1$ and 
\[
f = (f_0, f_1) \in ( \dot H^s \cap |\nabla|^{-s} L^{3/2,1} ) \times (  H^{s-1} \cap \langle \nabla \rangle^{-(s-1)} L^{3/2,1}),
\]
and suppose that  $f_0$ satisfies
\begin{align}\label{hi1}
f_0 = \varphi(-\Delta) f_0
\end{align}
for a function $\varphi \in C^\infty$ supported away from zero. Then we have the following estimate:
\begin{align}
\||\nabla| P_{0} f_0 \|_{L^{3/2,1} \cap L^2} + \|P_{0} f_1 \|_{L^{3/2,1} \cap L^2} \lesssim \| f\|_{\cH^s} + \| |\nabla|^{s}f_0\|_{ L^{3/2,1} } +  \|\langle \nabla\rangle^{s-1} f_1\|_{L^{3/2,1} }.
\end{align}
\end{lemma}
\begin{proof}
We argue for
\[
(f_0, f_1) \in (L^2 \cap \dot H^s \cap |\nabla|^{-s} L^{3/2,1} ) \times L^2 \cap  H^{s-1} \cap \langle \nabla \rangle^{-(s-1)} L^{3/2,1}
\]
and then extend the estimate by density. We begin with the $f_0$ component. By Corollary \ref{cor:low}, and real interpolation we have 
\begin{align}
\||\nabla| P_0 f_0 \|_{L^{3/2,1} \cap L^2}  \lesssim  \|f_0 \|_{L^{3/2,1} \cap L^2},
\end{align}
and given the assumption on the (standard) Fourier support of $f_0$ from \eqref{hi1}, we have
\[
\|f_0 \|_{L^{3/2,1} \cap L^2} \lesssim \||\nabla|^s f_0 \|_{L^{3/2,1} \cap L^2},
\]
as required. For the $f_1$ component, we have by Corollary \ref{cor:low} that
\begin{align}
 \|P_{0} f_1 \|_{L^2}  =  \sup_{\|h\|_{L^2_x} \leq 1} | \langle P_{0} f_1, h \rangle|  =  \sup_{\|h\|_{L^2_x} \leq 1} |\langle f_1, P_{0} h \rangle| \leq \sup_{\|h\|_{L^2_x} \leq 1} \|f_1\|_{ H^{s-1}} \|P_{0} h\|_{ H^{1-s}} \leq \|f_1\|_{ H^{s-1}}
\end{align}
and similarly
\begin{align}
 \|P_{0} f_1 \|_{L^{3/2,1}} =  \sup_{\|h\|_{L^{3,\infty}_x} \leq 1}| \langle P_{0} f_1, h \rangle | \leq   \sup_{\|h\|_{L^{3,\infty}_x} \leq 1} \|f_1\|_{\langle \nabla\rangle^{1-s} L^{3/2,1}} \|\langle \nabla\rangle^{1-s} P_{0} h \|_{L^{3,\infty}} \leq  \|f_1\|_{\langle \nabla\rangle^{1-s} L^{3/2,1}},
 \end{align}
 by Corollary \ref{cor:low} and real interpolation, which concludes the proof.
\end{proof}

We now turn to the definition of our randomization. Fix $f = (f_0, f_1) \in \cH^s(\bR^3)$ which satisfies the orthogonality condition
\[
\langle \kappa f_0 + f_1, Y \rangle = 0.
\]
This condition will be used to ensure the cancellation of an exponentially growing term in the direction of the eigenfunction in the contraction mapping argument, see Section \ref{sec:contraction} for more details.
We write
\[
f = f_{hi} + f_{lo} := (f_0 - \psi_0(-\Delta) f_0, f_1 - \psi_0(-\Delta) f_1) + (\psi_0(-\Delta) f_0, \psi_0(-\Delta) f_1)
\]
for the smooth bump function $\psi_0 \in C_0^\infty$ used to define $P_0$. We note that
\[
f_{lo} \in \dot H^1 \times L^2, \qquad \textup{and} \qquad f_{hi} \in L^2 \times  H^{-1}.
\]
In the sequel, we will consider zero-mean real-valued Gaussian random variables with variance $\sigma$, i.e. random variables $g$ defined on fixed a probability space $(\Omega, \mathcal{F}, \mathbb{P})$ satisfying
\[
\mathbb{P}(g \leq x ) = \frac{1}{\sqrt{2\pi \sigma^2}} \int_{-\infty}^x e^{-t^2/2 \sigma^2} dt.
\]
For simplicity, we will fix $\sigma = 1$.
  
\begin{definition}\label{def:randomization}
Let $\{(g_k, h_k)\}_{k \in \mathbb{N}} $ be a sequence of iid, zero-mean real-valued Gaussian  random variables on a probability space $(\Omega, \cF, \mathbb{P})$. We set
\[
f^\omega = (f_0^\omega, f_1^\omega),
\]
where
\begin{equation}\label{equ:randomization}
\begin{split}
f_0^\omega &:=  \langle f_{0}, Y \rangle Y + (P_{ac} f_{0,lo} + P_0 f_{0,hi}) + \sum_{k \geq k_0} g_k(\omega) P_k f_{0,hi}\\
f_1^\omega &:=  \langle f_1, Y \rangle Y+ (P_{ac} f_{1,lo} + P_0 f_{1,hi})+ \sum_{k \geq k_0} h_k(\omega) P_k f_{1,hi} .
\end{split}
\end{equation}
\end{definition}
One could also randomize with respect to a more general sequence of random variables $\{ (g_k, h_k) \}_{k\in\bZ^4}$ satisfying the following condition: there exists $c > 0$ so that the distributions $\{ (\mu_{k}, \nu_k) \}_{k\in\bZ^4}$ of the random variables $\{ (g_k, h_k) \}_{k\in\bZ^4}$ fulfill
 \begin{equation} \label{equ:rvassumption}
  \left| \int_{-\infty}^{+\infty} e^{\gamma x} \, d\mu_{k}(x) \right| \leq e^{c \gamma^2} \quad \text{for all } \gamma \in \bR \text{ and for all } k \in \bZ^4,
 \end{equation}
 and similarly for the $\nu_k$. The assumption~\eqref{equ:rvassumption} is satisfied, for example, by standard Gaussian random variables, standard Bernoulli random variables, or any random variables with compactly supported distributions, but we will work with Gaussian random variables in the sequel.

To make sense of this definition, we use Remark \ref{rem:hs_bds}, and we regard the infinite series as a Cauchy limit in $L^2(\Omega ; \cH^{s}(\bR^3))$, see Remark \ref{rem:reg} for more details. It will also be useful in the sequel to fix the notation
\[
f_{\geq k_0}^\omega = \biggl(\sum_{k \geq k_0} g_k(\omega) P_k f_{0,hi}, \sum_{k \geq k_0} h_k(\omega) P_k f_{1,hi} \biggr).
\]

\begin{remark}\label{rem:still_orthog}
In particular, we have 
\begin{align}\label{equ:orthog_preserved}
\langle \kappa f_0^\omega + f_1^\omega, Y \rangle = \langle \kappa f_0 + f_1, Y \rangle \langle Y, Y\rangle= 0,
\end{align}
since all the other terms in \eqref{equ:randomization} are orthogonal to the eigenfunction $Y$.
\end{remark}

We will now discuss the regularity properties of the randomized initial data. Some of this discussion will need to be completed in Section~\ref{prob:strichartz} after introducing certain probabilistic estimates, but we will collect the necessary analytic tools here. Using the arguments from \cite[Appendix~B]{BT1}, we may argue that $f^\omega$ is almost surely no more regular than $f$. Indeed, the key fact which is needed in this argument is the equivalence
\begin{equation}\label{equiv}
\begin{split}
\|P_{\geq k_0} h\|_{\dot H^s(\bR^3) }^2 & \simeq  \sum_{k \geq k_0}  \|P_k  P_{\geq k_0}h\|^2_{\dot H^s(\bR^3)}, \quad 0 \leq s \leq 1,\\
\|P_{\geq k_0} h\|_{ H^s(\bR^3) }^2 & \simeq  \sum_{k \geq k_0}  \|P_k  P_{\geq k_0}h\|^2_{H^s(\bR^3)}, \quad -1 \leq s \leq 0,\\
\end{split}
\end{equation}
which will imply the randomization does not regularize $P_{\geq k_0} f$, and which follows from Theorem~\ref{distorted}, the fact that
\[
\sum_{k \geq k_0} |\psi_k(\xi)|^2 \sim 1 \qquad \textup{for } |\xi| \geq  k_0,
\]
and Corollary \ref{cor:equivalence} $(i)$ and $(ii)$ for the first claim, and part $(iii)$ for the second claim. Observing  that by Corollary \ref{cor:low}, Lemma \ref{lem:low_freq_bds}, the definition of $P_{ac}$ (see \eqref{p_ac}) and the fact that $Y$ is smooth and exponentially decaying, we have by Lemma \ref{lem:low_freq_bds} that
\begin{align}
&\langle f_0, Y \rangle Y +  (P_{ac} f_{0,lo} +P_0 f_{0,hi})  \in \dot H^1(\bR^n)\\
 &\langle f_1, Y \rangle Y+ (P_{ac} f_{1,lo} + P_0 f_{1,hi}) \in L^2(\bR^n).
\end{align}
which will in turn imply that our randomization does not regularize the initial data.

To show that the randomization preserves the regularity of the initial data and weighted estimates, we recall that
 \[
\|f\|_{X_s} := \|f\|_{\cH^s} + \|\langle x \rangle^{1-\nu}|\nabla|^{s_1} f_0\|_{L^2_x} + \|\langle x \rangle^{1-\nu}\langle \nabla\rangle^{s_1 -1} f_1\|_{L^2_x} < \infty,
 \]
 and for $f = (f_0, f_1) \in X_s$ we have
 \[
\bigl( \langle f_0, Y \rangle Y,  \langle f_1, Y \rangle Y \bigr) \in X_s,
 \]
 which follows by noting that 
 \[
  \|\langle x \rangle^{1-\nu}|\nabla|^{s_1} Y\|_{L^2_x} \leq    \|\langle x \rangle |\nabla|^{s_1} Y\|_{L^2_x} \lesssim   \| |\nabla|^{s_1} Y\|_{L^2_x} +   \sum_{j} \|x_j |\nabla|^{s_1} Y\|_{L^2_x}
 \]
and by Plancherel's theorem,
\[
\sum_{j} \|x_j |\nabla|^{s_1} Y\|_{L^2_x} = \sum_{j} \|\partial_j |\xi|^{s_1} \widehat{Y}\|_{L^2_x},
\]
and this expression is finite since $\widehat{Y}$ is Schwartz and for $0<s_1 <1$, $\partial_j|\xi|^{s_1}$ is integrable around the origin in three dimensions. A simplified argument works for the inhomogeneous term. Next, we can write
 \[
 P_{ac} f_{0,lo} +P_0 f_{0,hi} =  P_{ac} f_{0,lo} +P_{ac} f_{0,hi} - P_{\geq k_0} f_{0, hi} = P_{ac} f_0 - P_{\geq k_0} f_{0, hi}.
 \]
By \eqref{p_ac} and the observation that $(Y, Y) \in X_s$, the projection $P_{ac}$ is bounded on $X_s$, and by Lemma~\ref{lem:weighted_2} followed by Lemma \ref{lem:weighted_ests_free} we have 
 \begin{align}
 \| P_{\geq k_0} f_{0, hi}\|_{\dot H^s} +   \|\langle x \rangle^{1-\nu}|\nabla|^{s_1} P_{\geq k_0} f_{0, hi}\|_{L^2_x} &\lesssim   \|  f_{0, hi}\|_{\dot H^s} +   \|\langle x \rangle^{1-\nu}|\nabla|^{s_1} f_{0, hi}\|_{L^2_x}  \\
 & \lesssim   \|  f_{0}\|_{\dot H^s} +   \|\langle x \rangle^{1-\nu}|\nabla|^{s_1} f_{0}\|_{L^2_x}  
 \end{align}
Similarly, writing
 \[
 P_{ac} f_{1,lo} +P_0 f_{1,hi} =  P_{ac} f_1 - P_{\geq k_0} f_{1, hi},
 \]
and using again Lemma~\ref{lem:weighted_2} followed by Lemma \ref{lem:weighted_ests_free}, we ultimately obtain
\[
 ( P_{ac} f_{0,lo} +P_0 f_{0,hi} , P_{ac} f_{1,lo} +P_0 f_{1,hi} )\in X_s.
\]
The fact that $f_{\geq k_0}^\omega  \in X_s$ almost surely will follow from the large deviation estimate of Lemma \ref{lem:large_dev} below, as well as Lemma~\ref{lem:decay_sf} and \eqref{equiv}. See Remark~\ref{rem:reg} for more details.

\section{Estimates for the linearized flow} \label{sec:improved_strich}

We begin by recalling the Strichartz estimates for the wave equation. 
\begin{definition}
 We call an exponent pair $(q,r)$ wave-admissible in three dimensions if $2 \leq q \leq \infty$, $2 \leq r < \infty$ and
 \begin{equation}
 \label{equ:wave}
  \frac{1}{q} + \frac{1}{r} \leq \frac{1}{2}.
 \end{equation}
\end{definition}

\begin{proposition}[\protect{Strichartz estimates in three space dimensions, \cite[Corollary 1.3]{Kee-Tao}}] \label{prop:strichartz_estimates}
 Suppose $(q, r)$ and $(\tilde{q}, \tilde{r})$ are wave-admissible pairs. Let $u$ be a (weak) solution to the wave equation on some interval $0 \ni I$ with initial data $(u_0,u_1)$, that is
 \begin{equation*}
  \left\{ \begin{split}
   -u_{tt} + \Delta u \, &= \, h \text{ on } I \times \bR^3, \\
   (u, u_t)|_{t=0} \, &= \, (u_0,u_1)
  \end{split} \right.
 \end{equation*}
 for some function $h$. Then
 \begin{equation} \label{equ:strichartz_estimates}
  \begin{split}
   &\|u\|_{L^q_t L^r_x (I\times \bR^3)} + \|u\|_{L^{\infty}_t \dot{H}^{\gamma}_x(I \times \bR^3)} + \|u_t\|_{L^{\infty}_t \dot{H}^{\gamma-1}_x(I\times \bR^3)} \\
   &\qquad \qquad \lesssim \, \|u_0 \|_{\dot{H}_x^{\gamma}(\bR^3)} + \|u_1\|_{\dot{H}_x^{\gamma-1}(\bR^3)} + \|h\|_{L_t^{\tilde{q}'} L^{\tilde{r}'}_x(I\times \bR^3)}
  \end{split}
 \end{equation}
 under the assumption that the scaling conditions
 \begin{equation}
 \label{equ:strichartz1}
\qquad  \frac{1}{q} + \frac{3}{r} \, = \, \frac{3}{2} - \gamma
 \end{equation}
 and
 \begin{equation}
 \label{equ:strichartz2}
\frac{1}{\tilde{q}'} + \frac{3}{\tilde{r}'} - 2 \, = \, \frac{3}{2} - \gamma
 \end{equation} 
 hold.
\end{proposition}

The following reverse Strichartz estimates for Lorentz spaces are proved in \cite[Theorem 3]{BeGo} for certain Schr\"odinger operators, implying the result for the free half-wave propagators.
\begin{proposition}\label{prop:lorentz_est}
Let $0 \leq \gamma < 3/2$ and let $(q,r)$ be a wave admissible pair satisfying
\[
\frac{1}{q} + \frac{3}{r} = \frac{3}{2} - \gamma.
\]
Then provided
\[
\max\biggl( 6, \frac{6}{3-2\gamma} \biggr) \leq r \leq \frac{3}{\max(1-\gamma, 0)}, \qquad q \neq \infty,
\]
we have
\begin{align}
\|e^{i t |\nabla|} f\|_{L^{r,2}_x L^q_t} &\lesssim \||\nabla|^\gamma f\|_{L^{2}_x}.
\end{align}
\end{proposition}

We also record the following estimates from \cite[Lemma 2.7]{Beceanu}.
\begin{proposition}\label{prop:linfty_est}
The free sine and cosine evolutions satisfy:
\begin{align}
\|\cos(t \sqrt{-\Delta}) f\|_{L^\infty_x L^1_t} &\lesssim \||\nabla| f\|_{L^{3/2,1}_x},\\
\left \|\frac{\sin(t \sqrt{-\Delta}) }{\sqrt{-\Delta}} f\right \|_{L^\infty_x L^1_t} &\lesssim \| f\|_{L^{3/2,1}_x},\\
\|\cos(t \sqrt{-\Delta}) f\|_{L^\infty_x L^2_t} &\lesssim \||\nabla| f\|_{L^{2}_x},\\
\left \|\frac{\sin(t \sqrt{-\Delta}) }{\sqrt{-\Delta}} f\right \|_{L^\infty_x L^2_t} &\lesssim \| f\|_{L^{2}_x}.
\end{align}
\end{proposition}

In certain cases, we will use the radiality of the initial data and appeal to the larger range of Strichartz estimates available in the radial setting. We say that the pair $(q,r)$ is (non-sharp) \emph{radial-admissible} if $2 \leq q \leq \infty$, $2 \leq r < \infty$ and
\begin{equation*}
 \frac{1}{q} + \frac{2}{r} < 1.
\end{equation*}
\begin{proposition} [\protect{Radial Strichartz estimates; \cite{Sterbenz}, \cite{Fang_Wang}}] \label{prop:radial_strichartz_estimates}
 Let $(u_0, u_1) \in \dot{H}^\gamma_x(\bR^3) \times \dot{H}^{\gamma - 1}_x(\bR^3)$ for some $\gamma > 0$ be radially symmetric and let $u$ be a radially symmetric solution to the linear wave equation on some time interval $I \ni 0$ with initial data $(u_0, u_1)$. Suppose $(q,r)$ is a radial-admissible pair satisfying the scaling condition
 \begin{align} \label{equ:scaling}
  \frac{1}{q} + \frac{3}{r} = \frac{3}{2} - \gamma.
 \end{align}
 Then
 \begin{equation*}
  \|u\|_{L^q_t L^r_x(I \times \bR^3)} \lesssim \|u_0\|_{\dot{H}^\gamma_x(\bR^3)} + \|u_1\|_{\dot{H}^{\gamma-1}_x(\bR^3)}.
 \end{equation*}
\end{proposition}

\medskip
In order to establish Strichartz estimates for the linearized flow, and the improved Strichartz estimates for the linearized flow of the random initial data, we will follow the methods of \cite{Beceanu, BeGo} and appeal to the framework of Weiner spaces.

\begin{definition}[\protect{\cite[Definition 4]{Beceanu}}]
For a Banach lattice $X$, let the Wiener space $\mathcal{V}_X$ consist of kernels $T(x,y,t)$ such that for every pair $(x,y)$, $T(x,y,t)$ is a finite measure in $t$ on $\mathbb{R}$ and 
\[
M(T)(x,y) := \int_{\bR} d |T(x,y,t)|
\]
is an $X$-bounded operator.
\end{definition}

The space $\mathcal{V}_X$ is an algebra under composition given by
\begin{align}\label{equ:comp_rule}
(T_1 \circ T_2)(t) := \int_{-\infty}^{\infty} \int_{\mathbb{R}^3}T_1(x,y,s) T_2(y,z,t-s) dy ds,
\end{align}
and occasionally we will abuse notation slightly and write the composition as the product $T_1 T_2$. Elements $T \in \mathcal{V}_X$ have Fourier transforms 
\[
\widehat{T}(x,y,\lambda) := \int_{\mathbb{R}} e^{-it\lambda} dT(x, y, t),
\]
and by the properties of the Fourier transform and the definition of composition by convolution, one can verify that, 
\[
(\widehat{T}_1(\lambda) \widehat{T}_2(\lambda) )^{\vee} = (T_1 \circ T_2),
\]
where
\[
\widehat{T}_1(\lambda) \widehat{T}_2(\lambda) = \int \widehat{T}_1(x, y, \lambda) \widehat{T}_1(z, y, \lambda) dy.
\]

Before stating our next result, we set the notation 
\[
 (\varphi \otimes V \varphi  )f = \langle V \varphi ,   f \rangle \varphi
\]
The framework of Wiener spaces enables Beceanu in \cite{Beceanu} to establish a following representation formula for the linearized flow, see also \cite{KS07}.

\begin{theorem}[\protect{\cite[Proposition 1.6]{Beceanu}}]\label{thm:rep}
Assume that $\langle x \rangle V \in L^{3/2,1}$, and that $H = -\Delta + V$ has a resonance $\varphi$ at zero. Then for $t \geq 0$,
\begin{align}\label{equ:rep_form}
\frac{\sin(t \sqrt{H})}{\sqrt{H}}P_{ac} &= - \frac{4 \pi }{|\langle V , \varphi \rangle|^2}  \varphi \otimes V \varphi \int_0^t \frac{\sin(s \sqrt{-\Delta})}{\sqrt{-\Delta}}  ds + \mathcal{S}(t)\\
\cos(t \sqrt{H}) P_{ac} &= - \frac{4 \pi }{|\langle V , \varphi \rangle|^2}  \varphi \otimes V \varphi \int_0^t \cos(s \sqrt{-\Delta}) ds + \mathcal{C}(t),
\end{align}
where $\mathcal{S}(t)$ and $C(t)$ satisfy the reversed Strichartz estimates
\begin{align}
\|\mathcal{S}(t) f\|_{L^{6,2}_x L^\infty_t \cap L^\infty_x L^2_t (\bR^3)} &\lesssim \|f\|_{L^2(\bR^3)}\\
\|\mathcal{S}(t) f\|_{L^\infty_x L^1_t (\bR^3)} &\lesssim \|f\|_{L^{3/2,1}(\bR^3)}\\
\|\mathcal{C}(t) f\|_{L^{6,2}_x L^\infty_t \cap L^\infty_x L^2_t (\bR^3)} &\lesssim \||\nabla| f\|_{L^2_x(\bR^3)}\\
\|\mathcal{C}(t) f\|_{L^\infty_x L^1_t (\bR^3)} &\lesssim \||\nabla| f\|_{L^{3/2,1}(\bR^3)},
\end{align}
and the estimates
\begin{align}
\biggl\|\int_0^t \mathcal{S}(t-s) F(s) \biggr\|_{L^{6,2}_x L^\infty_t} &\lesssim \|F\|_{L^{6/5}_x L^\infty_t}\\
\biggl\|\int_0^t \mathcal{S}(t-s) F(s) \biggr\|_{L^{\infty}_x L^2_t \cap L^\infty_x L^1_t} &\lesssim \|F\|_{L^{3/2,1}_x L^2_t}\\
\biggl\|\int_0^t \mathcal{C}(t-s) F(s) \biggr\|_{L^{6,2}_x L^\infty_t} &\lesssim \||\nabla|F\|_{L^{6/5}_x L^\infty_t}\\
\biggl\|\int_0^t \mathcal{C}(t-s) F(s) \biggr\|_{L^{\infty}_x L^2_t\cap L^\infty_x L^1_t} &\lesssim \||\nabla|F\|_{L^{3/2,1}_x L^2_t}.
\end{align}
\end{theorem}

\begin{remark}
We note that, in the notation of Section \ref{sec:dist_four}, we have
\begin{align}
\frac{\sin(t \sqrt{H})}{\sqrt{H}}P_{ac} &= \frac{\sin(t \sqrt{|H|})}{\sqrt{|H|}}P_{ac},\\
\cos(t \sqrt{H}) P_{ac} & = \cos(t \sqrt{|H|}) P_{ac},
\end{align}
but we use the notation on the left for the free wave propagators since it is more standard in the literature. Additionally, it will often occur in the sequel that the projection $P_{ac}$ will be implicit in some distorted Fourier multiplier and we will not record its presence explicitly, see for instance Proposition \ref{prop:new_rep} below.
\end{remark}

In order to obtain improved probabilistic estimates, we will need a modified version of the representation formula from \cite{Beceanu}. The key point is that if we stay away from zero in the spectrum of $H$, the resulting linearized flows enjoy a simpler representation formula, from which we may infer many of the same Strichartz estimates as the free evolution. Before stating this result, we record the following lemma from \cite{Beceanu}.

\begin{lemma}[\protect{\cite[Lemma 2.4]{Beceanu}}] \label{lem:bec_24}
Let $T(\lambda) = I + V R_0^+(\lambda^2)$. Assume that $V \in L^{3/2,1}$ and let $\lambda_0 \neq 0$. Consider a smooth bump function $\chi \in C_0^\infty(\bR)$ supported around zero. Then for $0 < \varepsilon \ll 1$ and $R \gg 1$,
\[
(\chi((\lambda - \lambda_0) / \varepsilon) T(\lambda)^{-1})^{\wedge} \in \mathcal{V}_{L^1} \cap \mathcal {V}_{L^{3/2,1}} , \quad ((1 - \chi(\lambda / R)) T(\lambda)^{-1})^{\wedge} \in \mathcal{V}_{L^1} \cap \mathcal {V}_{L^{3/2,1}}.
\]
Consequently, for $\varepsilon, R$ as above, the dual operator satisfies
\[
(\chi((\lambda - \lambda_0) / \varepsilon) T^*(\lambda)^{-1})^{\wedge} \in \mathcal {V}_{L^{3, \infty}}  \cap \mathcal{V}_{L^\infty} , \quad ((1 - \chi(\lambda / R)) T^*(\lambda)^{-1})^{\wedge} \in\mathcal {V}_{L^{3, \infty}}  \cap \mathcal{V}_{L^\infty} .
\]

\end{lemma}

The next theorem is in the spirit of Lemma \ref{lem:bec_24}, modified to include a smooth cut-off function supported away from zero.

\begin{lemma}\label{lem:s_reg}
Suppose that $V \in L^{3/2,1}$. Let $\chi_1$ denote a smooth function supported away from zero. Let
\[
U(\lambda) = \chi_1(\lambda^2)  (I + V R_0^{+}(\lambda^2) )^{-1},
\]
then $\widehat{U} \in \mathcal{V}_{L^1} \cap \mathcal{V}_{L^{3/2,1}}$. Likewise,
 \[
U^*(\lambda) = \chi_1(\lambda^2) (I + R_0^{+}(\lambda^2) V)^{-1},
\]
satisfies $\widehat{U^*} \in  \mathcal{V}_{L^{3,\infty}} \cap \mathcal{V}_{L^\infty}$.
\end{lemma}
\begin{proof}
We prove the result for $U^*$, with an identical proof working for $U$. Let $\chi_0$ be a bump function with $\chi_0 \equiv 1$ for $|x| \leq 1$ and $\chi_0 \equiv 0$ for $|x| \geq 2$. Fix $0 < \varepsilon \ll 1$ sufficiently small so that
\begin{align}\label{sup}
\chi_0(\lambda / \varepsilon)  \chi_1(\lambda^2) = 0.
\end{align}
Let $R$ be as in Lemma \ref{lem:bec_24}. Now for every $\lambda \in [\varepsilon, R]$ let $\varepsilon_\lambda$ be chosen according to the Lemma \ref{lem:bec_24}. Then $[\varepsilon, R]$ is covered by the neighborhoods $I_{\lambda}=(\lambda - \varepsilon_\lambda, \lambda + \varepsilon_\lambda)$. By compactness, there are finitely many such neighborhoods associated to centers $\lambda_1, \ldots, \lambda_N$ which cover $[\varepsilon, R]$. We consider a partition of unity which includes the function $\chi_0(\cdot / \varepsilon)$ and which is subordinate to $(-\varepsilon, \varepsilon)$, the neigborhoods $I_{\lambda_1}, \ldots I_{\lambda_N}$ and the neighborhood $(R, \infty)$, and we write
\begin{align}
1 = \chi_0(\lambda / \varepsilon) + \sum_{k=1}^N \chi_k((\lambda - \lambda_k) / \varepsilon_k) + (1- \chi_\infty(\lambda/R)).
\end{align}
Now,  by \eqref{sup}, we have
\[
\chi_1(\lambda^2)  = \sum_{k=1}^N \chi_k((\lambda - \lambda_k) / \varepsilon_k) \chi_1(\lambda^2)  + (1- \chi_{\infty}(\lambda/R)) \chi_1(\lambda^2) ,
\]
and hence
\[
U^*(\lambda)  = \sum_{k=1}^N \chi_k ((\lambda - \lambda_k) / \varepsilon_k) \chi_1(\lambda^2) (1 + R_0^{+}(\lambda^2) V)^{-1} + (1- \chi_{\infty}(\lambda/R))\chi_1(\lambda^2) (1 + R_0^{+}(\lambda^2) V)^{-1}.
\]
Since
\[
\chi_k((\lambda - \lambda_k) / \varepsilon_k) \chi_1(\lambda^2)  = \chi_k(s) \chi_1((\lambda_k + \varepsilon_k s)^2), \qquad s = \frac{\lambda - \lambda_k}{\varepsilon_k}
\]
is another smooth bump function with support contained in the support of $\chi ( ( \cdot - \lambda_k) / \varepsilon_k)$, by Lemma \ref{lem:bec_24}, since $\lambda_k \neq 0$, we have
\[
(\chi_k((\lambda - \lambda_k) / \varepsilon_k)\chi_1(\lambda^2) (1 + R_0^{+}(\lambda^2) V)^{-1} )^{\wedge} \in \mathcal{V}_{L^{3,\infty}} \cap \mathcal{V}_{L^\infty},
\]
and similarly for
\[
(1- \chi_{\infty}(\lambda/R))\chi_1(\lambda) (1 + R_0^{+}(\lambda^2) V)^{-1},
\]
hence the same is true of the sum, which concludes the proof.
\end{proof}

Before turning to our main result, we record the following.

\begin{lemma}\label{lem:wave_soln}
Let $f \in \mathcal{S}(\mathbb{R})$, then
\begin{align}
W_0^{\pm}(t) f := \frac{1}{\pi i} \int_{-\infty}^\infty e^{\pm it \lambda} \lambda R_0^{+}(\lambda^2)  f 
\end{align}
and
\begin{align}
W_1^{\pm}(t) f := \frac{1}{\pi i} \int_{-\infty}^\infty e^{\pm it \lambda} R_0^{+}(\lambda^2)  f 
\end{align}
are half-wave propagators for the linear wave equation on $\mathbb{R}_t^\pm \times \mathbb{R}^3$ and consequently, satisfy the same Strichartz estimates as the free propagators 
\[
\cos(t \sqrt{-\Delta}) \quad \textup{and}\quad \frac{\sin(t \sqrt{-\Delta})}{\sqrt{-\Delta}},
\]
respectively.
\end{lemma}
\begin{proof}
The proof follows by direct computation.
\end{proof}

We are now prepared to state the modified representation formula which will be used to establish many of our improved Strichartz estimates for the linearized evolution of the random data. 
\begin{proposition}\label{prop:new_rep}
Let $W_0^{\pm}, W_1^{\pm}$ be as defined in Lemma \ref{lem:wave_soln} and let $\chi_1\in C^\infty(\bR) \cap L^\infty(\bR)$ be supported in $\mathbb{R}_+$ away from zero. Then
\begin{equation}\label{equ:x_ops}
\begin{split}
\cos(t \sqrt{H}) \chi_1(H) f & = \int_{-\infty}^{\infty} \widehat{X^+}(t-s) W^+_0(s) fds + \int_{-\infty}^{\infty} \widehat{X^-}(t-s) W^-_0(s) fds \\
\frac{\sin(t \sqrt{H})}{\sqrt{H}}   \chi_1(H) f & =  \frac{1}{i}\int_{-\infty}^{\infty} \widehat{X^+}(t-s) W^+_1(s) fds - \frac{1}{i}\int_{-\infty}^{\infty} \widehat{X^-}(t-s) W^-_1(s) fds
\end{split}
\end{equation}
where $\widehat{X^\pm} \in \mathcal{V}_{L^{3,\infty}} \cap \mathcal{V}_{L^\infty}$.
\end{proposition}

\begin{proof}
We only treat the cosine term, as the sine term is similar. By Stone's formula \eqref{stone}, since $\cos(t \sqrt{H}) \chi_1(H)$ is a bounded multiplier, the propagator is defined using the functional calculus, and is given by
\begin{align}
\cos(t\sqrt{H})P_{ac} \chi_1(H) f  = \int_0^\infty \cos(t\sqrt{\lambda}) \chi_1 (\lambda)[R_V^{+}(\lambda) - R_V^{-}(\lambda)] f d\lambda.
\end{align}
Since $\chi_1$ is supported in $\bR_+$ away from zero, using the change of variables $\lambda \mapsto \lambda^2$, we may write
\begin{align}
\cos(t\sqrt{H}) \chi_1(H)  f = \int_{-\infty}^\infty \cos(t \lambda) \lambda \chi_1(\lambda^2) R_V^+(\lambda^2) f d\lambda,
\end{align}
where we have used that
\[
R_V^+(\lambda^2) = \lim_{\varepsilon \to 0} R_V((\lambda + i \varepsilon)^2),
\]
which agrees with $R_V^{-}(\lambda^2)$ along the negative half-line. By the resolvent identity (see \eqref{res_identity} and comments thereafter)  we have
\[
R_V^{+}(\lambda^2) = (I + R_0^{+}(\lambda^2) V)^{-1} R_0^{+}(\lambda^2),
\]
which yields
\begin{equation} \label{cos_ests}
\begin{split}
\cos(t\sqrt{H}) \chi_1(H)  f &= \int_{-\infty}^\infty \cos(t \lambda)  \chi_1 (\lambda^2)  (I + R_0^{+}(\lambda^2) V)^{-1} \lambda R_0^{+}(\lambda^2)  f d\lambda\\
&= \frac{1}{2} \int_{-\infty}^\infty (e^{it\lambda} + e^{-it\lambda}) \chi_1 (\lambda^2)  (I + R_0^{+}(\lambda^2) V)^{-1} \lambda R_0^{+}(\lambda^2)  f d\lambda.
\end{split}
\end{equation}
This is a symmetrization of the inverse Fourier transform of $\chi_1 (\lambda^2) R_V^{+}(\lambda^2)$. By Lemma \ref{lem:s_reg},
\[
\bigr(S^*(\lambda)\bigl)^{\wedge} = \bigl( \chi_1 (\lambda^2)  (I + R_0^{+}(\lambda^2) V)^{-1} \bigr)^{\wedge} \in \mathcal{V}_{L^{3,\infty}} \cap \mathcal{V}_{L^\infty} .
\]
Hence we obtain the desired result by Lemma \ref{lem:wave_soln}, setting
\[
X^{\pm} = \int_{-\infty}^\infty e^{\pm it\lambda} S^*(\lambda)  d\lambda
\]
and using the composition rule \eqref{equ:comp_rule}.
\end{proof}

\begin{corollary}\label{cor:new_rep2}
Let $\chi_1 \in C^\infty(\bR) \cap L^\infty(\bR)$ be supported in $\mathbb{R}_+$ away from zero. Then we have the representation formula
\begin{equation}\label{equ:x2_ops}
\begin{split}
\frac{\cos(t \sqrt{H})}{\sqrt{H}} \chi_1(H) f & = \int_{-\infty}^\infty \widehat{X^+}(t-s) W^+_1(s) fds + \int_{-\infty}^\infty \widehat{X^-}(t-s) W^-_1(s) fds
\end{split}
\end{equation}
where $\widehat{X^{\pm}} \in \mathcal{V}_{L^{3,\infty}} \cap \mathcal{V}_{L^\infty}$.
\end{corollary}
\begin{proof}
The argument proceeds identically to our estimates in Proposition \ref{prop:new_rep} but the extra factor of $\lambda^{-1}$ leaves us with the operator
\[
R_0^{+}(\lambda^2)  f
\]
instead of $\lambda R_0^{+}(\lambda^2)  f$ in \eqref{cos_ests}.
\end{proof}
We are now prepared to prove the following theorem which enables us to pass from estimates for the standard free evolution to estimates for the linearized evolution. 
\begin{theorem}\label{thm:free_evol_bds}
For any $ 1 \leq r \leq \infty$, and $\chi_1 \in C^\infty(\bR) \cap L^\infty$ supported in $\bR_+$ away from zero we have
\[
\| \cos( \cdot \sqrt{H}) \chi_{1}(H)  f\|_{L_x^{3,\infty} L_t^r \cap L_x^\infty L_t^r} \lesssim \bigl\| W_0^+(\cdot)f\bigr\|_{L_x^{3,\infty} L_t^r \cap L_x^\infty L_t^r} + \bigl\| W_0^-(\cdot)f\bigr\|_{L_x^{3,\infty} L_t^r \cap L_x^\infty L_t^r},
\]
and
\[
\biggl\| \frac{\sin(\cdot \sqrt{H})}{\sqrt{H}} \chi_{1}(H)  f\biggr\|_{L_x^{3,\infty} L_t^r \cap L_x^\infty L_t^r} \lesssim \biggl\|  W_1^+(\cdot) f\biggr\|_{L_x^{3,\infty} L_t^r \cap L_x^\infty L_t^r} + \biggl\|  W_1^-(\cdot) f\biggr\|_{L_x^{3,\infty} L_t^r \cap L_x^\infty L_t^r}.
\]
\end{theorem}
\begin{remark}
We note that for certain exponents $1 \leq r \leq \infty$ the resulting bounds will be outside the range of Strichartz estimates for the free flow.
\end{remark}

\begin{proof}
We only treat the sine term as the cosine term is handled in an identical manner. Using the representation formula \eqref{equ:x_ops}, we write
\begin{align}
 \frac{\sin(t \sqrt{H})}{\sqrt{H}} \chi_1(H)  f &=     \int_{-\infty}^{\infty} \widehat{X^+}(t-s) W^+_1(s) fds + \int_{-\infty}^{\infty} \widehat{X^-}(t-s) W^-_1(s) fds
\end{align}
By Proposition \ref{prop:new_rep}, $\widehat{X^\pm}\in  \mathcal{V}_{L^{3,\infty}} \cap \mathcal{V}_{L^\infty}$, hence by Young's inequality as stated in Lemma \ref{lem:lorentz_young} and the definition of $\mathcal{V}_X$, we can bound
\begin{align}
\biggl\|  \frac{\sin(\cdot \sqrt{H})}{\sqrt{H}} \chi_1(H)  f\biggr\|_{ L_x^{3,\infty} L_t^r \cap L_x^\infty L_t^r } &\lesssim \biggl\| \|\widehat{X^+}\|_{L^1_t}  \biggl\|W_1^+(\cdot)f \biggr\|_{L_t^r } \biggr\|_{L_x^{3,\infty} \cap L_x^\infty }  +  \biggl\| \|\widehat{X^-}\|_{L^1_t}  \biggl\|W_1^-(\cdot)f \biggr\|_{L_t^r } \biggr\|_{L_x^{3,\infty} \cap L_x^\infty } \\
& \lesssim  \biggl\| W_1^+(\cdot)f\biggr\|_{L_x^{3,\infty} L_t^r \cap L_x^\infty L_t^r} + \biggl\| W_1^-(\cdot)f\biggr\|_{L_x^{3,\infty} L_t^r \cap L_x^\infty L_t^r},
\end{align}
which concludes the proof.
\end{proof}

\begin{remark}\label{rem:cos_analogue}
An identical proof yields that under the same assumptions,
\[
\biggl\| \frac{\cos(\cdot \sqrt{H})}{\sqrt{H}} \chi_{1}(H)  f\biggr\|_{L_x^{3,\infty} L_t^r \cap L_x^\infty L_t^r} \lesssim \biggl\|  W_1^+(\cdot) f\biggr\|_{L_x^{3,\infty} L_t^r \cap L_x^\infty L_t^r} + \biggl\|  W_1^-(\cdot) f\biggr\|_{L_x^{3,\infty} L_t^r \cap L_x^\infty L_t^r}.
\]
\end{remark}

\begin{remark}\label{rem:interp_linear_evol}
If $T: L^{3,\infty} \to L^{3,\infty}$ and $T:L^{\infty} \to L^\infty$, then $T: L^{p,q} \to L^{p,q}$ for any $3 < p < \infty$ and $1 \leq q \leq \infty$ by real interpolation (see \cite[Theorem 5.3.2]{Ber-Lof}), hence by Proposition \ref{prop:new_rep}, and the proof of Theorem \ref{thm:free_evol_bds}, we obtain that for any $ 1 \leq r \leq \infty$, $3 < p < \infty$, $1 \leq q \leq \infty$   and $\chi_1 \in C^\infty(\bR) \cap L^\infty(\bR)$ supported away from zero we have
\[
\| \cos( \cdot \sqrt{H}) \chi_{1}(H)  f\|_{L_x^{p,q} L_t^r} \lesssim \bigl\| W_0^+(\cdot)f\bigr\|_{L_x^{p,q} L_t^r } + \bigl\| W_0^-(\cdot)f\bigr\|_{L_x^{p,q} L_t^r},
\]
and similarly for the sine component.
\end{remark}

\begin{corollary}\label{cor:strichartz_interp}
Let $f$ be a radial function satisfying $P_{\geq k_0} f = f$. Then for any $0 \leq \theta \leq 1$ and $r > 3$ it holds that
\begin{align}
\biggl\| \cos (t \sqrt{H}) f\biggr\|_{ L_t^{\frac{r}{\theta}} L_x^{\frac{2r}{r-(r-2)\theta}}} &\lesssim \||\nabla|^{\theta_{r}} f\|_{L^2_x},\\
\biggl\|\frac{ \sin (t \sqrt{H})}{\sqrt{H}} f\biggr\|_{ L_t^{\frac{r}{\theta}} L_x^{\frac{2r}{r-(r-2)\theta}}} &\lesssim \||\nabla|^{\theta_{r} - 1} f\|_{L^2_x},
\end{align}
where
\[
\theta_r = \frac{3}{2} - \frac{\theta}{r} - \frac{3(r-(r-2)\theta)}{2r}.
\]
\end{corollary}

\begin{proof}
Under the above assumptions, $ f= \chi_1(H) f$ for $\varphi \in C^\infty$ supported in $\bR^+$ away from zero. Hence, by the spectral theorem and the coercivity estimates of Corollary \ref{cor:equivalence} we obtain
\begin{align}
\biggl\| \cos (t \sqrt{H}) f\biggr\|_{ L_x^2} & \lesssim \|f\|_{L^2_x}, \qquad \textup{for all } t\\
\biggl\|  \frac{\sin(t \sqrt{H})}{\sqrt{H}}  f\biggr\|_{L_x^2} &\lesssim \|f\|_{\dot H^{-1}_x}, \qquad \textup{for all } t.
\end{align}
Next, we note that by the radial Strichartz estimates of Proposition \ref{prop:radial_strichartz_estimates}, for any $r > 3$, $(r,r)$ is a radial admissible pair at regularity
\[
\gamma = \frac{3}{2} - \frac{4}{r}.
\]
Hence, Remark \ref{rem:interp_linear_evol} yields
\begin{align}
\biggl\| \cos (t \sqrt{H}) f\biggr\|_{L_{t,x}^{r}} &\lesssim  \|f\|_{\dot H^{\frac{3}{2} - \frac{1}{r} - \frac{3}{r}}_x}\\
\biggl\|  \frac{\sin(\cdot \sqrt{H})}{\sqrt{H}}  f\biggr\|_{L_{t,x}^{r}} &\lesssim \|f\|_{\dot H^{\frac{1}{2} - \frac{1}{r} - \frac{3}{r} }_x}.
\end{align}
The result then follows by complex interpolation with the analytic family of operators
\[
\cos (t \sqrt{H})\cdot A e^{-z^2} (-\Delta)^z  \qquad \textup{and} \qquad \frac{\sin(\cdot \sqrt{H})}{\sqrt{H}}\cdot A e^{-z^2} (-\Delta)^z
\]
for some $A > 0$, see the proof of Corollary \ref{cor:low} for a similar argument.
\end{proof}

To establish certain decay estimates for the kernel of the linearized propagator inside the light-cone, we rely on the radiality of the function $f$, and arguments inspired by those in \cite{Schlag07}.  We write
\begin{align}
e^{it\sqrt{H}} \psi_k(\sqrt{|H|}) f &= \int e(x,\xi) \cos(t |\xi|) \psi_k(|\xi|) \mathcal{F}_V f(\xi) d\xi \\
&= \int e(x,\xi) \cos(t |\xi|) \psi_k(|\xi|) \int \overline{e(y,\xi)} f(y) dy d\xi
\end{align}
and we change to polar coordinates, setting
\[
x = r \theta, \quad y = r'\theta', \quad \xi = \rho \omega, \quad \theta, \theta', \omega \in S^2.
\]
Thus
\begin{align}
e^{it\sqrt{H}} \psi_k(\sqrt{|H|}) f &= \int_{\rho} \int_{\omega} e(r \theta, \rho \omega) e^{it \rho} \psi_k(\rho) \int_{r'} \int_{\theta'}  \overline{e(r' \theta',\rho \omega )} f(r') (r')^2  \rho^2dr' \, d\rho \, \sigma(d\omega) \, \sigma(d\theta').
\end{align}
For any $\xi$ with $|\xi| = \rho$, we set
\begin{align}\label{etilde}
\widetilde{e}(r, \rho) = r\rho \int_{S^2} e(r \theta, \rho \omega) \sigma(d\theta).
\end{align}
This expression is well-defined for radial potentials. We remark that in the free case, $e(x,\xi) = e^{ix \cdot \xi}$ while $\widetilde{e}(r,\rho) = \sin(r \rho)$. 

\medskip
We thus obtain that for radial functions $f$, $g$,
\begin{align}
\langle e^{it\sqrt{H}} \psi_k(\sqrt{|H|}) f, g \rangle =   \int_{\rho} \int_{r'} \int_r e^{i t \rho} \psi_k(\rho) \widetilde{e}(r, \rho)   \overline{\widetilde{e}(r',\rho)} r' f(r')  r g(r) dr' \, d\rho dr .
\end{align}
Hence, for $k \geq 1$ we define
\begin{align}\label{equ:linearized_kernel}
K_k(t,r,r') =  \int_{\rho} e^{i t \rho} \psi_k(\rho) \widetilde{e}(r, \rho)   \overline{\widetilde{e}(r',\rho)} d\rho,
\end{align}
which yields
\begin{align}
\langle e^{it\sqrt{H}} \psi_k(\sqrt{|H|}) f, g \rangle = \int_r \int_{r'} K_k(t,r,r') r'f(r') dr' r g(r) dr.
\end{align} 

We record the key kernel estimate in the following proposition. The proof of this estimate is contained in Appendix \ref{a:kernel}.

\begin{restatable}{proposition}{kernelests}
\label{prop:kernel_ests}
Let $(t, r, r') \in \bR_+ \times \bR_+ \times \bR_+$. Then there exists $k_0 \in \bN$ such that for all $k \geq k_0$,  the kernel $K_k(t, r, r')$ defined in \eqref{equ:linearized_kernel} satisfies
\[
|K_k(t,r, r')| \lesssim  \sum_{c_i = \pm 1} \biggl( \frac{\langle r \rangle^{-1} + \langle r' \rangle^{-1}}{|k| \langle t + c_1 (r +  c_2 r') \rangle^3} +  \frac{1}{\langle t + c_1 (r + c_2 r')\rangle^3} \biggr),
\]
where the sum is taken over the four possible combinations of signs and the implicit constant is independent of $k \geq k_0$.
\end{restatable}

\section{Improved probabilistic Strichartz estimates for the linearized flow} \label{prob:strichartz}

We are now prepared to turn to the proof of the improved probabilistic estimates for the linearized flow.  We first need to recall the following moment bounds for Gaussian random variables, see \cite[p148]{Papoulis}.
\begin{lemma}\label{lem:gauss_bds}
Let $2 \leq p < \infty$ and let $g$ be a complex Gaussian random variable. Then
\begin{align}
\|g\|_{L^p_\omega} = \frac{\sqrt{2} \Gamma(\frac{p+1}{2})^{\frac{1}{p} }}{\sqrt{\pi}^{\frac{1}{2p}}}
\end{align}
where $\Gamma(z)$ denotes the usual Gamma function. In particular, by Stirling's formula, there exists $C >0$ such that for all $2 \leq p < \infty$,
\begin{align}
\|g\|_{L^p_\omega} \leq C \sqrt{p}.
\end{align}
\end{lemma}

\begin{lemma}[\protect{\cite[Lemma 3.1]{BT1}}] \label{lem:large_dev}
 Let $\{g_n\}_{n=1}^{\infty}$ be a sequence of complex-valued independent identically distributed (iid) mean-zero Gaussian random variables on a probability space $(\Omega, {\mathcal A}, \bP)$. Then there exists $C > 0$ such that for every $p \geq 2$ and every $\{c_n\}_{n=1}^{\infty} \in \ell^2(\bN; \bC)$, we have
 \begin{equation*}
  \Bigl\| \sum_{n=1}^{\infty} c_n g_n(\omega) \Bigr\|_{L^p_\omega} \leq C \sqrt{p} \Bigl( \sum_{n=1}^{\infty} |c_n|^2 \Bigr)^{1/2}.
 \end{equation*}
\end{lemma}

We will also use the following variant of \cite[Lemma 4.5]{Tz10} to bound the probability of certain subsets of the probability space.
\begin{lemma}\label{lem:large_dev2}
Let $F$ be a real valued measurable function on a probability space $(\Omega, \mathcal{F}, \bP)$. Suppose that there exists $\alpha > 0$, $N > 0$, $k\in \bN \setminus \{0\}$ and $C > 0$ such that for every $p\geq p_0$ one has
\begin{align}
\|F\|_{L^p_\omega} \leq C N^{-\alpha} p^{\frac{k}{2}}
\end{align}
Then, there exists $C_1$, and $\delta$ depending on $C$ and $p_0$ such that for $\lambda > 0$
\begin{align}
\bP( \omega \in \Omega : |F(\omega)| > \lambda) \leq C_1 e^{-\delta N^{\frac{2\alpha}{k}} \lambda^{\frac{2}{k}}}.
\end{align}
\end{lemma}

\begin{remark}\label{rem:reg}
As a consequence of Lemma \ref{lem:large_dev} and the discussion at the end of Section~\ref{sec:randomization}, the sequence
\[
\biggl(\sum_{k_0 \leq k \leq N} g_k(\omega) P_k f_{0, hi},\sum_{k_0 \leq k \leq N} h_k(\omega)P_k f_{1, hi} \biggr) 
\]
is Cauchy in $L^2(\Omega ; \cH^s)$, since for any $N, M > k_0$, using the independence of the Gaussians and \eqref{equiv}, we have
\[
\biggl\| \biggl(\sum_{M \leq k \leq N} g_k(\omega) P_k f_{0, hi},\sum_{M \leq k \leq N} h_k(\omega) P_k f_{1, hi} \biggr) \biggr\|_{L^2_\omega \cH^s} \lesssim \biggl \| \biggl(\sum_{M \leq k \leq N} P_k f_{0, hi},\sum_{M \leq k \leq N} P_k f_{1, hi} \biggr) \biggr\|_{\cH^s}
\]
 Furthermore, by Lemma \ref{lem:large_dev}, Lemma \ref{lem:large_dev2}, and the discussion at the end of Section~\ref{sec:randomization} may conclude that for $(f_0, f_1) \in X_s$, its randomization $f^\omega$ lies in $X_s$ almost surely. Thus, we restrict ourselves to the subset of $\Omega$ of full measure where this holds. 
\end{remark}

We will only establish $L^p_{\omega}$ bounds for the linearized evolution of the random data  since the probabilistic bounds then follow by Lemma \ref{lem:large_dev2}. Moreover, we will take advantage of Lemma \ref{lem:large_dev2} and only establish $L^p_\omega$ estimates for $p$ larger than some $p_0 > 2$. For simplicity of exposition, we make the following definition: 

\medskip
\noindent \textbf{Assumption (A2):}
We say $f \in \cH^s(\bR^3)$ satisfies Assumption (A2) if $f = (f_{0,hi}, f_{1, hi})$.

\medskip
\noindent We note that in this case, 
\begin{align}
f_{\geq k_0}^\omega  = \biggl( \sum_{k \geq k_0} g_k(\omega) P_k f_{0} , \sum_{k \geq k_0} h_k(\omega) P_k f_{1} \biggr).
\end{align}
Moreover, we note that for such functions we have
\begin{align}\label{weighted_sobolev_bds}
\| \langle x \rangle^{\alpha} |\nabla|^s f \|_{L^2} \lesssim \| \langle x \rangle^{\alpha} \langle \nabla\rangle^s f \|_{L^2}
\end{align}
for $s \in \bR$ by applying Lemma \ref{lem:weighted_ests_free} to 
\[
 \langle \nabla \rangle^{-s}  |\nabla|^{s}\varphi(-\Delta) 
\]
for $\varphi \in C^\infty \cap L^\infty$ supported away from zero.

 We collect our improved probabilistic estimates in several propositions. We begin with the more direct estimates, where we will appeal mainly to the Bernstein estimates for the projection operators $P_k$, see Proposition \ref{prop:pk_bern}. We recall the definition of the linearized wave evolution:
\[
W(t)(f_0, f_1) = \cos(t \sqrt{|H|}) f_0 + \frac{\sin(t \sqrt{|H|}) }{\sqrt{|H|}} f_1.
\]

\begin{proposition}\label{improved_strichartz1}
Let $s > 3/4$, let $f \in \cH^s(\bR^3)$ satisfy Assumption (A2) and let $f^\omega$ be its randomization as defined in Definition \ref{def:randomization}. Then for $p \geq 6$, we have
\begin{align}
\|\|W(t)f_{\geq k_0}^\omega  \|_{L^6_x L^\infty_t (\bR \times \bR^3)} \|_{L^p_\omega}  \lesssim \sqrt{p} \|P_{\geq k_0}  f\|_{ \cH^s_x(\bR^3)}.
\end{align}
\end{proposition}

\begin{proof}
Fix $s > 3/4$. Fix $M > \max(1, -\inf \sigma(H))$. By Sobolev embedding in time, we have
\begin{align}
\|\|W(t) f_{\geq k_0}^\omega  \|_{L^6_x L^\infty_t } \|_{L^p_\omega}  \lesssim \|\|(M +  |\partial_t |^2 )^{\frac{1}{12}+} W(t) f_{\geq k_0}^\omega \|_{L^6_x L^6_t } \|_{L^p_\omega} .
\end{align}
We now estimate
\[
\cos(t \sqrt{H}) \sum_{k \geq k_0} g_k(\omega)  P_k f_0 = \sum_{k \geq k_0} g_k(\omega)\cos(t \sqrt{H})   P_k f_0.
\]
Applying the large deviation estimate of Lemma \ref{lem:large_dev} and Minkowski's inequality (using that $p \geq 6$) we obtain
\begin{align}
&\biggl\| \biggl\|(M +  |\partial_t |^2 )^{\frac{1}{12}+}  \cos(t \sqrt{H}) \sum_{k \geq k_0} g_k(\omega) P_k f_0 \biggr\|_{L^6_x L^6_t } \biggr\|_{L^p_\omega}  \\
&\lesssim \sqrt{p} \biggl( \sum_{k \geq k_0} \| (M +  |\partial_t |^2 )^{\frac{1}{12}+} \cos(t \sqrt{H})  P_k f_0  \|_{L^6_tL^6_x }^2 \biggr)^{1/2}.
\end{align}
Fix $\varepsilon > 0$ to be determined. By Proposition \ref{h_sobolev}, we have
\begin{align}\label{equ:h_sobolev}
\|f\|_{L^6_x} \leq \|(H+M)^{\beta} f \|_{L^{6-\varepsilon}_x}, \qquad \textup{for } \beta > \frac{3}{2} \left( \frac{1}{6-\varepsilon} - \frac{1}{6}\right)
\end{align}
hence taking $\beta = \varepsilon /2$, we can estimate
\begin{equation}\label{l6_est}
\begin{split}
&\sqrt{p} \biggl( \sum_{k \geq k_0} \| (M +  |\partial_t |^2 )^{\frac{1}{12}+}\cos(t \sqrt{H})   P_k f_0  \|_{L^6_tL^6_x }^2 \biggr)^{1/2}\\
&\lesssim \sqrt{p} \biggl( \sum_{k \geq k_0} \| (M +  |\partial_t |^2 )^{\frac{1}{12}+}  \cos(t \sqrt{H})  (H+M)^{\frac{\varepsilon}{2}} P_k f_0  \|_{L^6_tL^{6-\varepsilon}_x }^2 \biggr)^{1/2}.
\end{split}
\end{equation}
We now revisit the proof of Proposition \ref{prop:new_rep} to handle the time derivative. By \eqref{cos_ests}, we have
\[
\cos(t \sqrt{H}) \chi_1(H) f =\frac{1}{2} \int_{-\infty}^\infty (e^{it\lambda} + e^{-it\lambda}) \chi_1 (\lambda^2)  (I + R_0^{+}(\lambda^2) V)^{-1} \lambda R_0^{+}(\lambda^2)  f d\lambda,
\]
hence 
\begin{align}
&(M + |\partial_t|^2)^{\alpha/2}  \cos(t \sqrt{H}) \chi_1(H) f \\
&=\frac{1}{2} \int_{-\infty}^\infty (M + \lambda^2)^{\alpha/2} (e^{it\lambda} + e^{-it\lambda}) \chi_1 (\lambda^2)  (I + R_0^{+}(\lambda^2) V)^{-1} \lambda R_0^{+}(\lambda^2)  f d\lambda,
\end{align}
which by Stone's formula \eqref{stone} yields
\[
(M +  |\partial_t |^2 )^{\alpha}  \cos(t \sqrt{H}) \chi_1(H) f  =  \cos(t \sqrt{H}) (M + H)^{\alpha/2} \chi_1(H) f .
\]
Hence, we obtain
\begin{align}
&\sqrt{p} \biggl( \sum_{k \geq k_0} \| (M +  |\partial_t |^2 )^{\frac{1}{12}+}  \cos(t \sqrt{H})  (H+M)^{\frac{\varepsilon}{2}} P_k f_0  \|_{L^6_tL^{6-\varepsilon}_x }^2 \biggr)^{1/2}\\
& \lesssim  \sqrt{p} \biggl( \sum_{k \geq k_0} \|  \cos(t \sqrt{H})  (H+M)^{\frac{\varepsilon}{2} +\frac{1}{12}+} P_k f_0  \|_{L^6_tL^{6-\varepsilon}_x }^2 \biggr)^{1/2}.
\end{align}
Fix $2 < r < 3$, and $0 \leq \theta \leq 1$ satisfying
\begin{align}\label{theta1}
6 - \varepsilon \leq \frac{r}{\theta} \quad \Longrightarrow \quad \theta \leq \frac{r}{6-\varepsilon}.
\end{align}
Using that $\cos(t \sqrt{H})( H+M)^{\alpha}$ acts diagonally on distorted Fourier space and hence doesn't change the support of $P_k f$, by the Bernstein estimates of Proposition \ref{prop:pk_bern}, we obtain
\begin{align}
 &\sqrt{p} \biggl( \sum_{k \geq k_0} \|  \cos(t \sqrt{H})  (H+M)^{\frac{\varepsilon}{2} + \frac{1}{12} + } P_k f_0  \|_{L^6_tL^{6-\varepsilon}_x }^2 \biggr)^{1/2} \label{starting_pt}\\
 & \lesssim \sqrt{p} \biggl( \sum_{k \geq k_0} |k|^{4 \left(\frac{r-(r-2)\theta}{2r} - \frac{1}{6 - \varepsilon} \right)} \| \cos(t \sqrt{H})  (H+M)^{\frac{\varepsilon}{2} + \frac{1}{12} + }  P_k f_0  \|_{L^6_t L^{\frac{2r}{r-(r-2)\theta}}_x}^2 \biggr)^{1/2}.
\end{align}
In order to apply Corollary \ref{cor:strichartz_interp}, we need to find $3 < r_1 \leq \infty$ and $0 \leq \theta_1 \leq 1$ such that
\[
\frac{2r}{r-(r-2)\theta} = \frac{2r_1}{r_1-(r_1-2)\theta_1}, \qquad \frac{r_1}{\theta_1} = 6.
\]
We find that we need 
\begin{align}\label{theta_eq}
 \theta = \frac{r}{6} \left( \frac{r_1 - 2}{r-2}\right)
\end{align}
by solving using the restrictions on $r_1, \theta_1$, and hence from \eqref{theta1} and \eqref{theta_eq} we need to choose $\varepsilon >0$ sufficiently small and $r < 3 < r_1$ so that
\[
 \frac{r}{6} \left( \frac{r_1 - 2}{r-2}\right) < \frac{r}{6-\varepsilon}.
\]
Under these constraints we obtain by the Strichartz estimates of Corollary \ref{cor:strichartz_interp} and  the coercivity results of Corollary \ref{cor:equivalence}, followed by Remark \ref{equ:inhomog_equiv}, Lemma \ref{lem:bernstein2} and Lemma \ref{lem:bernstein3} that
\begin{align}
&\sqrt{p} \biggl( \sum_{k \geq k_0 } |k|^{4\left(\frac{r-(r-2)\theta}{2r} - \frac{1}{6 - \varepsilon} \right) } \| \sqrt{|H|}^{ \frac{3}{2} - \frac{1}{6} -\frac{3(r-(r-2)\theta)}{2r} }  (H+M)^{\frac{\varepsilon}{2} + \frac{1}{12} + }  P_k f_0  \|_{L^2_x}^2 \biggr)^{1/2}\\
& \lesssim \sqrt{p} \| \sqrt{|H|}^{\frac{1 + \theta}{2} + \varepsilon + } P_{\geq k_0} f_0  \|_{L^2_x},
\end{align}
where we have used Theorem~\ref{distorted} and the fact that
\[
\sum_{k \geq k_0} |\psi_k(\xi)|^2 \simeq \biggl( \sum_{k \geq k_0} \psi_k \biggr)^2.
\]
By Corollary \ref{cor:equivalence},  we obtain the desired estimate provided
\[
\frac{1 + \theta}{2} + \varepsilon < s,
\]
which may be satisfied for any $s > \frac{3}{4}$ by choosing $r = 3-$, $\varepsilon$ sufficiently small and $\theta < \frac{r}{6-\varepsilon}$. Note that for the $f_1$ component, the same proof will yield
\begin{align}
\biggl\| \biggl\| \sum_{k \geq k_0} g_k(\omega) \frac{\sin(t \sqrt{H})}{\sqrt{H}} P_k f_1 \biggr\|_{L^6_x  L^\infty_t} \biggr\|_{L^p_\omega} 
& \lesssim \biggl( \sum_{k \geq k_0}  \| \sqrt{|H|}^{\frac{\theta + 1 }{2} + \varepsilon - 1 + } P_k f_1 \|^2_{L^2_x} \biggr)^{1/2}\\ &\simeq \|P_{\geq k_0}  f_1 \|_{\dot H_x^{\frac{\theta -1}{2} + \varepsilon+}},
\end{align}
which concludes the proof.
\end{proof}

We need global bounds on the $L^{8}_{t,x}$ Strichartz norm of the solution to establish scattering. We present this result now. 
\begin{proposition}\label{prop:l8}
Let $s > 3/4$ let $f \in \cH^s(\bR^3)$ satisfy Assumption (A2) and let $f^\omega$ be its randomization as defined in Definition \ref{def:randomization}. Then for $p \geq 6$, we have
\begin{align}
\| \|  W(t) f_{\geq k_0} ^\omega\|_{L^{8}_{t,x}(\bR \times \bR^3)} \|_{L^p_\omega} &\lesssim \sqrt{p} \|P_{\geq k_0}  f\|_{\cH^{s}}.
\end{align}
\end{proposition}
\begin{proof}
We use Sobolev embedding in space and time to obtain
\[
\| W(t)  f_{\geq k_0} ^\omega \|_{L^{8}_{t,x}} \lesssim  \|  (|\partial_t| + M)^{\left( \frac{1}{6} - \frac{1}{8}\right)+ } (H+M)^{\frac{3}{2}\left( \frac{1}{6} - \frac{1}{8}\right)+}W(t) f_{\geq k_0}^\omega \|_{L^{6}_{t,x}}.
\]
Using the large deviation estimate of Lemma \ref{lem:large_dev} and Minkowski's inequality (using that $p \geq 6$) we may bound the $L^p_\omega$ norm of the $f_0$ contribution to the previous expression by
\[
\sqrt{p}\biggl( \sum_{k \geq k_0} \|  (|\partial_t| + M)^{\left( \frac{1}{6} - \frac{1}{8}\right)+ } (H+M)^{\frac{3}{2}\left( \frac{1}{6} - \frac{1}{8}\right)+} \cos(t \sqrt{H})   P_k f_0 \|^2_{L^{6}_{t,x}} \biggr)^{1/2}
\]
and we may argue as in  the proof of Proposition \ref{improved_strichartz1}, beginning from \eqref{l6_est}, to obtain the result.
\end{proof}

\begin{proposition}\label{prop:improved_strichartz2}
Let $s \geq 5/6$, let $f \in \cH^s(\bR^3)$ satisfy Assumption (A2) and let $f^\omega$ be its randomization as defined in Definition \ref{def:randomization}. Then for $p \geq 9$, we have
\begin{align}
\| \|  W(t) f_{\geq k_0} ^\omega\|_{L^{9,2}_x L^3_t  (\bR \times \bR^3)} \|_{L^p_\omega} &\lesssim \sqrt{p} \|P_{\geq k_0}  f\|_{\cH_x^{s}(\bR^3)},
\end{align}
and for $p \geq 36/5$, we have
\begin{align}
\|\|W(t) f_{\geq k_0} ^\omega \|_{L^{36/5}_x L^\infty_t (\bR \times \bR^3)} \|_{L^p_\omega}  \lesssim \sqrt{p} \|P_{\geq k_0}  f\|_{ \cH^s_x(\bR^3)}.
\end{align}
Furthermore, for any $s > 2/3$ and $p\geq 36/5$, we have
\begin{align}
\| \|  W(t)f_{\geq k_0} ^\omega\|_{L^{36/5}_x L^6_t  (\bR \times \bR^3)} \|_{L^p_\omega} &\lesssim \sqrt{p} \|P_{\geq k_0}  f\|_{\cH^{s}_x(\bR^3)}.
\end{align}
\end{proposition}
\begin{proof}
We only estimate the frequency localized pieces for the first and third estimates since the rest of the argument follows identically to that of Proposition \ref{improved_strichartz1}. For the first estimate, we have  by Remark \ref{rem:interp_linear_evol} that for $k\geq k_0$
\[
\| \cos(t \sqrt{H})  P_k f_0 \|_{L^{9,2}_x L^3_t } \lesssim \| W_0^{\pm}(t)  P_k f_0\|_{L^{9,2}_x L^3_t } \lesssim |k|^{\frac{5}{6}} \| P_k f_0\|_{L^{2}_{x}} ,
\]
where the last inequality follows from the Strichartz estimates of Proposition \ref{prop:lorentz_est} for the admissible pair $(r,q) = (9,3)$ at regularity
\[
\frac{3}{2} - \frac{1}{3} - \frac{3}{9} = \frac{5}{6}.
\]
using Corollary \ref{cor:equivalence} and Lemma \ref{lem:bernstein2} The sine term may be estimated similarly. 

For the $L^{36/5}_x L^6_t$ norm, we use Minkowski's inequality and the Sobolev embedding estimate of Proposition \ref{h_sobolev} to obtain
\[
\| \cos(t \sqrt{H})  P_k f_0 \|_{L^{36/5}_x L^6_t } \leq \| \cos(t \sqrt{H})  P_k f_0 \|_{ L^6_t L^{36/5}_x} \lesssim  \| (H + M)^{\beta} \cos(t \sqrt{H})  P_k f_0 \|_{ L^6_t L^6_x},
\]
for any
\[
\beta > \frac{3}{2} \left( \frac{1}{6} - \frac{5}{36}\right).
\]
By the arguments used to prove Proposition \ref{improved_strichartz1} (cf. \eqref{starting_pt}) we then obtain
\[
 \| (H+M)^{\beta}  \cos(t \sqrt{H})  P_k f_0 \|_{ L^6_t L^6_x} \leq  \|\sqrt{|H|}^{\frac{1 + \theta}{2}  - \frac{1}{12}+  \varepsilon }  P_k f_0 \|_{L^2_x},
\]
for $\theta$ close to $1/2$ and $\varepsilon > 0$ a small constant. This concludes the proof.
\end{proof}

The following estimate will be essential in handling a linear term arising in the contraction mapping scheme. This estimate is the most delicate and requires a careful analysis based on the kernel estimates of Proposition \ref{prop:kernel_ests}.

\begin{proposition} \label{prop:delicate}
Fix parameters $s_1, \nu > 0$ satisfying $s_1 > 3\nu>0$. Let $f \in \cH^{s}(\bR^3)$ satisfy Assumption (A2) and suppose
\[
\|\langle x\rangle^{1-\nu} |\nabla|^{s_1}  f_0\|_{L^2_x} + \|\langle x\rangle^{1-\nu} \langle \nabla\rangle^{s_1-1}  f_1\|_{L^2_x} < \infty,
\]
and let $f^\omega$ be its randomization as defined in Definition \ref{def:randomization}. Then for $p \geq 2$ and $\theta > 1/2$, we have
\begin{align}
&\| \| \langle x \rangle^{-1 - \theta} W(t) f_{\geq k_0}^\omega \|_{L^{2}_x L^1_t} \|_{L^p_\omega} \\
&\lesssim \sqrt{p} \left(  \| P_{\geq k_0} f\|_{\cH^s} +\|\langle x\rangle^{1-\nu} |\nabla|^{s_1}  f_0\|_{L^2_x} + \|\langle x\rangle^{1-\nu} \langle\nabla\rangle^{s_1-1}  f_1\|_{L^2_x} \right).
\end{align}
\end{proposition}

\begin{proof}
It suffices to consider $t \geq 1$, since otherwise we apply Minkowski's inequality and H\"older's inequality in time and space, together with $L^2$ bounds for $W(t)$ to estimate
\begin{align}
\| W(t) f_{\geq k_0} ^\omega \|_{L^2_x L^{1}_t([0,1] \times \bR^3)} &\leq \| W(t) f_{\geq k_0} ^\omega \|_{ L^{1}_tL^2_x([0,1] \times \bR^3) } \\
&\leq \| W(t) f_{\geq k_0} ^\omega \|_{ L^{\infty}_tL^2_x([0,1] \times \bR^3)} \\
&\lesssim \|f_{\geq k_0} ^\omega \|_{L^2 \times \dot H^{-1}}.
\end{align}
Observing that
\begin{align}\label{b1}
\|f_{\geq k_0} ^\omega \|_{L^2 \times \dot H^{-1}} \lesssim  \|f_{\geq k_0} ^\omega \|_{\cH^s_x}
\end{align}
by Theorem \ref{distorted}, Corollary \ref{cor:equivalence} and the fact that $f_{\geq k_0} ^\omega$ is supported away from zero in distorted Fourier space, we note that for $p \geq 2$,
\begin{align}\label{b2}
\|\|f_{\geq k_0} ^\omega \|_{\cH^s_x}\|_{L^p_\omega} \lesssim  \sqrt{p} \|P_{\geq k_0} f \|_{\cH^s_x},
\end{align}
by Lemma \ref{lem:large_dev}, and thus this handles the contribution for $t \in[0,1]$. We set  
\[
\Lambda = \left\{ |x| \leq \frac{t}{4} \right\},
\]
and expand
\begin{align}
 \| \langle x \rangle^{-1-\theta} &W(t)  f_{\geq k_0} ^\omega \|_{L^{2}_x L^1_t}  \\
 &\leq  \| (1 - \chi_\Lambda) \langle x \rangle^{-1-\theta} W(t) f_{\geq k_0}^\omega \|_{L^{2}_x L^1_t}  +  \| \chi_{\Lambda} \langle x \rangle^{-1-\theta} W(t)   f_{\geq k_0}^\omega \|_{L^{2}_x L^1_t} \\
 & :=\textup{Term I} + \textup{Term II}.
\end{align}
To estimate Term I, we use Minkowski's inequality, H\"older's inequality, and \eqref{b1} to obtain
\begin{align}
\| (1 - \chi_\Lambda) \langle x \rangle^{-1-\theta} W(t)   f_{\geq k_0}^\omega\|_{L^{2}_x L^1_t}  
 & \leq\| \| (1 - \chi_\Lambda)  \langle x \rangle^{-1-\theta }\|_{L^\infty_x}  \| W(t)  f_{\geq k_0}^\omega \|_{L^2_x} \|_{L^1_t}\\
 & \leq \|  \langle t \rangle^{-1-\theta}  \| W(t)   f_{\geq k_0}^\omega  \|_{L^2_x} \|_{L^1_t}\\
  & \leq \|  \langle t \rangle^{-1-\theta} \|_{L^1_t} \| \| W(t)    f_{\geq k_0}^\omega  \|_{L^2_x} \|_{L^\infty_t}\\
  & \lesssim  \|  f_{\geq k_0}^\omega  \|_{L^2 \times \dot H^{-1}}\\
  & \lesssim \| f_{\geq k_0}^\omega\|_{\cH^s_x},
  \end{align}
and the $L^p_\omega$ estimate follows as in \eqref{b2}. 

Now we turn to Term II. We will handle only the cosine term as the sine term follows similarly. Up to ``enlarging'' the multiplier $\psi_k$ associated with the projection $P_k$, we may write
\[
\widetilde{P}_k = \widetilde{\psi}_k(\sqrt{|H|}) P_k ,
\]
and we drop the tilde notation for convenience. We have
\begin{align}
&\biggl \| \chi_{\Lambda} \langle x \rangle^{-1-\theta} \cos(t\sqrt{H})  \sum_{k \geq k_0} g_k(\omega)  P_k f_0   \biggr\|_{L^{2}_x L^1_t}\\
&= \biggl\|\sum_{k \geq k_0} g_k(\omega) \chi_{\Lambda} \langle x \rangle^{-1-\theta} \cos(t\sqrt{H}) P_k f_0  \biggr\|_{L^{2}_x L^1_t}\\
&= \biggl\|\sum_{k \geq k_0} g_k(\omega) \chi_{\Lambda} \langle x \rangle^{-1-\theta} \cos(t\sqrt{H}) \psi_k(\sqrt{|H|}) P_k f_0   \biggr\|_{L^{2}_x L^1_t}\\
& \leq \biggl \|\sum_{k \geq k_0} g_k(\omega) \chi_{\Lambda} \langle x \rangle^{-1-\theta} \cos(t\sqrt{H})\psi_k(\sqrt{|H|}) ((1 - \chi_\Lambda) P_k f_0  )  \biggr\|_{L^{2}_x L^1_t}\\
& \hspace{24mm} + \biggl \|\sum_{k \geq k_0} g_k(\omega) \chi_{\Lambda} \langle x \rangle^{-1-\theta} \cos(t\sqrt{H}) \psi_k(\sqrt{|H|}) ( \chi_\Lambda P_k f_0  )  \biggr\|_{L^{2}_x L^1_t} \label{termiib}\\
& := \textup{Term IIa} + \textup{Term IIb}
\end{align}
For Term IIa, we let
\[
\Lambda_{k} = \bigl\{ |x| \leq |k|^\beta\bigr\}
\]
for some $\beta > 0$ to be determined, and we expand further
\begin{align}
& \biggl \|\sum_{k \geq k_0} g_k(\omega) \chi_{\Lambda} \langle x \rangle^{-1-\theta} \cos(t\sqrt{H}) \psi_k(\sqrt{|H|}) ((1 - \chi_\Lambda) P_k f_0)  \biggr\|_{L^{2}_x L^1_t}\\
 & \leq  \biggl \|\sum_{k \geq k_0} g_k(\omega) \chi_{\Lambda} \langle x \rangle^{-1-\theta} \cos(t\sqrt{H}) \psi_k(\sqrt{|H|})( (1 - \chi_\Lambda) \chi_{\Lambda_k} P_k f_0)  \biggr\|_{L^{2}_x L^1_t} \\
 & \hspace{24mm}+  \biggl \|\sum_{k \geq k_0} g_k(\omega) \chi_{\Lambda} \langle x \rangle^{-1-\theta} \cos(t\sqrt{H})\psi_k(\sqrt{|H|})( (1 - \chi_\Lambda)(1- \chi_{\Lambda_k})P_k f_0)  \biggr\|_{L^{2}_x L^1_t}\\
 & := \textup{Term IIa.1} + \textup{Term IIa.2}.
\end{align}
For Term IIa.1, we let $p \geq2$ and by  Minkowski's inequality and the large deviation estimate of Lemma \ref{lem:large_dev}, we have
\begin{align}
&\biggl \| \biggl \|\sum_{k \geq k_0} g_k(\omega) \chi_{\Lambda} \langle x \rangle^{-1-\theta} \cos(t\sqrt{H}) \psi_k(\sqrt{|H|}) ( (1 - \chi_\Lambda) \chi_{\Lambda_k}P_k f_0)  \biggr\|_{L^{2}_x L^1_t} \biggl \|_{L^p_\omega} \\
& \lesssim \sqrt{p}  \biggr\| \biggl(\sum_{k \geq k_0} \| \chi_{\Lambda} \langle x \rangle^{-1-\theta} \cos(t\sqrt{H}) \psi_k(\sqrt{|H|})((1 - \chi_\Lambda) \chi_{\Lambda_k}P_k f_0 ) \|_{L^{2}_x}^2 \biggr)^{1/2}   \biggr\|_{ L^1_t}. \label{whattobd0}
\end{align}
We consider the terms in the sum. We use the $L^2$-boundedness of the linearized wave evolution and of the Fourier multiplier $\psi_k(\sqrt{|H|})$ to estimate
\begin{align}
 &\| \chi_{\Lambda} \langle x \rangle^{-1-\theta} \cos(t\sqrt{H}) \psi_k(\sqrt{|H|}) ((1 - \chi_\Lambda) \chi_{\Lambda_k}P_k f_0 ) \|_{L^{2}_x}\\
 & \leq  \|  \cos(t\sqrt{H}) \psi_k(\sqrt{|H|}) ((1 - \chi_\Lambda) \chi_{\Lambda_k} P_k f_0 ) \|_{L^{2}_x}\\
 & \lesssim  \|  (1 - \chi_\Lambda) \chi_{\Lambda_k} P_k f_0  \|_{L^{2}_x}.  \phantom{\int}
 \end{align}
By the definition of $\Lambda_k$ and $\Lambda$, we can bound this expression by
 \begin{align}
&  \frac{1}{|t| (\log t)^2} \| |x| (\log |x| )^2 (1 - \chi_\Lambda) \chi_{\Lambda_k} P_k f_0 \|_{L^{2}_x}\\
& \lesssim \frac{(\log k^\beta)^2 |k|^{\beta \nu} }{|t| (\log t)^2} \| |x|^{1-\nu} P_k f_0  \|_{L^{2}_x}.
\end{align}
We substitute this into \eqref{whattobd0}. Using Lemma \ref{lem:decay_sf} and integrating $|t|^{-1} (\log t)^2$ over $t \geq 1$, we obtain
\begin{align}
 &\sqrt{p}  \biggr\| \biggl(\sum_{k \geq k_0} \| \chi_{\Lambda} \langle x \rangle^{-1-\theta} \cos(t\sqrt{H}) \psi_k(\sqrt{|H|})((1 - \chi_\Lambda) \chi_{\Lambda_k}P_k f_0 ) \|_{L^{2}_x}^2 \biggr)^{1/2}   \biggr\|_{ L^1_t} \\
 & \lesssim  \sqrt{p}  \biggr\| \biggl(\sum_{k \geq k_0} \frac{(\log k^\beta)^4 |k|^{2\beta \nu} }{|t|^2 (\log t)^4} \| |x|^{1-\nu} P_k f_0  \|_{L^{2}_x}^2 \biggr)^{1/2}   \biggr\|_{ L^1_t} \\
 & \lesssim  \| \langle x \rangle^{1- \nu}   |\nabla|^{\nu \beta+} P_{\geq k_0}\, f\|_{L^2(\bR^3)}.
\end{align}
We note that we are using that $f$ satisfies Assumption (A2) in order to verify the condition \eqref{hi} of Lemma \ref{lem:decay_sf}. 

For Term IIa.2, we have
\begin{align}
&\biggl \| \biggl \|\sum_{k \geq k_0} g_k(\omega) \chi_{\Lambda} \langle x \rangle^{-1-\theta} \cos(t\sqrt{H}) \psi_k(\sqrt{|H|})(  (1 - \chi_\Lambda)  (1 - \chi_{\Lambda_k})P_k f_0 )  \biggr\|_{L^{2}_x L^1_t} \biggl \|_{L^p_\omega}\\
& \lesssim  \biggl\| \biggl \|\sum_{k \geq k_0} g_k(\omega) \chi_{\Lambda} \langle x \rangle^{-1-\theta} \cos(t\sqrt{H}) \psi_k(\sqrt{|H|})( (1 - \chi_{\Lambda_k})P_k f_0 )  \biggr\|_{L^{2}_x L^1_t} \biggl \|_{L^p_\omega}\\
& \hspace{14mm} +  \biggl\| \biggl \|\sum_{k \geq k_0} g_k(\omega) \chi_{\Lambda} \langle x \rangle^{-1-\theta} \cos(t\sqrt{H}) \psi_k(\sqrt{|H|})( \chi_\Lambda  (1 - \chi_{\Lambda_k})P_k f_0 )  \biggr\|_{L^{2}_x L^1_t} \biggl \|_{L^p_\omega}\\
& =: \textup{IIa.2.1} + \textup{IIa.2.2}.
\end{align}
We first address Term IIa.2.1. For $p\geq2$, by Minkowski's inequality and Lemma \ref{lem:gauss_bds} we obtain
\begin{align}
&\biggl \| \biggl \|\sum_{k \geq k_0} g_k(\omega) \chi_{\Lambda} \langle x \rangle^{-1-\theta} \cos(t\sqrt{H}) \psi_k(\sqrt{|H|})(  (1 - \chi_{\Lambda_k})P_k f_0 )  \biggr\|_{L^{2}_x L^1_t} \biggl \|_{L^p_\omega}\\
& \leq \sum_{k \geq k_0}  \| g_k\|_{L^p_\omega} \| \chi_{\Lambda} \langle x \rangle^{-1-\theta} \cos(t\sqrt{H}) \psi_k(\sqrt{|H|}) (  (1 - \chi_{\Lambda_k})P_k f_0 )  \|_{L^{2}_x L^1_t}  \\
& \lesssim \sqrt{p} \sum_{k \geq k_0 }  \| \chi_{\Lambda} \langle x \rangle^{-1-\theta} \cos(t\sqrt{H}) \psi_k(\sqrt{|H|}) (  (1 - \chi_{\Lambda_k})P_k f_0 )  \|_{L^{2}_x L^1_t} . \label{equ:whattobd}
\end{align}
Now, since $\theta > 1/2$, we use H\"older's inequality to estimate
\begin{align}
&\| \chi_{\Lambda} \langle x \rangle^{-1-\theta} \cos(t\sqrt{H})\psi_k(\sqrt{|H|})((1 - \chi_{\Lambda_k})P_k f_0 )  \|_{L^{2}_x L^1_t}\\
& \lesssim \|\cos(t\sqrt{H}) \psi_k(\sqrt{|H|}) ( (1 - \chi_{\Lambda_k})P_k f_0 )  \|_{L^{\infty}_x L^1_t}.\phantom{\int}
\end{align}
Since $\psi_k(\cdot)$ is supported away from zero, we can write
\begin{align}
& \| \cos(t \sqrt{H}) \psi_k(\sqrt{|H|}) (  (1 - \chi_{\Lambda_k})P_k f_0  )  \|_{L^\infty_x L^1_t}\\
& =  \biggl\|  \frac{\cos(t \sqrt{H})}{\sqrt{H}}  \sqrt{|H|}\psi_k(\sqrt{|H|}) (   (1 - \chi_{\Lambda_k})P_k f_0  ) \biggr\|_{L^\infty_x L^1_t},
\end{align}
and using Corollary \ref{cor:new_rep2} and Proposition \ref{prop:linfty_est}, Lemma \ref{lem:bernstein2}, and H\"older's inequality for Lorentz spaces (Lemma \ref{lem:lorentz_holder}), we have
\begin{align}
&\biggl\| \frac{\cos(t \sqrt{H})}{\sqrt{H}} \sqrt{|H|} \psi_k(\sqrt{|H|})(  (1 - \chi_{\Lambda_k}) P_k f_0  )\biggr\|_{L^\infty_x L^1_t}\\
& \lesssim   \| \sqrt{|H|}   \psi_k(\sqrt{|H|}) \bigl(   (1 - \chi_{\Lambda_k})P_k f_0  \bigr)\|_{L^{3/2,1}_x}\\
& \lesssim |k|  \| \psi_k(\sqrt{|H|}) \bigl( (1 - \chi_{\Lambda_k}) P_k f_0 \bigr)\|_{L^{3/2,1}_x}\\
& \lesssim |k|  \| \langle x \rangle^{\frac{1}{2}+} \psi_k(\sqrt{|H|}) \bigl(   (1 - \chi_{\Lambda_k}) P_k f_0 \bigr)\|_{L^{2}_x}.
\end{align}
By Corollary \ref{cor:weighted_pk} and the definition of the $\Lambda_k$ we obtain
\begin{align}
&|k|  \| \langle x \rangle^{\frac{1}{2}+} \psi_k(\sqrt{|H|}) \bigl(  (1 - \chi_{\Lambda_k})P_k f_0  \bigr)\|_{L^{2}_x} \\
& \lesssim  |k|  \| \langle x \rangle^{\frac{1}{2}+} \bigl( (1 - \chi_{\Lambda_k}) P_k f_0 \bigr)\|_{L^{2}_x}\\
&  \lesssim |k|^{1 - (\frac{1}{2} - \nu)\beta+}  \| \langle x\rangle ^{1-\nu} P_k f_0  \|_{L^{2}_x}.
\end{align}
Now, inserting this bound into \eqref{equ:whattobd}, and applying Cauchy-Schwarz and Lemma \ref{lem:decay_sf}, we obtain
\begin{align}
 & \sqrt{p} \sum_{k \geq k_0} |k|^{1 - \frac{1}{2}\beta+} |k|^{\nu \beta +} \| \langle x\rangle ^{1-\nu} P_k f_0 \|_{L^2_x}   \\
 & \leq \sqrt{p} \left( \sum_{k\geq k_0} \frac{1}{|k|^{\beta - 2-} } \right)^{1/2} \left( \sum_{k} |k|^{2 \nu \beta +}  \| \langle x\rangle ^{1-\nu} P_k f_0 \|^2_{L^2_x} \right)^{1/2}\\
 & \lesssim \sqrt{p} \left( \sum_{k \geq k_0} \frac{1}{|k|^{\beta - 2-} } \right)^{1/2}  \| \langle x \rangle^{1- \nu}   |\nabla|^{\nu \beta+}  f_0\|_{L^2(\bR^3)},
\end{align}
which yields the desired bound provided $\beta > 3$.  This leads to the conditions $s_1 > 3 \nu > 0$.

Finally, we are left to estimate Term IIa.2.2 and Term IIb. We only estimate the latter, as the former may be treated similarly. For $p \geq 2$, we use Minkowski's inequality and the large deviation estimate of Lemma \ref{lem:large_dev}   to estimate
\begin{align}
 &\biggl \|\sum_{k \geq k_0} g_k(\omega) \chi_{\Lambda} \langle x \rangle^{-1-\theta} \cos(t\sqrt{H}) \psi_k(\sqrt{|H|}) (\chi_\Lambda P_k f_0 )  \biggr\|_{L^{2}_x L^1_t}\\
& \lesssim \sqrt{p} \biggl \| \biggl ( \sum_{k \geq 1}\| \chi_{\Lambda} \langle x \rangle^{-1-\theta} \cos(t\sqrt{H}) \psi_k(\sqrt{|H|}) ( \chi_\Lambda P_k f_0 )\|_{L^2_x}^2 \biggr)^{1/2}  \biggr\|_{ L^1_t}.
\end{align}
By triangle inequality, it suffices to estimate the half-wave propagators $e^{\pm it\sqrt{|H|}}$, and we consider only $e^{it\sqrt{|H|}}$. Here we will use radiality and the kernel estimates of Proposition \ref{prop:kernel_ests}.  We define
\[
 \widetilde{P_k  f_0}(r) = r  P_k  f_0(r).
\]
By duality we need to estimate
\begin{align}
&|\langle \chi_{\Lambda} P_k  f_0,  e^{-it \sqrt{H}}\psi_k(\sqrt{|H|})P_{ac} \langle x \rangle^{-1-\theta} \chi_{\Lambda} g \rangle |\\
& =: \biggl| \iint \widetilde{P_k  f_0}(r)  \overline{ \bigl(\chi_{\Lambda} (r)  K_k(t, r, r')\chi_{\Lambda}(r')\bigr) \langle r' \rangle^{-1-\theta}  r' g(r') } dr' dr \biggr|\\
& \lesssim   \iint |\widetilde{P_k f_0}(r)|  | \bigl(\chi_{\Lambda} (r)  K_k(t, r, r')\chi_{\Lambda}(r')\bigr)| \langle r' \rangle^{-1-\theta} |\widetilde{g}(r')| dr' dr
\end{align}
where
\[
K_k(t,r,r') =  \int_{\rho} e^{i t \rho} \psi_k(\rho) \widetilde{e}(r, \rho)   \overline{\widetilde{e}(r',\rho)} d\rho.
\]
By Proposition \ref{prop:kernel_ests}, for $r, r' \in \Lambda$, we have established that for $k\geq k_0 $,
\begin{align}
|K_k(t,r,r')| \lesssim  \frac{1}{\langle t \rangle^3}
\end{align}
Hence, we can estimate
\begin{align}
& \iint |\widetilde{P_k f_0}(r)| | \bigl(\chi_{\Lambda} (r)  K_k(t, r, r')\chi_{\Lambda}(r')\bigr)| \langle r' \rangle^{-1-\theta} |\widetilde{g}(r')| dr' dr \\
 & \lesssim \frac{1}{\langle t \rangle^3}  \iint |\widetilde{P_k f_0}(r)|  \chi_{\Lambda} (r)  \chi_{\Lambda} (r')   \langle r' \rangle^{-1-\theta} |\widetilde{g}(r')| dr' dr.
\end{align}
From here, we use Cauchy-Schwarz inequality in $r$ and $r'$ which yields a bound of
\begin{align}
& \frac{1}{\langle t \rangle^3}\left(  \iint  \chi_{\Lambda} (r)  \chi_{\Lambda} (r')   \langle r' \rangle^{-2-2\theta}  dr' dr \right)^{1/2} \|P_k f_0\|_{L^2_x}\|g\|_{L^2_x} \lesssim \frac{1}{\langle t \rangle^3} |t|^{1/2}\| P_k f_0\|_{L^2_x},
\end{align}
and we may integrate in time to obtain an estimate for \eqref{termiib}. The term corresponding to the sine evolution is bounded similarly.

Putting everything together, we have proved that for $p \geq 2$,
\begin{align}
 &\bigl\| \| \langle x \rangle^{-1-\theta} W(t)   f_{\geq k_0}^\omega \|_{L^{2}_x L^1_t}  \bigr\|_{L^p_\omega} \\
 & \lesssim \sqrt{p} \bigl( \| f  \|_{\cH^s} + \|\langle x\rangle^{1-\nu} |\nabla|^{s_1}  f_0\|_{L^2_x}+\|\langle x\rangle^{1-\nu} |\nabla|^{s_1-1}  f_1\|_{L^2_x} \bigr),
\end{align}
provided $s_1 > 3 \nu > 0$ and we use \eqref{weighted_sobolev_bds} to conclude the proof.
\end{proof}

\begin{remark}\label{rem:delicate}
In the estimate of Term IIa.2.1 in the previous proposition, we do not use probabilistic improvements. However, since this term is not the worst contribution to the estimates, by using probabilistic arguments for the other terms, one is nonetheless able to obtain an improvement over the required regularity for deterministic data in the estimate as a whole.
\end{remark}

In light of the proof of the previous proposition, we obtain the following estimate for the weighted $L^2_{t,x}$ norm of the linearized evolution of the random data. The condition for $\nu$ in the following theorem is weaker than the condition from Proposition~\ref{prop:delicate} since the $L^2_t$ norm enables us to use the standard large deviation estimates in the proof.

\begin{corollary}
Fix $\theta > 1/2$ and parameters $s_1, \nu > 0$ satisfying $s_1 > 2 \nu > 0$. Let  $f \in \cH^{s}(\bR^3) $ satisfy Assumption (A2) and suppose
\[
\|\langle x\rangle^{1-\nu} |\nabla|^{s_1}  f_0\|_{L^2_x} + \|\langle x\rangle^{1-\nu} \langle \nabla\rangle^{s_1-1}  f_1\|_{L^2_x} < \infty,
\]
and let $f^\omega$ be its randomization as defined in Definition \ref{def:randomization}. Then
\begin{align}
&\| \| \langle x \rangle^{-1 - \theta} W(t)  f_{\geq k_0}^\omega \|_{L^{2}_x L^2_t} \|_{L^p_\omega} \\
&\lesssim \sqrt{p} \left(  \| P_{\geq k_0}f\|_{\cH^s} +\|\langle x\rangle^{1-\nu} |\nabla|^{s_1}  f_0\|_{L^2_x}+\|\langle x\rangle^{1-\nu} \langle \nabla\rangle^{s_1-1}  f_1\|_{L^2_x} \right).
\end{align}
\end{corollary}
\begin{proof}
The proof of the previous proposition goes through identically except for the term
\begin{align}
&\biggl \| \biggl \|\sum_{k \geq k_0} g_k(\omega) \chi_{\Lambda} \langle x \rangle^{-1-\theta} \cos(t\sqrt{H}) \psi_k(\sqrt{|H|})P_{ac} ( (1 - \chi_{\Lambda_k})P_k f_0 )  \biggr\|_{L^{2}_x L^2_t} \biggl \|_{L^p_\omega}.
\end{align}
Here, we use the fact that we are in $L^2_x$ to use the large deviation estimate and Minkowski's inequality to estimate this term by
\begin{align}
& \sqrt{p} \left( \sum_{k \geq k_0}  \| \chi_{\Lambda} \langle x \rangle^{-1-\theta} \cos(t\sqrt{H}) \psi_k(\sqrt{|H|})P_{ac} (  (1 - \chi_{\Lambda_k})P_k f_0 )  \|_{L^{2}_x L^2_t}^2 \right)^{1/2}  \\
&\leq \sqrt{p} \left( \sum_{k \geq k_0}  \| \cos(t\sqrt{H}) \psi_k(\sqrt{|H|})P_{ac} (  (1 - \chi_{\Lambda_k})P_k f_0 )  \|_{L^{\infty}_x L^2_t}^2 \right)^{1/2} . \label{equ:whattobd2}
\end{align}
By Theorem \ref{thm:free_evol_bds}, Remark \ref{rem:cos_analogue} and Proposition \ref{prop:linfty_est}, together with Corollary \ref{cor:equivalence} we obtain
\begin{align}
&  \| \cos(t\sqrt{H}) \psi_k(\sqrt{|H|})P_{ac} (  (1 - \chi_{\Lambda_k})P_k f_0 )  \|_{L^{\infty}_x L^2_t}\\
&  \bigl \| \frac{\cos(t\sqrt{H})}{\sqrt{H}} \sqrt{H} \psi_k(\sqrt{|H|})P_{ac} ( (1 - \chi_{\Lambda_k})P_k f_0 ) \bigr \|_{L^{\infty}_x L^2_t}\\
& \lesssim |k| \| (1 - \chi_{\Lambda_k}) P_k f_0 \|_{L^2_x} \\
& \lesssim |k|^{1 - (1- \nu)\beta}  \| |x|^{1-\nu} P_k f_0  \|_{L^{2}_x}
\end{align}
Now, inserting this bound into the above expression and applying Cauchy-Schwarz yields
\begin{align}
 C_p \sum_{k \geq k_0} |k|^{1 - \beta}& |k|^{\nu \beta } \| |x|^{1-\nu} P_k f_0 \|_{L^2_x}   \\
 & \leq C_p \left( \sum_k \frac{1}{|k|^{2\beta - 2} } \right)^{1/2} \left( \sum_{k} |k|^{\nu \beta}  \| |x|^{1-\nu} P_k f_0 \|^2_{L^2_x} \right)^{1/2},
\end{align}
which is finite by Lemma \ref{lem:decay_sf} provided $\beta > 3/2$. The sine term is handled similarly.
\end{proof}

Finally, we record, in the next corollary, the bounds for the linearized evolution of the random data we will need to carry out the contraction argument. For $I \subseteq \mathbb{R}$, we fix $\theta > 1/2$ and define a Banach space $Z(I)$ with norm
\begin{equation} \label{equ:y_def}
\begin{split}
\|u\|_{Z(I)} &= \|u\|_{L^6_x L^\infty_t \cap L^{36/5}_x L^\infty_t(I \times \mathbb{R}^3)}  + \|u\|_{L^{36/5}_x L^6_t (I \times \mathbb{R}^3)} + \|u\|_{L^{9,2}_x L^3_t (I \times \mathbb{R}^3)}  \\
& \hspace{25mm} + \| \langle x \rangle^{-1 - \theta} u \|_{L^{2}_x L^1_t \cap L^{2}_x L^2_t (I \times \mathbb{R}^3)} ,
\end{split}
\end{equation}
and we let $Z \equiv Z(\bR_+)$. We also recall that the Banach space $X_s$ is defined according to the norm
\[
\| f\|_{X_s} = \|f\|_{\cH^s} + \|\langle x \rangle^{1-\nu}|\nabla|^{s_1} f_0\|_{L^2_x} + \|\langle x \rangle^{1-\nu}\langle \nabla\rangle^{s_1-1} f_1\|_{L^2_x}.
\]
We obtain the following result.
\begin{corollary}[Almost sure bounds for linearized evolution]  \label{cor:as}
Let $f \in X_s$ for $s > 5/6$  satisfy Assumption (A2) and consider its randomization $f^\omega$ defined as in Definition~\ref{def:randomization}. Then if $s_1 > 3\nu > 0$, there exists $C, c > 0$ such that
\[
\mathbb{P} \left( \|W(t)  f_{\geq k_0}^\omega \|_{Z}  > \lambda \right) < C e^{-c \lambda^2 / \| f \|^2_{ X_s} }.
\]
\end{corollary}
\begin{proof}
For functions $f \in \cH^s(\bR^3)$, we have 
\[
\| P_{\geq k_0} f \|_{\cH^s} \lesssim  \|f \|_{\cH^s}
\]
by Remark \ref{rem:hs_bds}. The previous propositions yield that for $p \geq 9$, and $\beta > 3$,
\begin{align}
&\| \|W(t)  f_{\geq k_0}^\omega \|_{Z} \|_{L^p_\omega} \\
&\lesssim  \sqrt{p} \bigl( \| P_{\geq k_0} f  \|_{\cH^s} + \| \langle x \rangle^{1- \nu}   \, |\nabla|^{\nu \beta+} f_{0}\|_{L^2(\bR^3)} +  \| \langle x \rangle^{1- \nu}   \, \langle \nabla\rangle^{\nu \beta-1+} f_{1}\|_{L^2(\bR^3)} \bigr).
\end{align}
By Assumption (A2), we have
\[
\| P_{\geq k_0} f  \|_{\cH^s} \lesssim \| f  \|_{\cH^s},
\]
thus for all $p \geq 9$ we have
\[
\| \|W(t)  f_{\geq k_0}^\omega \|_{Z} \|_{L^p_\omega}  \lesssim \sqrt{p} \| f \|_{X_s},
\]
and the result then follows using Lemma \ref{lem:large_dev2}.
\end{proof}

\section{Contraction mapping argument}\label{sec:contraction}
\subsection{Set-up}
We are studying solutions of the equation
\begin{equation} \label{equ:ivp_ansatz_sec5}
 \left\{ \begin{aligned}
  -\partial_t^2 u + H u &= - \partial_t^2 \phi_{a(t)} + (V - V_{a(t)}) u + N(u, \phi_{a(t)}) , \\
  (u, \partial_t u)|_{t=0} &=  (\psi_0 - \phi , \psi_1 - \dot a(0) \partial_{a} \phi_{a(0)}) + f^\omega, \quad  a(0) = 1
 \end{aligned} \right.,
\end{equation}
where $N(x,y)$ is defined in \eqref{equ:nonlin}, $(\psi_0, \psi_1) \in \cN_1$ for $\cN_1$ defined in \eqref{n1}, $f^\omega$ is the randomization of a function $f \in \cM_s$ according to Definition~\ref{def:randomization} and $\partial_{a} \phi_{a(0)} = \varphi_{a(0)}$ is the resonance. We will decompose the solution as
\begin{align}
u = P_p u + P_{ac} u = (2\kappa)^{-1/2} ( x_+(t) + x_-(t)) Y + P_{ac} u.
\end{align}
 We will need to set up several equations, one for the coefficients of the components of $(u, u_t)$ in the direction of the eigenvectors, denoted $x_{\pm}(t)$, one for $P_{ac} (u, u_t)$, as well as  an ODE for $a(t)$.

\medskip
Following \cite{KS14} and \cite{Beceanu}, we write the Cauchy problem \eqref{equ:ivp_ansatz_sec5} in vector form, with
\begin{equation}
U = \begin{pmatrix} u \\ 
\partial_t u
\end{pmatrix}, \qquad
\mathcal{H} = \begin{pmatrix} 0 & 1 \\ 
-H & 0
\end{pmatrix}, 
\end{equation}
and
\begin{equation}
\textbf{N} = \begin{pmatrix} 0 \\
- \partial_t^2 \phi_{a(t)} + (V - V_{a(t)}) u + N(u, \phi_{a(t)}) ,
\end{pmatrix} 
\end{equation}
we obtain
\begin{align}\label{equ:vec_eqn}
\partial_t U = \cH U + \textbf{N}.
\end{align}
Since $-\kappa^2$ is the negative eigenvalue of $H$, the operator $\mathcal{H}$ has spectrum $i \mathbb{R} \cup \{ \pm \kappa \}$ with eigenvector for $\pm \kappa$ given by
\begin{align}
\textbf{Y}_{\pm} = (2\kappa)^{-1/2} \begin{pmatrix}
 Y \\
 \pm \kappa Y
\end{pmatrix}.
\end{align}
The Riesz projections to $\pm \kappa$ are given by
\begin{align}
P_{\pm} = \mp \langle \cdot, J \textbf{Y}_\mp \rangle \textbf{Y}_{\pm} , \qquad J = \begin{pmatrix} 0 & 1 \\ -1 & 0 \end{pmatrix},
\end{align}
hence setting
\[
 x_\pm(t) = \mp \langle U(t), J \textbf{Y}_\mp \rangle,
\]
we apply the projections to \eqref{equ:vec_eqn} and we obtain
\begin{align}
\partial_t x_\pm(t) = \pm \kappa x_{\pm}(t) \mp \langle \textbf{N}, J \textbf{Y}_\mp \rangle.
\end{align}
Letting
\[
N_2(u):= (V - V_{a(t)}) u + N(u, \phi_{a(t)})
\]
yields the equations
\begin{align}
x_\pm(t) & = e^{\pm t\kappa} x_{\pm}(0) \pm \int_0^t e^{\pm(t-s) \kappa} (2\kappa)^{-1/2} \langle \textbf{N}, J\textbf{Y}_{\mp} \rangle ds \\
& =  e^{\pm t\kappa} x_{\pm}(0) \pm \int_0^t e^{\pm(t-s) \kappa} (2\kappa)^{-1/2} \langle - \partial_s^2 \phi_{a(s)} + N_2(u), Y \rangle ds
\end{align}
We integrate by parts in the $\partial_s^2 \phi_{a(s)}$ term and use that for the resonance $\varphi_a$, we have $\langle \varphi_a, Y \rangle = 0$ to obtain
\begin{align}
x_\pm(t) & = (2\kappa)^{-1/2} e^{\pm t\kappa} \bigl( \langle \kappa (\psi_0 - \phi) \pm \psi_1, Y \rangle + \langle \kappa f_0^\omega \mp f_1^\omega, Y \rangle \bigr) \\
& \pm (2\kappa)^{-1/2} \dot a(t) \langle \varphi_{a(t)} - a(t)^{-5/4}\varphi , Y \rangle \phantom{\int}  \\
&\mp (2\kappa)^{-1/2} \int_0^t e^{\pm(t-s) \kappa} \langle \mp \kappa \dot a ( \varphi_{a(s)} - a(s)^{-5/4} \varphi ) + N_2(u), Y \rangle ds.
\end{align}

The orthogonality condition for $f^\omega$ (see Remark \ref{equ:orthog_preserved}) as well as the additional orthogonality condition for $(\psi_0, \psi_1) \in \cN_1$ ensures that we can choose $h$ below so that $x_+(t) \in L^1_t \cap L^\infty_t$ on $\bR_+$.

\medskip
By projecting the equation for $u$, we obtain an equation for $P_{ac} u$, assuming $a(0) = 1$:
  \begin{equation} \label{equ:projected}
 \left\{ \begin{aligned}
  -\partial_t^2 P_{ac} u + H P_{ac} u &= - P_{ac} \bigl( \partial_t^2 \phi_{a(t)} +  N_2(u) \bigr), \\
  (u, \partial_t u)|_{t=0} &= (\psi_0 - \phi, \psi_1- \dot a(0) \varphi ) + f^\omega.
 \end{aligned} \right.
\end{equation}
To solve \eqref{equ:projected}, we will recenter the solution around the linearized evolution of the random data. Recall from Definition \ref{def:randomization} that
\[
f^\omega = (f_0^\omega, f_1^\omega),
\]
where
\begin{align}\label{equ:randomiWation}
f_0^\omega &:=  \langle f_{0}, Y \rangle Y + (P_{ac} f_{0,lo} + P_0 f_{0,hi}) + \sum_{k \geq k_0} g_k(\omega) P_k f_{0,hi}\\
f_1^\omega &:=  \langle f_1, Y \rangle Y+ (P_{ac} f_{1,lo} + P_0 f_{1,hi})+ \sum_{k \geq k_0} h_k(\omega) P_k f_{1,hi} ,
\end{align}
and we recall that
\[
\langle \kappa f_0^\omega \pm f_0^\omega, Y \rangle = \langle \kappa f_0 \pm f_1, Y \rangle.
\]
We set
\[
(\gamma_0, \gamma_1) := \biggl(\langle f_{0}, Y \rangle Y + (P_{ac} f_{0,lo} + P_0 f_{0,hi}),  \langle f_1, Y \rangle Y+ (P_{ac} f_{1,lo} + P_0 f_{1,hi}) \biggr) ,
\]
then by 
%Lemma \ref{lem:low_freq_bds}, the exponential decay of $Y$, the fact that
%\[
%P_{ac} f = f - \langle f , Y \rangle Y
%\]
%from the definition of $P_{ac}$ in \eqref{p_ac}, and 
the condition that $f \in \cM_s$, we have that
\begin{align}\label{equ:phi_bds}
\|(\gamma_0, \gamma_1) \|_{\dot H^1 \cap |\nabla|^{-1} L^{3/2,1} \times L^2 \cap L^{3/2,1}} \lesssim \varepsilon,
\end{align}
and $(\gamma_0, \gamma_1)$ satisfies the orthogonality condition
\begin{align}\label{equ:gamma_orthog}
\langle \kappa \gamma_0 + \gamma_1, Y \rangle = 0.
\end{align}
We recall that
\[
 f_{\geq k_0}^\omega :=  \biggl(\sum_{k \geq k_0} g_k(\omega) P_k f_{0,hi}, \sum_{k \geq k_0} h_k(\omega) P_k f_{1,hi} \biggr).
\]
We define $v(t)$ by
\[
P_{ac} v(t) = P_{ac} u(t) - W(t) f_{\geq k_0}^\omega, \qquad P_p v = P_p u,
\]
and by substituting this identity into the equation \eqref{equ:projected}, we may generalize the study of \eqref{equ:projected} to that of a deterministic forced equation
  \begin{equation} \label{equ:forced}
 \left\{ \begin{aligned}
  -\partial_t^2 P_{ac} &v + H P_{ac} v \\
  &= - P_{ac} \bigl( \partial_t^2 \phi_{a(t)} +  (V -V_{a(t)}) (P_p u + P_{ac} v(t) + F) \\
  & \hspace{44mm}+  N\bigl(P_p u + P_{ac} v + F, \phi_{a(t)}\bigr) \bigr), \\
  (v, \partial_t v)|_{t=0} &= (\psi_0 - \phi, \psi_1- \dot a(0) \varphi) ,\quad (\psi_0, \psi_1) \in  \cN_1,
 \end{aligned} \right.
\end{equation}
where the forcing term $F$ generalizes the term $ W(t) f_{\geq k_0}^\omega$. In Duhamel form, we have
\begin{align}
P_{ac} v(t) &= \cos(t \sqrt{H}) P_{ac} (\psi_0 - \phi) + \frac{\sin(t \sqrt{H})}{\sqrt{H}} P_{ac} ( \psi_1  - \dot a(0) \varphi) \\
& \hspace{34mm}+ \int_0^t \frac{\sin((t-s) \sqrt{|H|})}{\sqrt{|H|}} P_{ac} (- \partial_s^2 \phi_{a(s)} + N_2(P_p u + P_{ac} v + F)) ds,
\end{align}
and integrating by parts, we obtain
\begin{align*}
\int_0^t \frac{\sin((t -s) \sqrt{|H|})P_{ac}}{\sqrt{|H|}}  (- \partial_s^2 \phi_{a(s)})  ds&= \frac{\sin(t  \sqrt{|H|})P_{ac}}{\sqrt{|H|}} \dot a(0) \varphi \\& \hspace{14mm} - \int_0^t \cos((t -s) \sqrt{|H|})P_{ac}  \dot a(s) \varphi_{a(s)} ds.
\end{align*}
Hence
\begin{align} \label{equ:pacv1}
P_{ac} v(t) &=  \cos(t \sqrt{H}) P_{ac} ( \psi_0- \phi) + \frac{\sin(t  \sqrt{|H|})P_{ac}}{\sqrt{|H|}}  \psi_1 \\
&+ \int_0^t \frac{\sin((t -s) \sqrt{|H|})P_{ac}}{\sqrt{|H|}} N_2(P_p u + P_{ac} v + F) \, ds\\
& -  \int_0^t \cos((t -s) \sqrt{|H|})P_{ac}  \dot a(s) \varphi_{a(s)} ds.
\end{align}
Using that
\[
\cos(t \sqrt{H}) P_{ac} \varphi = \varphi,
\]
the last term in this equation yields
\begin{align*}
 \int_0^t \cos((t -s) \sqrt{|H|})P_{ac} & \dot a(s) \varphi_{a(s)} ds\\
  &=  \int_0^t a(s)^{- 5/4} \dot a(s) \varphi  ds \\
 & +  \int_0^t \cos((t -s) \sqrt{|H|})P_{ac}  \dot a(s)(  \varphi_{a(s)} - a(s)^{- 5/4}  \varphi) ds .
\end{align*}
We recall the representation formula of Theorem \ref{thm:rep}:
\begin{align}
\frac{\sin(t \sqrt{|H|})}{\sqrt{|H|}}P_{ac} &= - \frac{4 \pi }{|\langle V , \varphi \rangle|^2}  \varphi \otimes V \varphi \int_0^t \frac{\sin(s \sqrt{-\Delta})}{\sqrt{-\Delta}}  ds + \mathcal{S}(t)\\
\cos(t \sqrt{H}) P_{ac} &= - \frac{4 \pi }{|\langle V , \varphi \rangle|^2}  \varphi \otimes V \varphi \int_0^t \cos(s \sqrt{-\Delta}) ds + \mathcal{C}(t),
\end{align}
where $\mathcal{S}(t)$ and $\mathcal{C}(t)$ will satisfy Strichartz estimates.  

We substitute these expressions into the equation for $P_{ac} v$ and impose the condition that the terms involving $\varphi$ should cancel for any $t \geq 0$. Taking a time derivative of the resulting equation yields an equation for $a(t)$  (see \cite[(2.5)]{Beceanu}):
\begin{align}
\frac{\dot a(t)}{a(t)^{5/4}} & =  \frac{4 \pi} {\langle V, \varphi \rangle^2}  \Bigl \langle  \cos (t \sqrt{-\Delta}) ( \psi_0 - \phi) + \frac{\sin(t \sqrt{-\Delta})}{\sqrt{-\Delta}} \psi_1 \\
&\hspace{34mm} + \int_0^t \frac{\sin((t -s) \sqrt{-\Delta})}{\sqrt{-\Delta}} N_2(P_pu + P_{ac} v + F) ds \\
&\hspace{34mm}  + \int_0^t \cos((t -s) \sqrt{-\Delta}) \dot a(s)(\varphi_{a(s)} - a(s)^{-5/4} \varphi) ds, V \varphi \Bigr \rangle\\
\dot a(0) &= \frac{4 \pi}{ {\langle V, \varphi\rangle^2}  } \langle V \varphi, ( \psi_0 - \phi) \rangle
\end{align}
and the equation
\begin{align}
P_{ac} v(t) &= \mathcal{C}(t)( \psi_0 - \phi) + \mathcal{S}(t)\psi_1 \\
&+ \int_0^t \mathcal{S}(t-s) N_2(P_p u + P_{ac} v(t) + F) + \int_0^t \mathcal{C}(t-s) \dot a(s) \bigl( \varphi_{a(s)} - a(s)^{-5/4} \varphi \bigr) ds,
\end{align}
see \cite[(2.6)]{Beceanu}.

\subsection{Contraction mapping argument: estimates} \label{sec:contraction_setup}
We recall that
\begin{align}\label{equ:comp_of_u}
u = P_p u + P_{ac} v + W(t)P_{\geq k_0}  f_{hi}^\omega, \quad v = P_p u + P_{ac} v.
\end{align}
In the sequel, we consider only $t \geq 0$.  We define
\begin{align}\label{banach_x}
X = \bigl\{(v,a) \,|\, v \in L^{6,2}_x L^\infty_t \cap L^\infty_x L^2_t \cap L^\infty_x L^1_t, \, \dot a \in L^1_t \cap L^\infty_t, \,a(0) = 1 \bigr\},
\end{align}
where all spacetime norms are be taken over $\bR_+ \times \bR^3$, with distance
\begin{align}
\|(v,a) - (0,1)\|_X = \|v\|_{L^{6,2}_x L_t^\infty \cap L^\infty_x L^2_t \cap L^\infty_x L^1_t} + \| \dot a\|_{L_t^\infty \cap L^1_t}.
\end{align}
We let $B_{\varepsilon}((0,1)) $ denote the ball of radius $\varepsilon$ centered at $(0,1)$ in the complete metric space $X$. We recall the definition of the Banach space $Z$ which captures the required a priori bounds for our forcing term, with norm given by
\begin{equation}
\begin{split}
\|F\|_{Z(I)} &= \|F\|_{L^6_x L^\infty_t \cap L^{36/5}_x L^\infty_t(I \times \mathbb{R}^3)}  + \|F\|_{L^{36/5}_x L^6_t (I \times \mathbb{R}^3)} + \|F\|_{L^{9,2}_x L^3_t (I \times \mathbb{R}^3)}  \\
& \hspace{25mm} +  \| \langle x \rangle^{-1 - \theta} F \|_{L^{2}_x L^1_t \cap L^{2}_x L^2_t (I \times \mathbb{R}^3)}  , \qquad 1/2 < \theta < 3/2 .
\end{split}
\end{equation}
We will construct the solution by fixed point in $(v,a)$ in a ball in $X$ centered at $(0, 1)$. More precisely, 
for fixed $(f_0, f_1)$, $(\psi_0, \psi_1)$, and $F$ we define
\[
\Phi_{(f_0, f_1), \: (\psi_0, \psi_1),\:a(0) = 1, \: F}(v, a) = (w,b),
\]
where
\[
w(t) = (2\kappa)^{-1/2} ( x_+(t) + x_-(t)) Y + P_{ac} w,
\]
and
\begin{equation}\label{equ:x_eqn}
\begin{split}
\qquad \qquad x_\pm(t) & = (2\kappa)^{-1/2} e^{\pm t\kappa} \bigl( \langle \kappa (\psi_0 - \phi) \pm \psi_1, Y \rangle + \langle \kappa f_0 \pm f_1, Y \rangle \bigr) \\
& \pm (2\kappa)^{-1/2} \dot a(t) \langle \varphi_{a(t)} - a(t)^{-5/4}\varphi , Y \rangle \phantom{\int}  \\
&\mp (2\kappa)^{-1/2} \int_0^t e^{\pm(t-s) \kappa} \langle \mp \kappa \dot a ( \varphi_{a(t)} - a(t)^{-5/4} \varphi ) + N_2(v + F), Y \rangle ds.
\end{split}
\end{equation}
\begin{equation}\label{equ:pcv}
\begin{split}
P_{ac} w(t) &=  \mathcal{C}(t) ( \psi_0 - \phi) + \mathcal{S}(t)\psi_1 \\
& \hspace{14mm}+ \int_0^t  \mathcal{S}(t-s) N_2(v + F) \, ds +  \int_0^t  \mathcal{C}(t-s) \dot a(s) \varphi_{a(s)} ds.
\end{split}
\end{equation}
\begin{equation}\label{equ:a}
\begin{split}
\dot b(t) & =  \frac{4 \pi a(t)^{5/4}} {\langle V, \varphi \rangle^2}  \Bigl \langle  \cos (t \sqrt{-\Delta}) (\psi_0 - \phi) + \frac{\sin(t \sqrt{-\Delta})}{\sqrt{-\Delta}} \psi_1  \\
&\hspace{34mm} + \int_0^t \frac{\sin((t -s) \sqrt{-\Delta})}{\sqrt{-\Delta}} N_2(v + F) ds \\
&\hspace{34mm}  + \int_0^t \cos((t -s) \sqrt{-\Delta}) \dot a(s)(\varphi_{a(s)} - a(s)^{-5/4} \varphi) ds, V \varphi \Bigr \rangle,\\
b(0) &= a(0) = 1.
\end{split}
\end{equation}

We turn now to the main result of this section, which will imply the main theorem.

\begin{proposition}\label{prop:fixed}
There exists $\varepsilon_0 > 0$ such that for every $\varepsilon < \varepsilon_0$, and for every fixed $(f_0, f_1) \in \cM_s$ for $s > 5/6$ and $(\psi_0, \psi_1) \in \cN_1$, and forcing term $F(t,x) \in Z$ with $\|F\|_{Z} < \varepsilon$, the following holds:  for $(\overline{v},\overline{a}) \in B_\varepsilon((0,1)) \subseteq X$, there exists a unique function 
\[
h\equiv h_{F}(\overline{v},\overline{a}) \in \bR
\]
such that for $(v_1, a_1), (v_2, a_2) \in B_\varepsilon((0,1))$,
\[
|h(v_1, a_1) - h(v_2, a_2)| \lesssim \varepsilon \| (v_1, a_1) - (v_2, a_2)\|_X,
\]
and such that the mapping 
\[
\Phi_{(f_0,f_1),(\psi_0 + h(v,a) Y,  \psi_1  + \kappa h(v,a) Y), \:a(0) = 1, \:F}(v,a)
\]
defined by \eqref{equ:x_eqn},  \eqref{equ:pcv} and \eqref{equ:a} maps $B_\varepsilon((0,1))$ to itself and is a contraction. \end{proposition}

\begin{proof}
First we establish several estimates. We let
\begin{align}
I& := \|\mathcal{C}(t)(\psi_0 - \phi) + \mathcal{S}(t)\psi_1 \|_{L^{6,2}_x L^\infty_t \cap L^\infty_x L^2_t \cap L^\infty_x L^1_t} \\
II&:= \|\dot a(t) ( \varphi_{a(t)} - a(t)^{-5/4} \varphi) \|_{L^1_t |\nabla|^{-1} L^{3/2,1}_x \cap L^1_t \dot H^1_x \cap L^\infty_t |\nabla|^{-1} L^{3/2,1}_x \cap L^\infty_t \dot H^1_x} \\
III &:= \| (V - V_{a(t)})  ( v +  F)\|_{L^{6/5,2}_x L^\infty_t \cap L^{3/2,1}_x L^2_t \cap L^{3/2,1}_x L^1_t}  \\
IV&:= \|N(v + F, \phi_{a(t)})\|_{L^{6/5,2}_x L^\infty_t \cap L^{3/2,1}_x L^2_t\cap L^{3/2,1}_x L^1_t} .
\end{align}
We will estimate these expressions separately.
\subsection*{Term I} The estimate for the first term follows from Theorem \ref{thm:rep}. We have
\begin{align}
&\|\mathcal{C}(t)(\psi_0 - \phi) + \mathcal{S}(t)\psi_1 \|_{L^{6,2}_x L^\infty_t \cap L^\infty_x L^2_t \cap L^\infty_x L^1_t} \\
& \lesssim  \|\psi_0 - \phi \|_{|\nabla|^{-1} L^{3/2,1} \cap \dot H^1} + \|\psi_1\|_{L^{3/2,1} \cap L^2}
\end{align}
\subsection*{Term II} 
For the second term, we use Lemma \ref{lem:symb} to bound
\begin{align}
 &\|\dot a(t) ( \varphi_{a(t)} - a(t)^{-5/4} \varphi) \|_{L^1_t |\nabla|^{-1} L^{3/2,1}_x \cap L^1_t \dot H^1_x \cap L^\infty_t |\nabla|^{-1} L^{3/2,1}_x \cap L^\infty_t \dot H^1_x} \lesssim \| \dot a\|_{L^\infty_t \cap L^1_t} \|\dot a\|_{L^1_t}  \lesssim \varepsilon^2.
\end{align}
\subsection*{Term III}
For the third term we handle the contributions involving $v$ and $F$ separately. We note that on $B_\varepsilon(0) \subset X$, we have $1/2 < a < 3/2$, hence by Lemma \ref{v_bds}, we have
\begin{align}
&\| (V - V_{a(t)}) v \|_{L^{6/5,2}_x L^\infty_t \cap L^{3/2, 1}_x L^2_t \cap L^{3/2,1}_x L^1_t } \\
& \lesssim \| V -V_{a(t)} \|_{L^{3/2, 1}_x L^\infty_t} \left( \|v\|_{L_x^{6,2} L_t^\infty} +  \|v\|_{L_x^{\infty} L_t^2} + \|v\|_{L_x^{\infty} L_t^1}  \right)\\
& \lesssim \|\dot a\|_{L^1_t}  \left( \|v\|_{L_x^{6,2} L_t^\infty} +  \|v\|_{L_x^{\infty} L_t^2} + \|v\|_{L_x^{\infty} L_t^1} \right).
\end{align}
Similarly, for the term involving $F$ we estimate
\begin{align}
&\| (V - V_{a(t)}) F \|_{L^{6/5,2}_x L^\infty_t }\\
& \lesssim \|V -V_{a(t)} \|_{L^{3/2}_x L^\infty_t  } \|F\|_{L_x^{6} L_t^{\infty}}\\
& \lesssim \|\dot a\|_{L^1_t}  \| F\|_{Z}, 
\end{align}
while for the terms involving the $L^1_t \cap L^2_t$ norm, we fix $1/2<\theta < 5/2$ as in the definition of the $Z$ norm, and using Lemma \ref{v_bds} we estimate
\begin{align}
&\| (V - V_{a(t)}) F \|_{L^{3/2,1}_x L^1_t \cap L^{3/2,1}_x L^2_t}\\
& \lesssim \|\langle x \rangle^{1 + \theta} (V -V_{a(t)}) \|_{L^{6,2}_x L_t^\infty } \|\langle x \rangle^{-1 - \theta} F\|_{L_x^{2} L_t^1 \cap L^2_x L^2_t} \\
& \lesssim  \|\dot a\|_{L^1_t} \|F\|_{Z}.
\end{align}

\subsection*{Term IV}
First we estimate
\begin{align}
&\|(v +  F)^5\|_{L^{6/5,2}_x L^\infty_t \cap L^{3/2,1}_x L^2_t \cap L^{3/2}_x L^1_t}.
\end{align}
For the $v$ component we have (see \cite{Beceanu}) that
\begin{align}
\|v^5\|_{L^{6/5,2}_x L^\infty_t \cap L^{3/2,1}_x L^2_t \cap L^{3/2}_x L^1_t} \lesssim \|v\|_{L^{6,2}_x L^\infty_t}^5 + \|v\|_{L^{6,2}_x L^\infty_t}^4 \|v\|_{L^{\infty}_x L^2_t} + \|v\|_{L^{6,2}_x L^\infty_t}^4 \|v\|_{L^{\infty}_x L^1_t} .
\end{align}
Now, for the terms involving the forcing term $F$, we can estimate the $L^{6/5,2}_x L^\infty_t$ norm in the same way, while we estimate this term somewhat differently for the $L^{3/2,1}_x L^1_t$ component. We have
\begin{align}
\|(v +  F)^5- F^5\|_{L^{6/5,2}_x L^\infty_t  \cap L^{3/2,1}_x L^1_t\cap L^{3/2,1}_x L^2_t} &\lesssim \sum_{i=1}^4 \|F\|_{L^{6}_x L^\infty_t}^i \|v\|_{L^{6,2}_x L^\infty_t}^{4-i}  \|v\|_{L^{\infty}_x L^1_t \cap L^{\infty}_x L^2_t  \cap L^{6,2}_x L^\infty_t}  \\
& \lesssim  \sum_{i=1}^4 \|F\|_{Z}^i \|v\|_{L^{6,2}_x L^\infty_t}^{4-i} \|v\|_{L^{\infty}_x L^1_t \cap L^{\infty}_x L^2_t\cap L^{6,2}_x L^\infty_t}
\end{align}
and we estimate 
\begin{align}\label{equ:f5_ests}
\qquad \|F^5\|_{L^{3/2,1}_x L^1_t \cap L^{3/2,1}_x L^2_t  }  \lesssim \|F\|_{L^{6}_x L^\infty_t }^2 \|F\|_{L^{9,2}_x L^3_t }^3  + \|F\|_{L^{36/5}_x L^\infty_t }^{3} \|F\|_{L^{36/5}_x L^6_t } \|F\|_{L^{9,2}_x L^3_t } ,
\end{align}
where in the first term, we use Lemma \ref{lem:lorentz_holder} and the inequality
\[
\|f\|_{L^{p,q_1}} \leq \|f \|_{L^{p, q_2}}, \qquad 0 < p \leq \infty, \quad 0 < q_2 \leq q_1 \leq \infty, \quad \frac{1}{q_2} = \frac{2}{6} + \frac{3}{2}, \quad q_1 = 1.
\]
(see Lemma \ref{lem:interp}). The second term is handled similarly. For the other terms in the nonlinear component, we may estimate all the terms involving only $v$ as in \cite{Beceanu}, so we treat only those terms involving $F$. We have
\begin{align}
&\|\phi_{a(t)}^3 (v F + F^2)\|_{L^{6/5,2}_x L^\infty_t} \\
& \lesssim \|\phi_{a(t)}\|^3_{L^{6,2}_x L^\infty_t} \bigl( \|v\|_{L^{6,2}_x L^\infty_t}  \|F\|_{L^{6}_x L^\infty_t}  + \|F\|_{L^{6}_x L^\infty_t} ^2 \bigr) \\
& \lesssim  \|\phi_{a(t)}\|^3_{L^{6,2}_x L^\infty_t} \bigl( \|v\|_{L^{6,2}_x L^\infty_t}  \|F\|_{Z}  + \|F\|_{Z} ^2 \bigr) \\
\end{align}
and similarly for the terms
\begin{align}
&\|\phi_{a(t)}^2 (v^2 F + F^2v + F^3)\|_{L^{6/5,2}_x L^\infty_t} , \quad \|\phi_{a(t)} (v^3 F + F^2v^2 +F^3 v + F^4)\|_{L^{6/5,2}_x L^\infty_t} .
\end{align}
For the terms in $L^{3/2,1}_xL^1_t \cap L^{3/2,1}_xL^2_t $, we have
\begin{align}
\|\phi_{a(t)}^3 v F\|_{L^{3/2,1}_xL^1_t \cap L^{3/2,1}_xL^2_t } \lesssim \|\phi_{a(t)}\|_{L^{6,2}_x L^\infty_t}^3 \|v\|_{L^\infty_x L^1_t \cap L^\infty_x L^2_t} \| F\|_{L^6_x L^\infty_t}
\end{align}
and similarly for
\begin{align}
\|\phi_{a(t)}^2 (v^2 F + F^2 v)\|_{L^{3/2,1}_xL^1_t \cap L^{3/2,1}_xL^2_t }, \quad \|\phi_{a(t)}^2  (v^3 F + F^2v^2 +F^3 v)\|_{L^{3/2,1}_xL^1_t \cap L^{3/2,1}_xL^2_t }.
\end{align}
We next estimate 
\begin{align}
\|\phi_{a(t)}^3 F^2  \|_{L^{3/2,1}_x L^1_t \cap L^{3/2,1}_x L^2_t} \lesssim \|\langle x \rangle^{\frac{1+ \theta}{3}} \phi_{a(t)}\|^3_{L^{36, 108/13}_x L^\infty_t} \| F \|_{L^{36/5}_x L^\infty_t} \| \langle x \rangle^{-1 - \theta} F\|_{L^2_x L^1_t \cap L^2_x L^2_t},
\end{align}
By Corollary \ref{cor:phi_bds}, we have
\[
\|\langle x \rangle^{\frac{1+ \theta}{3}} \phi_{a(t)}\|^3_{L^{36,108/13}_x L^\infty_t} \lesssim (1 + \|\dot a\|_{L^1_t})
\]
provided
\[
\frac{1+ \theta}{3} < 1 - \frac{1}{12},
\]
and hence for $\theta < 7/4$ we obtain
\begin{align}
\|\phi_{a(t)}^3 F^2  \|_{L^{3/2,1}_x L^1_t \cap L^{3/2,1}_x L^2_t} \lesssim (1 + \|\dot a\|_{L^1_t}) \|F\|_{Z}^2.
\end{align}
Finally we estimate the terms
\begin{align}
 \|\phi_{a(t)}^2  F^3 \|_{L^{3/2,1}_x L^1_t\cap L^{3/2,1}_x L^2_t}, \quad  \|\phi_{a(t)}  F^4 \|_{L^{3/2,1}_x L^1_t\cap L^{3/2,1}_x L^2_t} 
\end{align}
using the estimates from \eqref{equ:f5_ests}, noting that for both these estimates, the $\phi_{a(t)}$ terms can be placed into $L^{6,2}_x L^\infty_t$. Putting these estimates together, we have shown that
\begin{align}
\|N(v + F, \phi_{a(t)})\|_{L^{6/5,2}_x L^\infty_t \cap L^{3/2,1}_x L^2_t\cap L^{3/2,1}_x L^1_t}  \lesssim \bigl( \|v\|_{X}^2 +  \|F\|_{Z}^2 \bigr) \lesssim \varepsilon^2
\end{align}

\medskip
If we replace $(\psi_0, \psi_1)$ by $(\psi_0 + hY, \psi_1 + \kappa h Y)$, then the equation for the $x_+$ term will dictate the value of $h$ in the statement of the proposition. Indeed, we can rewrite the equation for the new $x_+(t)$ as 
\begin{equation}
\begin{split}
x_+(t) & = (2\kappa)^{-1/2} e^{t\kappa} \biggl( 2h\kappa\langle Y, Y\rangle  - \int_0^\infty e^{-s \kappa} \langle - \kappa \dot a ( \varphi_{a(t)} - a(t)^{-5/4} \varphi ) + N_2(v + F), Y \rangle \biggr)\\
& +  (2\kappa)^{-1/2} \int_t^\infty e^{(t-s) \kappa} \langle - \kappa \dot a ( \varphi_{a(t)} - a(t)^{-5/4} \varphi ) + N_2(v + F), Y \rangle \\
& + (2\kappa)^{-1/2} \dot a(t) \langle \varphi_{a(t)} - a(t)^{-5/4}\varphi , Y \rangle. \phantom{\int} 
\end{split}
\end{equation}
Noting that
\begin{align}
&\| (2\kappa)^{-1/2} \int_t^\infty e^{(t-s) \kappa} \langle - \kappa \dot a ( \varphi_{a(t)} - a(t)^{-5/4} \varphi ) + N_2(v + F), Y \rangle \|_{L^1_t \cap L^\infty_t}\\
& \lesssim \|\dot a \|_{L^1_t \cap L^\infty_t} \| \varphi_{a(t)} - a(t)^{-5/4}\varphi \|_{L^\infty_t L^2_x}  + \| N_2(v + F)\|_{L^{3/2,1}_x L^1_t} ,
\end{align}
 we obtain that $x_+ \in L_t^1 \cap L_t^\infty$ if and only if
\begin{align}\label{h_eq}
2 h\kappa \langle Y, Y \rangle = \int_0^\infty e^{-s\kappa} \langle - \kappa \dot a ( \varphi_{a(t)} - a(t)^{-5/4} \varphi ) +  N_2(v + F), Y \rangle ds
\end{align}
which determines a unique value for $h \equiv h_F(v,a)$. Furthermore, we have by the same arguments used above that
\begin{align}
|h| &\lesssim \varepsilon ( \varepsilon + \|F\|_Z),
\end{align}
and arguing again as in \cite{Beceanu}, for $j = 1,2$ for
\[
\|(v_j, a_j) - (0,1) \|_{X} \lesssim \varepsilon < \frac{1}{2},
\]
we have
\begin{align}
| h(v_1,a_1) - h(v_2, a_2) | \lesssim \varepsilon \| (v_1,a_1) - (v_2, a_2)\|_{X}.
\end{align}

Hence we study the mapping given by  \eqref{equ:x_eqn},  \eqref{equ:pcv} and \eqref{equ:a} with $h_F(v,a)$ deinfed above and initial data
\[
\bigl(\widetilde{\psi}_0, \widetilde{\psi}_1 \bigr)  := \bigl(\psi_0 + h Y,  \psi_1+ \kappa h Y\bigr),
\]
and note that
\[
\kappa (\widetilde{\psi}_0 - \phi) + \widetilde{\psi}_1 = \kappa (\psi_0 - \phi) + \psi_1.
\]

We begin by treating \eqref{equ:x_eqn}.  For the $x_{\pm}(t)$ terms, thanks to the definition of $h$ we may argue as in \cite{Beceanu} using the above estimates, and we obtain
\begin{align}\label{equ:x_minus}
\|x_{\pm}\|_{L^1\cap L^\infty} \lesssim\|\psi_0 - \phi \|_{|\nabla|^{-1} L^{3/2,1} \cap \dot H^1}  + \|\psi_1\|_{L^{3/2,1} \cap L^1} + \varepsilon ( \varepsilon + \|F\|_Z).
\end{align}
Putting these estimates together, we have
\begin{align}
\|P_p w\|_{L^{6,2}_x L^\infty_t \cap L^\infty_x L^2_t \cap L^\infty_x L^1_t} \lesssim \|\psi_0 - \phi \|_{|\nabla|^{-1} L^{3/2,1} \cap \dot H^1} + \|\psi_1\|_{L^{3/2,1} \cap L^2}  + \varepsilon ( \varepsilon + \|F\|_Y).
\end{align}
Next we treat $P_{ac} w$. Since
\[
P_{ac} \widetilde{\psi}_0 = P_{ac} \psi_0, \qquad P_{ac} \widetilde{\psi}_1 = P_{ac} \psi_1,
\]
we have by the Strichartz estimates of Theorem \ref{thm:rep} 
\begin{align}
\|P_{ac} w\|_{L^{6,2}_x L^\infty_t \cap L^\infty_x L^2_t \cap L^\infty_x L^1_t} &=  \|\mathcal{C}(t)(\psi_0 - \phi) + \mathcal{S}(t) \psi_1 \|_{L^{6,2}_x L^\infty_t \cap L^\infty_x L^2_t \cap L^\infty_x L^1_t} \\
&+ \|\dot a(t) ( \varphi_{a(t)} - a(t)^{-5/4} \varphi) \|_{L^1_t |\nabla|^{-1} L^{3/2,1}_x \cap L^1_t \dot H^1_x \cap L^\infty_t |\nabla|^{-1} L^{3/2,1}_x \cap L^\infty_t \dot H^1_x} \\
& + \| (V - V_{a(t)})  ( v(t) +  F)\|_{L^{6/5,2}_x L^\infty_t \cap L^{3/2,1}_x L^2_t  \cap L^{3/2,1}_x L^1_t}  \\
&+ \|N(v + F, \phi_{a(t)})\|_{L^{6/5,2}_x L^\infty_t \cap L^{3/2,1}_x L^2_t  \cap L^{3/2,1}_x L^1_t} \\
& =: I + II + III + IV,
\end{align}
and using our estimates above we obtain
\begin{align}
\|P_{ac} w\|_{L^{6,2}_x L^\infty_t \cap L^\infty_x L^2_t \cap L^\infty_x L^1_t} &\lesssim  \|\psi_0 - \phi \|_{|\nabla|^{-1} L^{3/2,1} \cap \dot H^1} + \|\psi_1\|_{L^{3/2,1} \cap L^2} + \varepsilon ( \varepsilon + \|F\|_Z).
\end{align}
We now turn to the estimates for $b(t)$. From \eqref{equ:a} we have
\begin{align}
&\|\dot b\|_{L^1 \cap L^\infty}\\
& \lesssim \biggr\|  a(t)^{5/4}  \frac{4 \pi} {\langle V, \varphi \rangle^2}  \Bigl \langle  \cos (t \sqrt{-\Delta}) (\psi_0 - \phi) + \frac{\sin(t \sqrt{-\Delta})}{\sqrt{-\Delta}}  \psi_1, V \varphi \Bigr \rangle \biggr\|_{L_t^\infty \cap L_t^1} \\
&  + \biggr\|a(t)^{5/4}  \frac{4 \pi} {\langle V, \varphi \rangle^2}   \Bigl \langle  \int_0^t \frac{\sin((t -s) \sqrt{-\Delta})}{\sqrt{-\Delta}}  W_2 ds, V \varphi \Bigr \rangle \biggr\|_{L_t^\infty \cap L_t^1} \\
&  + \biggl\|a(t)^{5/4}  \frac{4 \pi} {\langle V, \varphi \rangle^2}   \Bigl \langle \int_0^t \cos((t -s) \sqrt{-\Delta}) \dot a(s)(\varphi_{a(s)} - a(s)^{-5/4} \varphi) ds, V \varphi \Bigr \rangle \biggr\|_{L_t^\infty \cap L_t^1}.
\end{align}
Since
\[
\|a \|_{L^\infty_t} \leq 1 + \|\dot a\|_{L^1_t},
\]
we obtain that for the first term 
\begin{align}
& \biggr\|  a(t)^{5/4}  \frac{4 \pi} {\langle V, \varphi \rangle^2}  \Bigl \langle  \cos (t \sqrt{-\Delta}) (\psi_0 - \phi) + \frac{\sin(t \sqrt{-\Delta})}{\sqrt{-\Delta}} \psi_1, V \varphi \Bigr \rangle \biggr\|_{L_t^\infty \cap L_t^1}\\
 & \lesssim  \frac{4 \pi (1 + \|\dot a\|_{L^1_t})^{5/4}} {\langle V, \varphi \rangle^2} \|V \varphi \|_{L^1_x} \bigl( \||\nabla| (\psi_0 - \phi)\|_{L^{3/2, 1}_x} + \|\psi_1\|_{L^{3/2,1}_x} +  \||\nabla|\gamma_0\|_{L^{3/2, 1}_x} + \|\gamma_1\|_{L^{3/2,1}_x}  \bigr) .
\end{align}
For the second term, we have
\begin{align}
& \biggr\|a(t)^{5/4}  \frac{4 \pi} {\langle V, \varphi \rangle^2}  \Bigl \langle  \int_0^t \frac{\sin((t -s) \sqrt{-\Delta})}{\sqrt{-\Delta}} N_s(v + F) ds, V \varphi \Bigr \rangle \biggr\|_{L_t^\infty \cap L_t^1} \\
 & \lesssim \frac{4 \pi (1 + \|\dot a\|_{L^1_t})^{5/4}} {\langle V, \varphi \rangle^2} \|V \varphi \|_{L^1_x}  \|(V -V_{a(t)})  ( v + F)  + N(v + F, \phi_{a(s)}) \|_{L^{3/2,1}_x L^1_t}\\
 & \lesssim  \frac{4 \pi (1 + \|\dot a\|_{L^1_t})^{5/4}} {\langle V, \varphi \rangle^2} \|\dot a\|_{L^1_t} \bigl( \|v\|_{L_x^{\infty} L_t^1} + \|F\|_{Y} \bigr),
\end{align}
using the estimates for Term III above. Finally
\begin{align}
&\biggl\|a(t)^{5/4}  \frac{4 \pi} {\langle V, \varphi \rangle^2}   \Bigl \langle \int_0^t \cos((t -s) \sqrt{-\Delta}) \dot a(s)(\varphi_{a(s)} - a(s)^{-5/4} \varphi) ds, V \varphi \Bigr \rangle \biggr\|_{L_t^\infty \cap L_t^1}\\
& \lesssim \frac{4 \pi (1 + \|\dot a\|_{L^1_t})^{5/4}} {\langle V, \varphi \rangle^2} \|V \varphi \|_{L^1_x}  \| |\nabla| ( \dot a(s)(\varphi_{a(s)} - a(s)^{-5/4} \varphi) \|_{L^{3/2,1}_xL^1_t } \\
& \lesssim \|\dot a\|_{L^1_t \cap L^\infty_t}^2.
\end{align}
Putting all these estimates together, and using the assumption $\|F\|_{Z} < \varepsilon$, we conclude that
\begin{align}
\|(w,b) - (0,1) \|_X &\lesssim \|\psi_0 - \phi \|_{|\nabla|^{-1} L^{3/2,1} \cap \dot H^1} + \|\psi_1\|_{L^{3/2,1} \cap L^2} 
\end{align}
and hence, provided  
\begin{align}
\|\psi_0 - \phi \|_{|\nabla|^{-1} L^{3/2,1} \cap \dot H^1} + \|\psi_1\|_{L^{3/2,1} \cap L^2} < \varepsilon < \varepsilon_0
\end{align}
for $\varepsilon_0 > 0$ from \eqref{equ:phi_bds} we conclude that
\begin{align}
\|(w,b) - (0,1) \|_X \lesssim \varepsilon.
\end{align}
We may estimate the difference between two solutions as in \cite{Beceanu} using that $h$ is Lipschitz to prove that the map $(v,a) \mapsto (w, b)$ is a contraction.
\end{proof}

We now prove the main theorem:

\begin{proof}[Proof of Theorem \protect{\ref{codim_1}}]
Fix $\varepsilon_ 0> 0$ as in Proposition \ref{prop:fixed}, $f \in \cM_s$ with $\varepsilon \leq \varepsilon_0$. Then for any $F \in Z$ with $\|F\|_{Z} < \varepsilon$,  there exists a unique $h \lesssim \varepsilon^2$ such that the system \eqref{equ:x_eqn}, \eqref{equ:pcv}, \eqref{equ:a} with $(\psi_0, \psi_1)$ replaced by $(\phi + \gamma_0+ h Y,\gamma_1  \kappa h Y)$ admits a unique solution $(v,a) \in X$ with
\[
\|(v,a) - (0,1) \|_{X} < \varepsilon.
\]
Note that we have used here \eqref{equ:phi_bds} and \eqref{equ:gamma_orthog}. 

In order to apply Corollary \ref{cor:as}, we observe that $f_{hi}$ satisfies Assumption (A2), then to obtain bounds in terms of $\|f\|_{X_s}$, we use Lemma \ref{lem:weighted_ests_free} and we estimate
\begin{align}
\| f_{hi}  \|_{\cH^s} \lesssim  \| f  \|_{\cH^s} \quad \textup{ and } \quad   \| \langle x \rangle^{1- \nu}   \,  |\nabla|^{\nu \beta+} f_{0,hi}\|_{L^2(\bR^3)} \lesssim \| \langle x \rangle^{1- \nu}   \,  |\nabla|^{\nu \beta+} f_0 \|_{L^2(\bR^3)},
\end{align}
and similarly using Lemma \ref{lem:weighted_ests_free} and \eqref{weighted_sobolev_bds}, we have
\[
\| \langle x \rangle^{1- \nu}   \,  |\nabla|^{\nu \beta-1+} f_{1,hi}\|_{L^2(\bR^3)} = \| \langle x \rangle^{1- \nu}   \langle \nabla \rangle^{\nu \beta -1 +} f_{1}\|_{L^2(\bR^3)} .
\]
 Hence, by  the conditions defining $\cM_s$ and by Corollary \ref{cor:as} with $\lambda = \varepsilon$, we have
\[
\mathbb{P} \left( \|W(t)  f_{\geq k_0}^\omega \|_Z  > \varepsilon \right) < C e^{-c \varepsilon^2 / \|f\|_{X_s}^2 },
\]
which is a non-trivial bound provided
\begin{align}
c \varepsilon^2 / \|f\|_{X_s}^2  > \log C \qquad\Longleftrightarrow \qquad \|f\|_{X_s} < \sqrt{ c \varepsilon^2/ \log C} .
\end{align}
Hence for $0 < \varepsilon < \varepsilon_0$, we set
\[
\Omega_{\varepsilon} = \{ \omega \in\Omega \,:\, \|W(t)  f_{\geq k_0}^\omega \|_Z  \leq \varepsilon \},
\]
and for $\omega \in \Omega_{\varepsilon}$, we apply Proposition \ref{prop:fixed} with $F = W(t) f_{\geq k_0}^\omega$. This concludes the proof.
\end{proof}

\begin{remark} \label{rmk:bdd_l8}
For the solution constructed by fixed point in Proposition \ref{prop:fixed}, we have
\[
\|v\|_{L^{8}_{x,t}(\mathbb{R}\times \mathbb{R}^3)} \leq \|v\|_{L^{6,2}_x L^\infty_t(\mathbb{R}\times \mathbb{R}^3)}^{3/4} \|v\|_{L^\infty_x L^2_t(\mathbb{R}\times \mathbb{R}^3)}^{1/4} \leq \varepsilon.
\]
Hence, in light of Proposition \ref{prop:l8}, under the hypotheses of Theorem \ref{codim_1}, we have the same bounds for the solution $u$ almost surely.
\end{remark}

\appendix
\section{Kernel Estimates}\label{a:kernel}

Throughout this section, we assume that $H = -\Delta + V$, where $V \in L^{3/2,1} \cap L^{\infty}$ is a radial, real-valued potential which satisfies the explicit decay estimates
\[
|V(r)| \lesssim \langle r \rangle^{-4}, \quad |V'(r)| \lesssim \langle r \rangle^{-5}.
\]
The goal of this section is to prove certain estimates for distorted Fourier multipliers by obtaining explicit kernel estimates. As we will see, these will depend on the decay rate of the potential and its derivative. We will begin by proving Proposition \ref{prop:kernel_ests} since many of the other estimates will follow from the methods used to prove this proposition. We recall that we defined in \eqref{equ:linearized_kernel} the kernel
\begin{align}
K_k(t,r,r') =  \int_{\rho} e^{i t \rho} \psi_k(\rho) \widetilde{e}(r, \rho)   \overline{\widetilde{e}(r',\rho)} d\rho
\end{align}
where
\begin{align}\label{e_app}
\widetilde{e}(r, \rho) = r\rho \int_{S^2} e(r \theta, \rho \omega) \sigma(d\theta).
\end{align}
We restate Proposition \ref{prop:kernel_ests} here:
\kernelests*

Before turning to the proof of Proposition \ref{prop:kernel_ests}, we need to establish certain estimates for the functions $\widetilde{e}(r,\rho)$ in \eqref{e_app}. For $\rho > 0$, we may write (see \cite[p441]{Schlag07}):
\begin{align}\label{equ:radial_e}
\widetilde{e}(r,\rho) = c_+(\rho) f(r,\rho)  - \frac{1}{2i} \overline{f(r,\rho)},
\end{align}
where
\[
c_+(\rho) = \frac{1}{2i} \frac{\overline{f(0,\rho)}}{f(0,\rho)},
\]
and where $f(r, \rho)$ are the Jost functions, given by the Volterra integral equation
\begin{align}\label{jost}
f(r,\rho) = e^{ir\rho} + \int_r^\infty \frac{\sin(\rho(r'-r))}{\rho} V(r') f(r', \rho) dr', \qquad r \geq0.
\end{align}
We set $f(r,\rho) = e^{ir\rho} m(r,\rho)$, then $m(r,\rho)$ satisfies the integral equation
\begin{align}\label{equ:m_eq}
m(r,\rho) = 1 + \int_r^\infty \frac{e^{2i\rho(r'-r)} - 1}{2i\rho} V(r')m(r', \rho) dr'.
\end{align}
From \cite[(16)]{Schlag07} we have that
\begin{align}\label{equ:lwr_bds}
C(V)^{-1} \rho(1+\rho)^{-1} < |f(0,\rho)| < C(V) .
\end{align}
We note that in particular, for $\rho \geq 1/2$, 
\[
|f(0,\rho)| \gtrsim 1. 
\]
We collect some bounds for the functions $m(r,\rho)$ in the following lemma.
\begin{lemma}\label{lem:m_bds}
Let $m(r,\rho)$ be given as in \eqref{equ:m_eq} with $V$ satisfying
\[
|V(r)| \lesssim \langle r \rangle^{-4}, \quad |V'(r)| \lesssim \langle r \rangle^{-5}.
\]
Then for all $r > 0$ and $\rho$ sufficiently large
\begin{align}
 |(m(r,\rho) -1)| &\lesssim  \rho^{-1}\langle r \rangle^{-3} , \label{equ:firstbd0}\\
 |\partial_r m(r,\rho)| &\lesssim \rho^{-1} \langle r \rangle^{-4}, \label{equ:firstbd}\\
 |\partial_\rho^\ell m(r,\rho) | &\lesssim \rho^{-1}\langle r \rangle^{-4 + \ell} , \quad \ell = 1,2, 3 \label{equ:second_bd},
\end{align}
where the implicit constants depend only on $V$ and fixed constants. 
\end{lemma}
\begin{proof}
In what follows, we will assume that we have replaced $V$ by a compactly supported potential. Ultimately, we will obtain estimates uniform in the truncation and can take limits to conclude, see for instance \cite[Lemma 2]{Schlag07}. By Gronwall's inequality,
\begin{align}\label{f_bds}
\sup_{r,\rho \geq 0} |f(r,\rho)| \leq \exp \left(\int_0^\infty r' |V(r')| dr' \right) =: C_1(V),
\end{align}
which implies the same about $m(r,\rho)$. Hence, from \eqref{equ:m_eq} we obtain
\begin{align*}
 |m(r,\rho) -1| & \leq  \int_r^\infty \frac{1}{ \rho} |V(r')| |m(r', \rho)| dr' \lesssim C_1(V) \int_r^\infty \rho^{-1} |V(r')| dr'.
\end{align*}
Since $|V(r')| \lesssim \langle r' \rangle^{-4}$, we conclude by integrating that
\[
 |m(r,\rho) - 1|  \lesssim \rho^{-1} \langle r \rangle^{-3},
\]
which proves \eqref{equ:firstbd0}.  

By taking a derivative in $r$ in \eqref{equ:m_eq}, we obtain
\begin{align}
\partial_r m(r,\rho) = - \int_r^\infty e^{2 i \rho(r-r')} V(r') m(r', \rho) dr' 
\end{align} 
which, together with \eqref{f_bds} and the decay assumption on $V$ implies
\begin{align}\label{equ:deriv_decay}
|\partial_r m(r,\rho)| \leq  \int_r^\infty |V(r') m(r',\rho)| \lesssim \langle r \rangle^{-3} .
\end{align}
Now we integrate \eqref{equ:m_eq} by parts in $r'$, and obtain
\begin{equation}\label{mb1}
\begin{split}
m(r,\rho) &= 1 - \frac{1}{(2i\rho)^2}\int_r^\infty e^{2i\rho(r'-r)} [V'(r') m(r',\rho) + V(r') \partial_{r'} m(r',\rho)] dr' \\
& \hspace{14mm}-  \frac{1}{(2i\rho)^2} V(r) m(r,\rho)  - \frac{1}{2i\rho} \int_r^\infty V(r') m(r',\rho)dr' 
\end{split}
\end{equation}
and again taking a derivative in $r$ we have
\begin{align}\label{mb2}
\partial_r m(r,\rho) &= \frac{1}{2i\rho}\int_r^\infty e^{2i\rho(r'-r)} [V'(r') m(r', \rho) + V(r') \partial_{r'} m(r', \rho)] dr' \\
& \hspace{14mm}   + \frac{1}{2i\rho} V(r) m(r, \rho).
\end{align}
Using \eqref{equ:deriv_decay} in \eqref{mb2}, we obtain the improved derivative decay estimate
\begin{align}\label{equ:deriv_decay2}
|\partial_r m(r,\rho)| \lesssim \rho^{-1} \langle r \rangle^{-4},
\end{align}
which proves \eqref{equ:firstbd}. 

Next, we compute $\partial_\rho m(r,\rho)$ from the formula \eqref{mb1}. We obtain
\begin{align}
\partial_\rho m(r,\rho) &= F_1(r, \rho)  - \frac{1}{(2i\rho)^2} \int_r^\infty e^{2i\rho(r'-r)} [V'(r') \partial_\rho m(r',\rho) + V(r') \partial_\rho \partial_{r'} m(r',\rho)] dr' \\
& \hspace{14mm}- \frac{1}{(2i\rho)^2} V(r) \partial_{\rho } m(r, \rho) - \frac{1}{2i\rho} \int_r^\infty V(r') \partial_\rho m(r', \rho) dr',
\end{align}
where
\[
|F_1(r,\rho)| \lesssim \rho^{-2} \langle r\rangle^{-3}.
\]
Since
\begin{align}
&\biggl| \frac{1}{(2i\rho)^2}\int_r^\infty e^{2i\rho(r'-r)}[V'(r') \partial_\rho m(r',\rho) + V(r') \partial_\rho \partial_{r'} m(r',\rho)] dr' \biggr| \\
& \leq \sup_{r' > r} |\partial_\rho m(r',\rho) | \int_r^\infty \frac{1}{(2\rho)^2}|V'(r')| dr'  +\sup_{r' > r} |\partial_\rho \partial_{r} m(r,\rho)|  \int_r^\infty \frac{1}{(2\rho)^2} |V(r')| dr' \\
& \lesssim \sup_{r' > r} |\partial_\rho m(r',\rho) | \rho^{-2} \langle r\rangle^{-4}  +\sup_{r' > r} |\partial_\rho \partial_{r'} m(r',\rho)| \rho^{-2} \langle r \rangle^{-3} \phantom{\int},
\end{align}
we obtain
\begin{align}
|\partial_\rho m(r,\rho) | &\lesssim \rho^{-2} \langle r\rangle^{-3} + \sup_{r' > r} |\partial_\rho m(r',\rho) | \rho^{-2} \langle r\rangle^{-4}  +\sup_{r' > r} |\partial_\rho \partial_{r'} m(r',\rho)| \rho^{-2} \langle r \rangle^{-3} \\
& \hspace{24mm} +  \sup_{r' > r} |\partial_\rho m(r',\rho) | \rho^{-1} \langle r\rangle^{-3}
\end{align}
Similarly, using \eqref{mb2}, we compute
\begin{align}
\partial_\rho \partial_r m(r,\rho) &= F_2(r,\rho) + \frac{1}{2i\rho} \int_r^\infty e^{2i\rho(r'-r)}[V'(r') \partial_\rho m(r',\rho) + V(r') \partial_\rho \partial_{r'} m(r',\rho)] dr'\\
& \hspace{14mm} + \frac{1}{2i\rho} V(r) \partial_\rho m(r,\rho),
\end{align}
where
\[
|F_2(r,\rho)| \lesssim \rho^{-1} \langle r\rangle^{-3}.
\]
Since
\begin{align}
&\left|  \int_r^\infty \frac{e^{2i\rho(r'-r)}}{2i\rho}[V'(r') \partial_\rho m(r',\rho) + V(r') \partial_\rho \partial_{r'} m(r',\rho)] dr' \right|\\
&  \lesssim \sup_{r' > r} |\partial_\rho m(r',\rho) | \rho^{-1} \langle r\rangle^{-4}  +\sup_{r' > r} |\partial_\rho \partial_{r'} m(r',\rho)| \rho^{-1} \langle r \rangle^{-3} ,
\end{align}
taking the supremum over $r > r_0$ we obtain
\begin{align}
&\sup_{r > r_0} |\partial_\rho m(r,\rho)| + \sup_{r > r_0} |\partial_\rho \partial_r m(r,\rho)|  \\
& \lesssim\rho^{-2} \langle r_0 \rangle^{-3} +  \sup_{r' > r_0} |\partial_\rho m(r',\rho) | \rho^{-1} \langle r_0\rangle^{-3}  +\sup_{r' > r_0} |\partial_\rho \partial_{r'} m(r',\rho)| \rho^{-2} \langle r_0 \rangle^{-3}\\
& \hspace{14mm}+ \rho^{-1} \langle r_0 \rangle^{-3}+  \sup_{r' > r_0} |\partial_\rho m(r',\rho) | \rho^{-1} \langle r_0\rangle^{-4}  +\sup_{r' > r_0} |\partial_\rho \partial_{r'} m(r',\rho)| \rho^{-1} \langle r_0 \rangle^{-3} .
\end{align}
Since $V$ has been replaced by a compactly supported potential, for $\rho$ sufficiently large
\[
\sup_{r' > r_0} |\partial_\rho \partial_{r'} m(r',\rho)| < \infty.
\] 
and we obtain for $\rho$ sufficiently large that
\[
|\partial_\rho m(r,\rho)| \lesssim \rho^{-1} \langle r\rangle^{- 3}.
\]
For the higher derivatives in $\rho$, we can repeat the same argument (for both $m(r,\rho)$ and $\partial_rm(r,\rho)$), and ultimately we obtain
\[
|\partial_\rho^{\ell} m(r,\rho)| \lesssim \rho^{-1} \langle r\rangle^{- 4 + \ell}, \qquad \ell = 1, 2, 3,
\]
which concludes the proof.
\end{proof}

\begin{proof}[Proof of Proposition \ref{prop:kernel_ests}]
Fix $k_0 \geq 1$ such that for $\rho \in \textup{supp } \psi_k$ for $k \geq k_0$, the estimates from Lemma \ref{lem:m_bds} hold. We insert the expression for $\widetilde{e}(r,\rho)$ from \eqref{equ:radial_e} into the definition of the kernel from \eqref{equ:linearized_kernel}, which yields four terms
\begin{align}
K_k(t,r,r') &= \frac{1}{4} \int_0^\infty e^{i(r-r')\rho} e^{it \rho} \psi_k(\rho) m(r,\rho) \overline{m(r', \rho)} d\rho\\
& +  \frac{1}{4}  \int_0^\infty e^{-i(r-r')\rho} e^{it \rho} \psi_k(\rho) m(r',\rho) \overline{m(r, \rho)} d\rho\\
& + \frac{1}{2i}  \int_0^\infty e^{-i(r+r')\rho} e^{it \rho} \psi_k(\rho) \overline{c_+(\rho)} \, \overline{m(r,\rho) m(r', \rho)} d\rho\\
& + \frac{1}{2i}  \int_0^\infty e^{i(r+r')\rho} e^{it \rho} c_+(\rho) \psi_k(\rho) m(r,\rho) m(r', \rho) d\rho\\
& =: K_k^{(+,+)}(t,r,r') +  K_k^{(-,-)}(t,r,r') + K_k^{(-,+)}(t,r,r') + K_k^{(+,-)}(t,r,r') .
\end{align}
We will estimate these terms using oscillatory integral estimates. We consider first
\begin{align}
K_k^{(+,+)}(t,r,r') &= \frac{1}{4} \int_0^\infty e^{i(r-r')\rho} e^{it \rho} \psi_k(\rho) m(r,\rho) \overline{m(r', \rho)} d\rho\\
& = \frac{1}{4} \int_0^\infty e^{i(r-r')\rho} e^{it \rho} \psi_k(\rho)  d\rho + \frac{1}{4} \int_0^\infty e^{i(r-r')\rho} e^{it \rho} \psi_k(\rho) (m(r,\rho) \overline{m(r', \rho)} - 1)d\rho.
\end{align}
Note that the first term is the expression for a standard Fourier multiplier. To estimate this first expression, we consider two cases. When $|t + (r-r')| \leq 1$ we estimate this term by bringing absolute values inside the integral, and using the support properties of $\psi_k$ to obtain
\[
\left |\chi_{|t + (r-r')| \leq 1} \frac{1}{4} \int_0^\infty e^{i(r-r')\rho} e^{it \rho} \psi_k(\rho)  d\rho \right| \lesssim 1.
\]
When $|t + (r-r')| \geq 1$, we integrate by parts three times to obtain
\begin{align}
 \frac{1}{4} \int_0^\infty e^{i(r-r')\rho} e^{it \rho} \psi_k(\rho)  d\rho & = - \frac{1}{4} \int_0^\infty \frac{1}{i (t+ (r-r'))^3} e^{i(r-r')\rho} e^{it \rho} \partial_\rho^3 ( \psi_k(\rho) )  d\rho ,
\end{align}
where we have used that $\psi_k(0) = 0 = \lim_{\rho \to \infty} \psi_k(\rho)$ to account for the boundary terms. Since
\begin{align}\label{equ:deriv_bds}
\int_0^\infty |\partial^3_\rho (\psi_k(\rho)) | d\rho = \int_0^\infty | \partial^3_\rho (\psi(\rho - k)) |d\rho = \int_0^\infty | (\partial^3_\rho\psi)(\rho - k) |d\rho \lesssim 1
\end{align}
uniformly in $k$, we obtain
\begin{align}
\left| \frac{1}{4} \int_0^\infty e^{i(r-r')\rho} e^{it \rho} \psi_k(\rho)  d\rho  \right| \lesssim \frac{1}{\langle t+ (r-r') \rangle^3}.
\end{align}

For the second term, we need to argue similarly, but here we will need to use bounds for the functions $m(r,\rho)$. When $|t + (r-r')| \leq 1$, we estimate
\begin{align}
&\left|\frac{1}{4} \int_0^\infty e^{i(r-r')\rho} e^{it \rho} \psi_k(\rho) (m(r,\rho) \overline{m(r', \rho)} - 1)d\rho \right|\\
& \lesssim \int_0^\infty \psi_k(\rho) |m(r,\rho) \overline{m(r', \rho)} - 1|d\rho\\
& \lesssim  |k|^{-1}\left( \langle r \rangle^{-3} +  \langle r' \rangle^{-3} \right) ,
\end{align}
where we have applied Lemma \ref{lem:m_bds}, using the identity
\[
m \overline{m} - 1 = (m-1)(\overline{m} - 1) + (\overline{m} - 1)  + (m-1).
\]
 When $|t + (r-r')| \geq 1$, we once again integrate by parts to obtain
\begin{align}
&\frac{1}{4} \int_0^\infty e^{i(r-r')\rho} e^{it \rho} \psi_k(\rho) (m(r,\rho) \overline{m(r', \rho)} - 1)d\rho \\
& =  \frac{1}{4i} \int_0^\infty \frac{1}{((r-r') + t)^3} e^{i(r-r')\rho} e^{it \rho} \partial^3_\rho\bigr( \psi_k(\rho) (m(r,\rho) \overline{m(r', \rho)} - 1) \bigr) d\rho,
\end{align}
using again that the boundary terms disappear due to the support of $\psi_k$. To establish bounds for 
\begin{align}
\partial_\rho^3 (m(r,\rho) \overline{m(r', \rho)} - 1),
\end{align}
we use Lemma \ref{lem:m_bds} to obtain
\begin{align}\label{rho_decay}
\sup_{r,r'} |\partial_\rho^\ell(m(r,\rho) \overline{m(r', \rho)} - 1) | \lesssim \rho^{-1} \bigl( \langle r \rangle^{-4 + \ell} + \langle r' \rangle^{-4 + \ell} \bigr), \quad \ell = 1,2,3
\end{align}
and
\[
\sup_{r,r'} | (m(r,\rho) \overline{m(r', \rho)} - 1) | \lesssim \rho^{-1} \bigl( \langle r \rangle^{-3} + \langle r' \rangle^{-3} \bigr).
\]
Thus, 
\begin{align}
\biggl| \frac{1}{4} \int_0^\infty e^{i(r-r')\rho} e^{it \rho} \psi_k(\rho) (m(r,\rho) \overline{m(r', \rho)} - 1)d\rho \biggr|& \lesssim  \frac{\langle r \rangle^{-1} + \langle r' \rangle^{-1}}{ |k| \langle t+ (r-r')  \rangle^3} ,
\end{align}
where we have once again used \eqref{equ:deriv_bds} and the support properties of $\psi_k$. An identical argument yields a similar bound for the $K^{(-,-)}(t,r,r')$ term:
\[
|K^{(-,-)}(t,r,r')| \lesssim \frac{\langle r \rangle^{-1} + \langle r' \rangle^{-1}}{|k| \langle t - (r-r')\rangle^3} + \frac{1}{\langle t - (r-r') \rangle^3} .
\]

Next we consider 
\begin{align}
 \frac{1}{2i}  \int_0^\infty e^{- i(r+r')\rho} e^{it \rho} \psi_k(\rho) \overline{c_+(\rho)} \, \overline{m(r,\rho) m(r', \rho)}  d\rho.
\end{align}
Again, we rewrite this kernel as
\begin{align}
\frac{1}{2i}  \int_0^\infty e^{- i(r+r')\rho} e^{it \rho}  \psi_k(\rho) \overline{c_+(\rho)} d\rho + \frac{1}{2i}  \int_0^\infty e^{- i(r+r')\rho} e^{it \rho}  \psi_k(\rho) \overline{c_+(\rho)}  (\overline{m(r,\rho) m(r', \rho)}  - 1)d\rho . 
\end{align}
Since
\[
c_+(\rho) = \frac{1}{2i} \frac{ \overline{m(0, \rho) }}{m(0,\rho) },
\]
by \eqref{equ:lwr_bds} and Lemma \ref{lem:m_bds}, we obtain
\begin{align}\label{equ:c_bds}
|c_+(\rho)| \leq \frac{1}{2}, \quad |\partial^\ell_\rho c_+(\rho) | \lesssim \rho^{-1}, \quad \ell = 1,2,3.
\end{align}
To estimate the first integral, we again consider two cases. In the region $|t - (r+r')| \leq 1$,  we may estimate the integral by a constant, while in the region $|t - (r+r')| > 1$, we integrate by parts, once again noting the boundary terms vanish due to the support of $\psi_k$, to obtain
\begin{align}
\biggl| \frac{1}{2i}  \int_0^\infty e^{- i(r+r')\rho} e^{it \rho}  \psi_k(\rho) \overline{c_+(\rho)} d\rho \biggr| \lesssim \frac{1}{\langle t-(r+r')\rangle^3}.
\end{align}
For the second term we integrate by parts and we use the previous observations, again together with Lemma \ref{lem:m_bds} to obtain
\begin{align}
\biggl|  \frac{1}{2i}  \int_0^\infty e^{- i(r+r')\rho} e^{it \rho}  \psi_k(\rho) \overline{c_+(\rho)}  (\overline{m(r,\rho) m(r', \rho)}  - 1)d\rho \biggr| \lesssim \frac{\langle r \rangle^{-1} + \langle r' \rangle^{-1}}{ |k| \langle t-(r+r')\rangle^3}.
\end{align}
This yields 
\[
|K^{(-, +)}(t,r,r')| \lesssim  \frac{\langle r \rangle^{-1} + \langle r' \rangle^{-1}}{|k| \langle t-(r+r')\rangle^3} +  \frac{1}{\langle t- (r+r')\rangle^3} .
\]
The estimates for the $K^{(+, -)}$ kernel follow in the same manner, yielding
\[
|K^{(+, -)}(t,r,r')| \lesssim \frac{\langle r \rangle^{-1} + \langle r' \rangle^{-1}}{ |k| \langle t+(r+r')\rangle^3} +  \frac{1}{\langle t+ (r+r')\rangle^3} . \qedhere
\] 
\end{proof}

Now we turn to the proofs of the multiplier estimates. First we prove Lemma \ref{lem:weighted_ests} and Lemma~\ref{lem:weighted_ests_free}. Arguing as in the proof of Proposition \ref{prop:kernel_ests}, provided $\varphi(\rho)$ satisfies the assumptions of Lemma~\ref{lem:weighted_ests},
\begin{equation}\label{mult_formula}
\varphi(H) f(r) =   r^{-1} \int_{r'} M(r,r') r' f(r') dr'  + \sum_{c_i = \pm} r^{-1} \int_{r'} \int_0^\infty e^{ i (c_1 r- c_2 r')\rho} \varphi_{ c_1, c_2}(\rho) r' f(r') dr' d\rho,
\end{equation}
where $\varphi_{c_1, c_2}$ is given by
\begin{align}
\varphi_{+,+} = \varphi(\rho), \quad \varphi_{ -,-} = \varphi(\rho), \quad \varphi_{ +,-} = \frac{1}{2i} \varphi(\rho) c_+(\rho), \quad \varphi_{k, -,+} = \frac{1}{2i}\varphi(\rho) \overline{c_+(\rho)}.
\end{align}
and where $M$ can be decomposed as
\begin{align}\label{mbds1}
M(r,r') &=  M_1(r,r') \\
& + \sum_{c_i = \pm}  \frac{1}{4} \left( \int_r ^\infty V(s) ds - \int_{r'} ^\infty V(s) ds \right)   \int_0^\infty e^{i(c_1 r-c_2r')\rho} \frac{\varphi_{c_1, c_2}(\rho)}{2i\rho}d \rho
\end{align}
for a kernel $M_1(r,r')$ which satisfies
\begin{align}\label{mbds2}
|M_1(r,r')| \lesssim \sum_{c_2 = \pm 1} \frac{\langle r \rangle^{-2} + \langle r' \rangle^{-2}}{\langle r +  c_2 r' \rangle^2}.
\end{align}
Furthermore, we have that
\begin{align}\label{mbds3}
|M(r,r') - M_1(r,r')| \lesssim  \sum_{c_2 = \pm 1} \frac{\langle r \rangle^{-3} + \langle r' \rangle^{-3}}{\langle r' + c_2 r \rangle^3},
\end{align}
We remark that the factor $\langle r \rangle^{-2} + \langle r' \rangle^{-2}$ in \eqref{mbds2} compared to the $\langle r \rangle^{-1} + \langle r' \rangle^{-1}$ factor in the bounds from Proposition \ref{prop:kernel_ests} comes from integrating by parts twice instead of three times, and we will use this in \eqref{equ:m_ker_bds} below to conclude $L^2$ boundedness for $\langle r \rangle^{\alpha} M$. 

We will sketch the argument for these bounds. As in the previous argument, $\varphi(H)$ has an associated kernel
\begin{align}
\widetilde{M}(r,r') &= \frac{1}{4} \int_0^\infty e^{i(r-r')\rho}  \varphi (\rho) m(r,\rho) \overline{m(r', \rho)} d\rho\\
& +  \frac{1}{4}  \int_0^\infty e^{-i(r-r')\rho}  \varphi(\rho) m(r',\rho) \overline{m(r, \rho)} d\rho\\
& + \frac{1}{2i}  \int_0^\infty e^{-i(r+r')\rho}\varphi(\rho) \overline{c_+(\rho)} \, \overline{m(r,\rho) m(r', \rho)} d\rho\\
& + \frac{1}{2i}  \int_0^\infty e^{i(r+r')\rho}  c_+(\rho) \varphi (\rho) m(r,\rho) m(r', \rho) d\rho.
\end{align}
Each of these four terms yields a contribution to $M(r,r')$ and another to the standard Fourier multiplier component in \eqref{mult_formula}. For instance, writing
\begin{align}
&\frac{1}{4} \int_0^\infty e^{i(r-r')\rho}  \varphi (\rho) m(r,\rho) \overline{m(r', \rho)} d\rho \\
&= \frac{1}{4} \int_0^\infty e^{i(r-r')\rho}  \varphi (\rho) (m(r,\rho) \overline{m(r', \rho)} -1)d\rho  + \frac{1}{4} \int_0^\infty e^{i(r-r')\rho}  \varphi (\rho) d\rho,
\end{align}
the second expression is one of the standard Fourier terms in \eqref{mult_formula}, and we demonstrate how to estimate
\begin{align}\label{equ:sample}
\biggl| \frac{1}{4} \int_0^\infty e^{i(r-r')\rho} \varphi(\rho) (m(r,\rho) \overline{m(r', \rho)} - 1)d\rho \biggr|.
\end{align}
By \eqref{mb1}, we can write
\begin{align}
m(r,\rho)& = 1 - \frac{1}{2i\rho} \int_r ^\infty V(r') dr' - \frac{1}{(2i\rho)^2}\int_r^\infty e^{2i\rho(r'-r)} [V'(r') m(r',\rho) + V(r') \partial_{r'} m(r',\rho)] dr' \\
& \hspace{14mm}-  \frac{1}{(2i\rho)^2} V(r) m(r,\rho)  - \frac{1}{2i\rho} \int_r^\infty V(r') (m(r',\rho) - 1)dr' .
\end{align}
Consequently, 
\begin{align}
m(r,\rho) = 1 - \frac{1}{2i\rho} \int_r ^\infty V(r') dr'  + \frac{1}{(2i\rho)^2} F(r,\rho),
\end{align}
and by Lemma \ref{lem:m_bds}, we have the bounds
\begin{align} \label{fbds}
|\partial_\rho^{\ell} F(r,\rho) | \lesssim \langle r \rangle^{-4 + \ell}, \quad \ell = 0,1, 2, 3.
\end{align}
Hence, using the identity
\[
m \overline{m} - 1 = (m-1) (\overline{m} - 1) +  (\overline{m} - 1)  + (m-1),
\]
once again, we obtain
\begin{equation}  \label{m_exp}
\begin{split}
&m(r,\rho) \overline{m(r',\rho)} - 1 \\
&=- \frac{1}{2i\rho}  \left( \int_r ^\infty V(s) ds - \int_{r'} ^\infty V(s) ds \right)\\
& \hspace{8mm} + \frac{1}{(2i\rho)^2} \left( \int_r^\infty V(s) ds \right)^2+  \frac{1}{(2i\rho)^2} F(r,\rho) + \frac{1}{(2i\rho)^2} F(r',\rho) \\
& \hspace{8mm}- \frac{1}{(2i\rho)^{3}} \left(F(r,\rho)   \int_{r'} ^\infty V(s)  + F(r',\rho)\int_r ^\infty V(s) ds  \right) +  \frac{1}{(2i\rho)^{4}} F(r,\rho) F(r',\rho).
\end{split}
\end{equation}
Substituting \eqref{m_exp} into \eqref{equ:sample}, we obtain a contribution arising from 
\[
 \frac{1}{2i\rho}  \left( \int_r ^\infty V(s) ds - \int_{r'} ^\infty V(s) ds \right)
\]
and the other terms have explicit decay in both $r$ and $r'$ and contribute to $M_1(r,r')$. To estimate this term, we argue as follows: note that
\begin{align}
&\left |\chi_{|r-r'| \leq 1} \frac{1}{4} \int_0^\infty e^{i(r-r')\rho} \varphi(\rho)\left[ \frac{1}{2i\rho}  \left( \int_r ^\infty V(s) ds - \int_{r'} ^\infty V(s) ds \right) \right] d\rho \right| \\
& = \chi_{|r-r'| \leq 1} \left | \frac{1}{4} \int_{r}^{r'} V(s) ds \right| \left|\int_0^\infty e^{i(r-r')\rho} \frac{\varphi(\rho)}{2i \rho}  \right| \\
& \lesssim \chi_{|r-r'| \leq 1} \frac{|r -r'|}{\langle r \rangle^4} \left|\int_0^\infty \frac{ e^{i(r-r')\rho}}{r-r'} \partial_\rho \left( \frac{\varphi(\rho)}{2i \rho} \right)  \right|.
\end{align}
Using the decay assumption on $\partial_\rho \varphi$, we estimate
\[
\left |\chi_{|r-r'| \leq 1} \frac{1}{4} \int_0^\infty e^{i(r-r')\rho} \varphi(\rho)\left[ \frac{1}{2i\rho}  \left( \int_r ^\infty V(s) ds - \int_{r'} ^\infty V(s) ds \right) \right] d\rho \right| \lesssim \langle r\rangle^{-4},
\]
and hence we may suppose that $|r-r'| \geq 1$. Then we have
\begin{align}
&\frac{1}{4} \int_0^\infty e^{i(r-r')\rho} \varphi(\rho) \left[ \frac{1}{2i\rho}  \left( \int_r ^\infty V(s) ds - \int_{r'} ^\infty V(s) ds \right) \right]  d\rho \\
& = \frac{1}{4} \int_0^\infty \frac{1}{- (r-r')^3} \partial_\rho^3 e^{i(r-r')\rho} \varphi(\rho) \left[ \frac{1}{2i\rho}  \left( \int_r ^\infty V(s) ds - \int_{r'} ^\infty V(s) ds \right) \right] d\rho ,
\end{align}
and integrating by parts, we obtain
\begin{align}\label{equ:int_by_parts}
 \frac{1}{4} \left( \int_r ^\infty V(s) ds - \int_{r'} ^\infty V(s) ds \right)  \int_0^\infty \frac{e^{i(r-r')\rho}}{(r-r')^3}  \partial_\rho^3\biggl[ \frac{\varphi(\rho)}{2i\rho} \biggr] d\rho.
\end{align}
Applying the Leibniz rule, and using the decay assumptions for $\varphi(\rho)$ and $V$, we obtain a bound of
\begin{align}
\frac{\langle r \rangle^{-3} + \langle r' \rangle^{-3}}{\langle r' - r \rangle^3},
\end{align}
which yields \eqref{mbds3}.

\medskip
We may now prove Lemma \ref{lem:weighted_ests} and Lemma \ref{lem:weighted_ests_free}.

\begin{proof}[Proof of Lemma \ref{lem:weighted_ests}]
The proof relies on \eqref{mult_formula}. First we assume that $f \in \mathcal{S}$. We write
\[
\langle r \rangle^\alpha \varphi(H) f = r^{-1} \int_{r'} \langle r \rangle^{\alpha} M(r,r') r' f(r') dr'  + \langle r \rangle^{\alpha}\sum_{c_i = \pm} r^{-1} \int_{r'} \int_0^\infty e^{ i (c_1 r- c_2 r')\rho} \varphi_{c_1, c_2}(\rho) r' f(r') dr' d\rho.
\]
For the first term, we use \eqref{mbds1}, \eqref{mbds2} and \eqref{mbds3} to see that for $0 < \alpha < 1$, the kernel $\langle r \rangle^{\alpha} M(r,r')$ is $L^2$ bounded. Indeed,
\begin{align}\label{equ:m_ker_bds}
\|\langle r \rangle^{\alpha} M(r,r')\|_{L_{r,r'}^2(0,\infty) } \lesssim  \left( \int_r \int_{r'} \langle r \rangle^{2\alpha} \left[  \frac{\langle r \rangle^{-4} + \langle r' \rangle^{-4}}{ \langle r - r' \rangle^4} \right]\right)^{1/2} < \infty,
\end{align}
and hence  for any $M > 0$,
\begin{align}\label{whatweget}
\biggl\| \chi_{|r| < M} \langle r \rangle^{\alpha} r^{-1} \int_{r'} \langle r \rangle^{\alpha} M(r,r') r' f(r') dr' \biggr\|_{L^2(\bR^3)} \lesssim \|f\|_{L^2} \leq \|\langle r \rangle^{\alpha} f \|_{L^2},
\end{align}
and we obtain the desired bound by Fatou's Lemma and density. The case of $-1 < \alpha < 0$ is symmetric, using that 
\[
\|\langle r' \rangle^{\alpha} M(r,r')\|_{L_{r,r'}^2(0,\infty) } < \infty.
\]
For the second term we use the following fact, which is a special case of \cite[Theorem 2.7]{KenigPCMI}: let $a(\rho)$ be a smooth symbol with bounded derivatives, and let
\[
T_a f = \int e^{ir \rho} a(\rho) \widehat{f}(\rho) d\rho.
\]
Then we claim that for any $\alpha \in \bR$,
\[
T_a : L^2(\langle r \rangle^{\alpha} dr) \to L^2(\langle r \rangle^{\alpha} dr),
\]
with the norm of this operator depending only on semi-norms of $a$ up to a certain order $N(\alpha) > 0$. By interpolation and duality, it suffices to prove this result for $ \alpha = -2m,$ with $m \in \bN$, and furthermore, it suffices to prove that
\[
\frac{1}{(1+|r|^2)^m} T_a \bigl((1+|r|^2)^m f \bigr): L^2 \to L^2.
\]
We compute
\[
( (1+|r|^2)^m f)^{\wedge} = (1 - \Delta_\rho)^m \widehat{f}(\rho),
\]
then, integrating by parts, we obtain
\begin{align}
\frac{1}{(1+|r|^2)^m} T_a \bigl((1+|r|^2)^m f \bigr)& = \frac{1}{(1+|r|^2)^m} \int e^{i r \rho } a(\rho)  (1 - \Delta_\rho)^m \widehat{f}(\rho) d\rho \\
&=  \int(1 - \Delta_\rho)^m \left[ \frac{e^{i r \rho} a(\rho) }{(1+|r|^2)^m} \right] \widehat{f}(\rho) d\rho.
\end{align}
We obtain the desired result by Leibniz rule, using that $\partial_\rho^k a(\rho)$ is the symbol of an $L^2$-bounded multiplier by our assumptions. This completes the proof.
\end{proof}

\begin{proof}[Proof of Lemma \ref{lem:weighted_ests_free}]
This follows from arguments in the proof of Lemma \ref{lem:weighted_ests}, noting that if we only treat the free case, the argument yields the result even when the symbol $\varphi$ is smooth but does not decay.
\end{proof}

Next, we turn to proving Lemma \ref{lem:proj_bds} and Lemma \ref{lem:decay_sf}. We begin by establishing a formula for the projection operators $P_k$. As in our set-up for Proposition \ref{prop:kernel_ests}, for radial functions $f$, $g$, we obtain that
\begin{align}
\langle \psi_k(\sqrt{|H|}) f, g \rangle =   \int_{\rho} \int_{r'} \int_r  \psi_k(\rho) \widetilde{e}(r, \rho)   \overline{\widetilde{e}(r',\rho)} r' f(r')  r g(r) dr' \, d\rho dr .
\end{align}
For $k \geq k_0$ we define
\begin{align}\label{equ:gk_ker}
K_k(r,r') =  \int_{\rho} \psi_k(\rho) \widetilde{e}(r, \rho)   \overline{\widetilde{e}(r',\rho)} d\rho,
\end{align}
and note
\[
K_k(r,r') = K_k(0, r,r')
\]
for $K_k(0, r,r')$ as in Proposition \ref{prop:kernel_ests}. Then
\begin{align}\label{radial_ker}
\langle  \psi_k(\sqrt{|H|}) f, g \rangle = \int_r \int_{r'} K_k(r,r') r'f(r') dr' r g(r) dr =: \langle \widetilde{  \psi_k}(\sqrt{|H|}) \widetilde{f}, \widetilde{g} \rangle_{L^2((0,\infty), dr)},
\end{align} 
where $\widetilde{f}(r) =\sqrt{4\pi} r f(r)$, $\widetilde{g}(r) =\sqrt{4\pi} r g(r)$.  As in \eqref{mult_formula}, from the proof of Proposition \ref{prop:kernel_ests}, we can write
\begin{equation}\label{pk_formula}
P_k f(r) =   r^{-1} \int_{r'} G_k(r,r') r' f(r') dr'  + \sum_{c_i = \pm} r^{-1} \int_{r'} \int_0^\infty e^{ i (c_1 r- c_2 r')\rho} \psi_{k, c_1, c_2}(\rho) r' f(r') dr' d\rho,
\end{equation}
for an integral kernel $G_k$ which satisfies (see \eqref{rho_decay}):
\begin{align} \label{g_bds}
|G_k(r,r')| \lesssim  \sum_{c_2 = \pm 1} \frac{\langle r \rangle^{-2} + \langle r' \rangle^{-2}}{|k|  \langle r +  c_2 r' \rangle^2} ,
\end{align}
and where
\begin{align}
\psi_{k, +,+} = \psi_k(\rho), \quad \psi_{k, -,-} = \psi_k(\rho), \quad \psi_{k, +,-} = \frac{1}{2i} \psi_k(\rho) c_+(\rho), \quad \psi_{k, -,+} = \frac{1}{2i}\psi_k(\rho) \overline{c_+(\rho)}.
\end{align}
We note that $\psi_{k, \pm,\pm}(\rho)$ are all well-localized, smooth bump functions provided $\rho \geq 1/2$, and hence can be treated (essentially) as standard Fourier projections, thus the second term in the expansion for $P_k f$ can be estimated as in the free case. 

As in \eqref{equ:m_ker_bds}, the factor $\langle r \rangle^{-2} + \langle r' \rangle^{-2}$ compared to the $\langle r \rangle^{-1} + \langle r' \rangle^{-1}$ factor in the bounds from Proposition \ref{prop:kernel_ests} comes from integration by parts twice instead of three times, and this additional decay will be used in the proof of Lemma \ref{lem:decay_sf}. Specifically, we exploit the fact that for $0 < \alpha < 1$, we can estimate
\begin{align}\label{equ:gk_bds}
\|\langle r \rangle^{\alpha} G_{k}(r,r')\|_{L_{r,r'}^2(0,\infty) } \lesssim \frac{1}{|k|} \left( \int_r \int_{r'} \langle r \rangle^{2\alpha} \left[  \frac{\langle r \rangle^{-4} + \langle r' \rangle^{-4}}{ \langle r - r' \rangle^4} \right]\right)^{1/2} < \infty.
\end{align}

\medskip
To prove Lemma \ref{lem:proj_bds}, we follow the argument from \cite{Schlag07}, and as in that work, we will use the theory of $A_p$ weights. We recall their definition and some of their properties, see \cite[Chapter V.1]{Stein_book} for further discussion.

\begin{definition}\label{def:ap}
For a weight function $w$, the class of $A_p$ weights are those weights which satisfy
\begin{align}
\sup_{B \subseteq \mathbb{R}^n } \frac{1}{|B|} \int_B w(x) dx \left(\frac{1}{|B|} \int_B w^{-\frac{p'}{p}}(x) dx\right)^{\frac{p}{p'}}< \infty.
\end{align}
\end{definition}
\begin{remark}\label{rem:ap_weights}
For $n = 1$ and $w(r) = r^\alpha$, one can check that $w \in A_p$ if and only if $-1 < \alpha < p-1$. Furthermore $w \in A_p$ if and only if the Hardy-Littlewood Maximal function, given for $f \in L^1_{loc}$ by
\[
Mf(x) = \sup_{r >0} \frac{1}{|B(x,r)|} \int_{B(x,r)} |f(y)| dy,
\]
is bounded on $L^p(w)$.
\end{remark}

\begin{proof}[Proof of \protect{Lemma~\ref{lem:proj_bds}}]
By Proposition \ref{prop:kernel_ests}, for the kernel $K_k(r,r')$ defined in \eqref{equ:gk_ker}, we have
\begin{align}\label{equ:gk_ker_decay}
|K_k(r,r')| \lesssim \frac{1}{\langle r -r' \rangle^3}
\end{align}
for a constant independent of $k \geq 1$. To conclude bounds for $L^p(\bR^3)$, we estimate (following \cite{Schlag07}) that
\begin{align}
\|P_k f\|_{L^p(\bR^3)} &= \sup_{\|g\|_{L^{p'}(\bR^3)} \leq 1} |\langle  \psi_k(\sqrt{|H|}) f, g \rangle |\\
&=  \sup_{\| r^{\frac{2}{p'} -1} \widetilde{g}\|_{L^{p'}(0,\infty)} \leq 1} |\langle r^{1 - \frac{2}{p'}}  \widetilde{  \psi_k}(\sqrt{|H|}) \widetilde{f}, r^{\frac{2}{p'} -1}  \widetilde{g} \rangle_{L^2((0,\infty), dr)}|\\
& \leq \|r^{1 - \frac{2}{p'}}  \widetilde{  \psi_k}(\sqrt{|H|}) \widetilde{f}\|_{L^p(0,\infty)}\\
& = \|r^{\frac{2}{p} - 1}  \widetilde{  \psi_k}(\sqrt{|H|}) \widetilde{f}\|_{L^p(0,\infty)},
\end{align}
where the operator $ \widetilde{  \psi_k}(\sqrt{|H|})$ is defined in \eqref{radial_ker}. Since $r^{2 - p}$ is an $A_p$ weight if and only if
\[
-1<2- p < p-1 \qquad \Longleftrightarrow \qquad 3/2 < p < 3,
\]
we conclude our proof noting that
\[
\int |K_k(r,r') \widetilde{f}(r')| dr'  \lesssim \int \frac{1}{\langle r - r' \rangle^3} |\widetilde{f}(r')| dr' \lesssim M\widetilde{f}(r),
\] 
and appealing to Remark \ref{rem:ap_weights}.
\end{proof}

\begin{proof}[Proof of Lemma \ref{lem:decay_sf}]
Let $f \in \mathcal{S}$. Note that
\begin{align}
|k|^{s} \| \langle x \rangle^{\alpha} P_k f\|_{L^2(\bR^3)}= \| \langle x \rangle^{\alpha} |k|^s P_k |H|^{-\frac{s}{2}}|H|^{\frac{s}{2}}  f\|_{L^2(\bR^3)},
\end{align}
and further that for $k \geq 1$, the operator
\begin{align}
|k|^s \psi_k(\sqrt{|H|}) |H|^{-\frac{s}{2}}
\end{align}
is a distorted Fourier multiplier with the same localization and smoothness properties as $\psi_k(\sqrt{|H|})$.  
We use \eqref{pk_formula}, and we let
\[
\mathbf{f} = |H|^{\frac{s}{2}} P_{\geq k_0} f
\]
and for first term involving $G_k$ we use \eqref{g_bds} and \eqref{equ:gk_bds} to obtain for $ \alpha < 1$ and any $M > 0$ that
\begin{align}
\biggl\| \chi_{|r| < M} \langle r\rangle^{\alpha} r^{-1} \int_{r'} G_k(r,r') r' \mathbf{f}(r') dr'  \biggr\|_{L^2_x(\bR^3)} &\lesssim \biggl\| \langle r\rangle^{\alpha}  \int_{r'} G_k(r,r') r' \mathbf{f} (r') dr'  \biggr\|_{L^2_r(0, \infty)}\\
& \lesssim |k|^{-1} \| \widetilde{\mathbf{f}} \|_{L^2_r(0, \infty)}\\
& = |k|^{-1} \| |H|^{\frac{s}{2}} P_{\geq k_0} f \|_{L^2_x(\bR^3)},
\end{align}
and we obtain by Fatou's Lemma
\begin{align}
\biggl\| \langle r\rangle^{\alpha} r^{-1} \int_{r'} G_k(r,r') r' \mathbf{f}(r') dr'  \biggr\|_{L^2_x(\bR^3)} \lesssim  |k|^{-1} \| |H|^{\frac{s}{2}} P_{\geq k_0} f \|_{L^2_x(\bR^3)}.
\end{align}
Using that this expression is square summable over $k \geq k_0$, together with Corollary \ref{cor:equivalence} and Remark~\ref{rem:hs_bds} to handle the $P_0$ term, and the fact that the eigenfunction $Y$ is exponentially decaying, we obtain the upper bound
\begin{align}
 \| |\nabla|^s f \|_{L^2_x(\bR^3)} \lesssim   \| \langle x \rangle^{\alpha} |\nabla|^s f \|_{L^2_x(\bR^3)},
\end{align}
from which the bound follows for $f \in \dot H^{s,\alpha}$ by density.

For the mixed signs, the $c_1r$ and $-c_2r'$ terms in the exponential carry the same sign, and hence we may integrate by parts twice, using the localization and smoothness of the bump function, to eliminate the $\langle r \rangle^\alpha$ growth. We will treat one of these terms. 

Let $\widetilde{P}_k$ denote a distorted Fourier projection associated to a slight ``enlargement'' of $\psi_k$. When $|r + r'| \leq 1$, 
\[
\chi_{|r + r'| \leq 1}\left| \int_0^\infty e^{ i (r+r')\rho} \psi_{k, +,-}(\rho) d \rho \right| \leq \chi_{|r + r'| \leq 1},
\]
while for $|r + r' | > 1$,  we integrate by parts to obtain
\[
\chi_{|r + r'| > 1} \left| \int_0^\infty e^{ i (r+r')\rho} \psi_{k, +,-}(\rho) d\rho \right| = \chi_{|r + r'| > 1} \left| \int_0^\infty \frac{1}{( r+r' )^3} e^{ i (r+r')\rho} ( \partial_\rho^3 \psi_{k, +,-}(\rho) )d\rho \right|
\]
and using that $\partial_\rho^3 \psi_{k, +, -} \in L_\rho^1$  and $\alpha < 1$, we obtain
\[
\chi_{|r + r'| > 1} \left| \int_0^\infty e^{ i (r+r')\rho} \psi_{k, +,-}(\rho) d\rho \right|  \lesssim \chi_{|r + r'| > 1} \frac{1}{( r+r')^3} .
\]
Thus
\begin{align}
&\biggl\|\langle r\rangle^{\alpha} \int_{r'} \int_0^\infty e^{ i (r+r')\rho} \psi_{k, +,-}(\rho) d\rho r' \widetilde{P}_k\bigl( \mathbf{f}(r') \bigr)dr'  \biggr\|_{L^2_r(0,\infty)}\\
& \lesssim  \biggl\|\langle r\rangle^{\alpha} \int_{r'}  \left(\chi_{|r + r'| > 1} \frac{1}{( r+r')^3} + \chi_{|r + r'| \leq 1} \right) |r' \widetilde{P}_k\bigl( \mathbf{f}(r') \bigr)|  dr' \biggr\|_{L^2_r(0,\infty)} \\
& \lesssim \|r \widetilde{P}_k\bigl( \mathbf{f}(r)  \bigr)\|_{L^2_{r}(0,\infty)},
\end{align}
and hence
\begin{align}
&\sum_{k \geq k_0} \biggl\|\langle r\rangle^{\alpha} \int_{r'} \int_0^\infty e^{ i (r+r')\rho} \psi_{k, +,-}(\rho) \widetilde{P}_k\bigl(r' \mathbf{f}(r') \bigr)dr' d\rho \biggr\|_{L^2_r(0,\infty)}^2\\
& \lesssim \sum_{k \geq k_0} \|r \widetilde{P}_k\bigl( \mathbf{f}(r)  \bigr)\|_{L^2_{r}(0,\infty)}^2 \\
& = \|\mathbf{f} \|^2_{L^2_x(\bR^3)},
\end{align}
which once again yields the desired upper bound of 
\begin{align}
 \| |\nabla|^s f \|_{L^2_x(\bR^3)} \lesssim   \| \langle x \rangle^{\alpha} |\nabla|^s f \|_{L^2_x(\bR^3)}
\end{align}
by Corollary \ref{cor:equivalence}.

For the other two terms in the second expression in \eqref{pk_formula}, we have $\psi_{k,+,+} = \psi_k(\rho) = \psi_{k,-,-}$, thus we have precisely the standard Fourier projections associated to $\psi_k(\rho)$, which we denote by $\mathbf{P}_k$. We proceed using an interpolation argument: we need to establish bounds for 
\begin{align}
\biggl\| \langle r \rangle^\alpha \biggl(\sum_{k \geq k_0} |  \mathbf{P}_k \widetilde{\textbf{f}} |^2 \biggr)^{1/2} \biggr\|_{L^2_r(0, \infty)},
\end{align}
where once again we use the notation $\widetilde{g} = r g$. To do so, we establish bounds for this expression when $\alpha = 0$ and $\alpha =1$ and interpolate. The $\alpha =0$ case is clear, while for $\alpha = 1$, we have
\begin{align}
\biggl\| \langle r \rangle \biggl(\sum_{k \geq k_0} |  \mathbf{P}_k \widetilde{\textbf{f}} |^2 \biggr)^{1/2} \biggr\|_{L^2_r(0, \infty)}^2 & = \int_0^\infty  \langle r \rangle^2 \biggl(\sum_{k \geq k_0} |  \mathbf{P}_k \widetilde{\textbf{f}} |^2 \biggr) dr \\
&=  \int_0^\infty  (1+ |r|^2) \biggl(\sum_{k \geq k_0} |  \mathbf{P}_k \widetilde{\textbf{f}} |^2 \biggr) dr .
\end{align}
Using Plancherel's theorem, we obtain
\begin{align}
 \int_0^\infty  (1+ |r|^2) \biggl(\sum_{k \geq k_0} |  \mathbf{P}_k \widetilde{\textbf{f}} |^2 \biggr) dr & \leq    \int_{-\infty}^\infty  \biggl(\sum_{k \geq k_0} |  \mathbf{P}_k  \widetilde{\textbf{f}} |^2 \biggr) d\rho +  \int_{-\infty}^\infty   \biggl(\sum_{k \geq k_0} |  \partial_\rho \bigl[ \psi_k(\rho) \mathcal{F}(\widetilde{\textbf{f}})(\rho)\bigr] |^2 \biggr) d\rho,
\end{align}
and the first term can be estimated by $\|\textbf{f}\|_{L^2_x(\bR^3)}$. For the second term, by Leibniz rule, we obtain two more terms. One term can be bounded by $\|\textbf{f}\|_{L^2_x(\bR^3)}^2$ by noting that $\{\partial_\rho  \psi_k(\rho) \}$ is another finitely overlapping decomposition of frequency space, while we bound the other term by writing
\[
\int_{-\infty}^\infty   \biggl(\sum_{k \geq k_0} |   \psi_k(\rho) \partial_\rho \mathcal{F}(\widetilde{\textbf{f}}) (\rho) |^2 \biggr) d\rho = \int_{-\infty}^\infty   \biggl(\sum_{k \geq k_0} |  \mathbf{P}_k \bigl[ r \widetilde{\textbf{f}}\bigr] |^2 \biggr) dr \lesssim \||x| \textbf{f}\|_{L^2_x(\bR^3)},
\]
and hence for $0 \leq \alpha \leq 1$,
\begin{align}
\biggl\| \langle r \rangle^\alpha \biggl(\sum_{k \geq k_0} |  \mathbf{P}_k \widetilde{\textbf{f}} |^2 \biggr)^{1/2} \biggr\|_{L^2_r(0, \infty)} \lesssim \|\langle x \rangle^{\alpha} \textbf{f}\|_{L^2_x(\bR^3)},
\end{align}
Since
\[
\|\langle x \rangle^{\alpha} \textbf{f}\|_{L^2_x(\bR^3)} = \| \langle x\rangle^{\alpha}  |H|^\frac{s}{2} P_{\geq k_0} f \|_{L^2_x} = \| \langle x\rangle^{\alpha}  |H|^s P_{\geq k_0} \varphi(-\Delta) f \|_{L^2_x}
\]
by the assumptions on $f$, we may apply Corollary \ref{cor:weighted_coercive} to complete the proof. 
\end{proof}

\bibliographystyle{myamsplain}
\bibliography{refs}
\end{document}